\documentclass{article}
\usepackage{arxiv}
\usepackage[T1]{fontenc}
\usepackage[utf8]{inputenc}
\usepackage{amsmath,amssymb,amsthm,mathtools}
\usepackage{microtype}
\usepackage{enumitem}
\usepackage{booktabs,tabularx,float,multirow}
\usepackage{algorithm,algpseudocode}
\usepackage{graphicx}
\usepackage{pifont}
\usepackage{xcolor}
\usepackage[colorlinks=true,linkcolor=blue!45!black,
 citecolor=blue!45!black,urlcolor=blue!50!black]{hyperref}
\newtheorem{theorem}{Theorem}[section]
\newtheorem{lemma}[theorem]{Lemma}
\newtheorem{proposition}[theorem]{Proposition}
\newtheorem{corollary}[theorem]{Corollary}
\newtheorem*{lowerboundrestated}{Theorem~\ref{thm:oracle-lower-bound} (restated)}
\theoremstyle{definition}
\newtheorem{assumption}[theorem]{Assumption}
\theoremstyle{remark}

\newcommand{\R}{\mathbb R}
\newcommand{\Span}{\operatorname{span}}
\newcommand{\dist}{\operatorname{dist}}
\newcommand{\argmin}{\operatorname*{arg\,min}}
\newcommand{\yesmark}{\textcolor{green!45!black}{\ding{51}}}
\newcommand{\nomark}{\textcolor{red!70!black}{\ding{55}}}
\renewcommand{\shorttitle}{CaCuSESOP with Automatic Local Acceleration}
\renewcommand{\headeright}{A Preprint}
\hypersetup{pdftitle={CaCuTe Meets SESOP: Better Theory and Practice with Lipschitz Hessians}}
\title{CaCuTe Meets SESOP: Better Theory and Practice with Lipschitz Hessians}
\author{Nazarii Tupitsa}
\date{}

\begin{document}
\maketitle

\begin{abstract}
We combine CaCuTe with sequential subspace optimization (SESOP) for smooth convex minimization under Lipschitz continuity of the Hessian. The resulting method, CaCuSESOP, achieves the global rate $O(k^{-2})$ with a fixed number of Hessian--vector products per iteration, while removing key limitations of classical SESOP. In particular, exact minimization of the original objective is replaced by minimization of a low-dimensional cubic upper model, which needs only be solved to a computable accuracy. The cubic regularization parameter is selected by backtracking.

When the Hessian at the minimizer is positive definite, CaCuSESOP automatically attains an accelerated local linear rate with complexity $O(\sqrt{\kappa_\star}\log(1/\varepsilon))$. The local strong-convexity parameter is not supplied to the algorithm. We give finite bounds for entering this regime and adapting the restart schedule. Finally, we prove that the global $O(k^{-2})$ rate is optimal for the considered fixed-budget gradient/Hessian--vector-product linear-span oracle class, even with momentum and unrestricted memory.
\end{abstract}


\section{Introduction}

We study convex minimization with a small, fixed number of
Hessian--vector products (HVPs) per iteration.
An HVP returns $\nabla^2f(x)v$ for a chosen vector $v$.

\subsection{Motivation}

CaCuTe \cite{tupitsa2025} separates two regimes of the cubic Taylor
model. When the quadratic Hessian term is small, a scaled gradient
step gives cubic-model decrease and an $O(k^{-2})$ rate while
this regime holds. One HVP checks the curvature condition. When
the quadratic term dominates, this certificate fails and the
gradient method falls back to standard gradient-descent decrease. Its general
convex guarantee is then $O(k^{-1})$.

The second regime is the starting point of this work. The model
is nearly quadratic, but gradient scaling alone does not use its
curvature efficiently. We seek a low-cost step for this regime
and a global $O(k^{-2})$ guarantee. The main question is how to
use a fixed number of HVPs to minimize the model without forming
or factorizing the Hessian. Automatic differentiation provides
these products directly \cite{pearlmutter1994}.

We construct the cubic model in a small SESOP space and minimize
that model instead of the original objective. Three HVPs suffice for
the convergence theory, so the inner solve requires no further
objective-oracle calls. The practical space also contains the previous
step and the Hessian--gradient direction. The algorithm first completes
a one-HVP backtracking search for the scaled-gradient step. If its
projected model residual is too large, the already computed HVP becomes
a direction of the full subspace model and the remaining columns are
formed. The resulting method combines both regimes with a fixed HVP
budget and admits automatic local acceleration by restarts.

\subsection{Related work}
\label{sec:related}

Table~\ref{tab:binary-comparison} gives a compact comparison of global
rates, directional curvature work, inner problems, parameter
dependence, and automatic local guarantees. The rate alone does not
describe the cost of a subspace method: minimizing the original
objective can require further oracle calls, whereas solving an
explicit model uses coefficients already computed.
Here an oracle call means an evaluation of $f$, $\nabla f$, or an HVP.

\begin{table}[htbp]
\centering\small
\caption{Compact comparison of rates, directional second-order work,
inner-problem complexity (IPC), parameter dependence, and automatic
local guarantees. A green
check means that the property is proved for the stated method; a red
cross means that no such result is established.}
\label{tab:binary-comparison}
\begin{minipage}{\textwidth}\footnotesize
\textbf{GR}: global convergence rate.
\par
\textbf{HVP-E}: number of Hessian--vector products, or equivalent
curvature probes, per outer iteration.
\par
\textbf{IPC}: inner-problem complexity, measured after the reduced
problem coefficients have been formed. It excludes objective-oracle
calls, construction of the subspace basis, and lifting the reduced
solution back to $\R^d$. A dense reduced problem of dimension $r$ costs
$O(r^3)$ under this convention. Hence, the three-, four-, and five-direction
Cubic SESOP variants have IPC $O(1)$, an $m$-dimensional Krylov or
Scieur subproblem has IPC $O(m^3)$, and a full-space problem has IPC
$O(d^3)$. The excluded basis operations cost $O(dr^2)$ to form or
orthogonalize a dense basis and $O(dr)$ to apply it.
\par
\textbf{PF}: parameter-free with respect to problem constants; neither
a smoothness constant nor a strong-convexity constant is supplied to
the algorithm. Ordinary algorithmic choices such as a memory budget
are not counted as problem parameters.
\par
\textbf{ALC}: automatic local convergence, without providing the local
strong-convexity parameter to the algorithm.
\par
\textbf{AALC}: automatic accelerated local convergence, without
providing the local strong-convexity parameter to the algorithm.
\end{minipage}
\par\smallskip
\begin{tabularx}{\textwidth}{@{}>{\raggedright\arraybackslash}Xcccccc@{}}
\toprule
Method
& GR
& HVP-E
& IPC
& PF
& ALC
& AALC\\
\midrule
SESOP \cite{narkiss2005}
& $O(k^{-2})$
& $\infty$\textsuperscript{$\dagger$}
& $\infty$\textsuperscript{$\dagger$}
& \yesmark
& \yesmark
& \nomark\\
\addlinespace
CaCuAdGD \cite{tupitsa2025}
& $O(k^{-1})$
& $1$
& ---
& \yesmark
& \yesmark
& \nomark\\
\addlinespace
Krylov CRN \cite{jiang2024}
& $O((mk)^{-1})$
& $m$
& $O(m^3)$
& \nomark
& \yesmark
& \nomark\\
\addlinespace
\begin{tabular}{@{}ll@{}}
\multirow{2}{*}{Generic Type-I \cite{scieur2024}}
& full/exact\\
& limited
\end{tabular}
& \begin{tabular}{@{}c@{}}$O(k^{-2})$\\$O(k^{-1})$\end{tabular}
& \begin{tabular}{@{}c@{}}$d^\ast$\\$m^\ast$\end{tabular}
& \begin{tabular}{@{}c@{}}$O(d^3)$\\$O(m^3)$\end{tabular}
& \begin{tabular}{@{}c@{}}\yesmark\\\yesmark\end{tabular}
& \begin{tabular}{@{}c@{}}\yesmark\\\yesmark\end{tabular}
& \begin{tabular}{@{}c@{}}\nomark\\\nomark\end{tabular}\\
\addlinespace
\begin{tabular}{@{}ll@{}}
\multirow{2}{*}{Adaptive Accelerated Type-I \cite{scieur2024}}
& full/exact\\
& limited
\end{tabular}
& \begin{tabular}{@{}c@{}}$O(k^{-3})$\\$O(k^{-2})$\end{tabular}
& \begin{tabular}{@{}c@{}}$d^\ast$\\$m^\ast$\end{tabular}
& \begin{tabular}{@{}c@{}}$O(d^3)$\\$O(m^3)$\end{tabular}
& \begin{tabular}{@{}c@{}}\yesmark\\\yesmark\end{tabular}
& \begin{tabular}{@{}c@{}}\nomark\\\nomark\end{tabular}
& \begin{tabular}{@{}c@{}}\nomark\\\nomark\end{tabular}\\
\addlinespace
Algorithm~\ref{alg:cubic}, exact baseline\textsuperscript{$\ddagger$}
& $O(\sqrt{\log k}/k^2)$
& $3$
& $O(1)$
& \nomark
& \yesmark
& \nomark\\
\addlinespace
Algorithm~\ref{alg:adaptive-cubic}, adaptive\textsuperscript{$\ddagger$}
& $O(\sqrt{\log k}/k^2)$
& $3$
& $O(1)$
& \yesmark
& \yesmark
& \nomark\\
\addlinespace
Algorithm~\ref{alg:deterministic-restarts}, deterministic restarts\textsuperscript{$\ddagger$}
& $O(k^{-2})$
& $3$
& $O(1)$
& \yesmark
& \yesmark
& \nomark\\
\addlinespace
Algorithm~\ref{alg:controlled-restarts}, controlled restarts\textsuperscript{$\ddagger$}
& $O(k^{-2})$
& $3$
& $O(1)$
& \yesmark
& \yesmark
& \yesmark\\
\addlinespace
CaCuSESOP (Algorithm~\ref{alg:cacusesop}), controlled restarts\textsuperscript{$\ddagger$}
& $O(k^{-2})$
& $1$--$5$
& $O(1)$
& \yesmark
& \yesmark
& \yesmark\\
\bottomrule
\end{tabularx}
\par\smallskip
\begin{minipage}{\textwidth}\footnotesize
\textsuperscript{$\dagger$} Classical SESOP prescribes no HVP oracle:
each outer iteration instead requires an exact minimizer of the original
objective over its current subspace. Under the stated smooth-convex
assumptions, this is a general low-dimensional smooth convex optimization
problem. A first-order iterative solver provides an
$\varepsilon$-accurate solution after finitely many calls, but no finite
uniform oracle budget returns the exact minimizer over this whole class.
Accordingly, $\infty$ denotes an unbounded worst-case budget for the
exact inner solve, including its gradient or equivalent curvature
probes. It does not assert literal HVP calls or nontermination on every
instance: quadratic objectives and other structured subproblems may be
solved exactly in finitely many operations. By contrast, a dash elsewhere
in the IPC column means that the method has no separate inner problem of
the type counted in that column.

\textsuperscript{$\ast$} Scieur does not call an HVP oracle. The $m$
retained gradient-difference directions play the role of curvature
directions and require forward gradients, $\Theta(dm)$ storage, and a
dense $m$-dimensional quasi-Newton/Anderson subproblem. Formation of
its reduced coefficients requires additional
ambient-vector arithmetic; IPC reports the subsequent $O(m^3)$ dense
solve. The precise fixed-memory bounds are $O(C_2/k+k^{-2})$ and
$O(C_4/k^2+k^{-3})$; no theorem makes $C_2$ or $C_4$ decay at a stated
rate with $m$. In the full/exact rows, ``exact'' refers to exact current
curvature, not to the accuracy of the reduced solver; both the
projection and secant errors vanish. Their faster outer rates require
$m=d$ and dense $O(d^3)$ algebra. The table therefore does not combine
$m$ and $k$ into a fictitious rate.

\textsuperscript{$\ddagger$} The same method may use the finite residual
criterion \eqref{eq:inexact-model} instead of an exact reduced argmin.
This requires no additional objective-oracle calls and preserves the
displayed convergence order, but changes the universal constants. The
constructive solver uses $J_{\mathrm{inner}}+1$ systems of order at most three
(or the dimension of the enlarged practical space),
so its IPC is $O(J_{\mathrm{inner}})$ rather than the idealized exact-solve entry. The finite count
$J_{\mathrm{inner}}$ is specified in
Proposition~\ref{prop:cubic-inner-work}.
For CaCuSESOP this finite criterion applies to the subspace fallback;
the gradient-line candidate is explicit and has its own residual test.

The PF checks for the two restart rows use their compatibility with
adaptive regularization; their fixed-coefficient implementations still
require a supplied Hessian-Lipschitz bound. The AALC check denotes an
asymptotic local rate of the form
$\exp(-ck/\sqrt{\kappa_\star})$ after a finite problem-dependent
transition; it is not a global linear-rate claim from the initial point.
The four displayed Cubic SESOP algorithms use three directions.
Corollary~\ref{cor:three-hvp-restarts} shows that the previous step may
be retained as an optional fourth practical direction without changing
their theoretical rates.
CaCuSESOP uses the five-direction space in the practical comparison
and the same controlled restart schedule. It completes backtracking for
the one-HVP gradient trial before checking the projected model residual.
Only a failed residual test activates the full subspace fallback.
\end{minipage}
\end{table}

\paragraph{SESOP and cubic subspace methods.}
The $O(k^{-2})$ SESOP certificate uses the current gradient, the
displacement from the starting point, and a weighted aggregate of
previous gradients \cite{narkiss2005,zibulevsky2014}. Practical SESOP
also retains a few previous-step directions. Our stated algorithms do
not need this memory direction; Corollary~\ref{cor:three-hvp-restarts}
allows it as an optional fourth direction. For a nonlinear function,
minimization in this space still requires access to the original
objective. The SESOP paper explains how structure can make these
calls inexpensive, but its theory counts outer iterations.
In our method the reduced cubic model is fully determined after
the HVPs. Its exact solve requires no more objective-oracle calls.
This statement concerns oracle cost; it does not claim a fixed
number of arithmetic operations for the reduced solve.

Second-order directions have appeared in SESOP before. SESOP--TN
computes an approximate Newton direction by several inner conjugate-gradient
iterations and then approximately minimizes the original objective over an
extended SESOP space \cite{zibulevsky2008sesoptn}. Zibulevsky also observed
that including the full Newton direction preserves Newton's local quadratic
behavior \cite{zibulevsky2014}; a related general comparison principle was
proved recently for methods that decrease the objective at least as much as
the full Newton step \cite{shea2024greedy}. These results do not give a
fixed-HVP construction of the SESOP certificate step: the truncated-Newton
budget is not fixed by the global theory, while forming the full Newton
direction requires solving a full-dimensional linear system.

Cubic regularization supplies the Taylor upper bound and descent
estimate \cite{nesterovpolyak2006}. Krylov CRN also minimizes a cubic
model in a subspace, but builds that space by repeated Hessian
products \cite{jiang2024}. We use the SESOP directions and control
the error caused by replacing exact objective minimization.

Random cubic models on sampled subspaces have also been studied in
Stochastic Subspace Cubic Newton \cite{hanzely2020sscn}; that construction
does not use the deterministic SESOP certificate directions considered here.

Kuruzov and Stonyakin analyze SESOP with inexact gradients and inexact
solutions of the original low-dimensional objective subproblem
\cite{kuruzovstonyakin2026}. With fixed gradient and subproblem errors,
their estimate contains explicit residual terms. In the current version
they also turn this perturbation analysis into a target-accuracy result:
under convexity of the restricted objective and a bounded search region,
they prescribe a subproblem tolerance depending on the desired outer
accuracy $\varepsilon$ and bound the finite number of ellipsoid iterations
needed to attain it. Thus this is a finite $\varepsilon$-accurate procedure,
not a finite computation of the exact subspace argmin; its tolerance and
inner complexity depend on the requested final accuracy. This dependence
makes their stated guarantee non-anytime. If a run was configured for
$\varepsilon_1$ and one later requests $\varepsilon_2<\varepsilon_1$, the
published theorem requires the tighter $\varepsilon_2$-dependent tolerance
at every outer iteration. Merely continuing the original run is not covered;
a direct application restarts it with the tighter tolerance. A continuation
schedule might avoid this repetition, but no target-independent schedule or
corresponding theorem is given there. Moreover, every inner ellipsoid
iteration queries the original objective or a projected original gradient at
a new point.

Our relative residual test concerns a different object. After at most three
HVPs have formed the reduced cubic model, the algorithm solves that explicit
model only to the bounded relative stationarity accuracy in
\eqref{eq:inexact-model}. The test contains no requested final accuracy and
can be checked and enforced using reduced quantities alone, without another
high-dimensional oracle call. The same run may therefore be stopped,
inspected, and continued to higher accuracy, while the global and accelerated
restart guarantees remain valid at every horizon. We therefore do not claim
to be the first analysis of inexact SESOP subproblems; our contribution is a
target-independent finite model solve with no additional original-oracle
calls and with the global-to-local guarantees of the proposed method.

We are not aware of an earlier result that replaces exact objective
minimization over the classical SESOP certificate space by one cubic
Taylor-model step constructed from a constant number of HVPs, controls the
resulting certificate defect, and simultaneously proves the inverse-square
global rate and automatic accelerated local convergence described below.

CaCuAdGD uses a scaled gradient step with one directional curvature
probe and function-value backtracking \cite{tupitsa2025}. Its general
convex rate is $O(k^{-1})$; the $O(k^{-2})$ regime requires the
directional-curvature condition stated in that work. Krylov CRN uses
$m$ new HVPs to form an $m$-dimensional cubic model and proves the
mixed bound $O((mk)^{-1}+k^{-2})$ \cite{jiang2024}. Thus a fixed $m$
generally leaves an $O(k^{-1})$ term. Its guarantee assumes a
Lipschitz Hessian and a supplied regularization bound, and its hidden
spectral constant may depend on the chosen Krylov dimension.

\paragraph{Scieur's framework.}
Scieur constructs reduced models from gradient differences
\cite{scieur2024}. Fixed memory is part of the construction:
Algorithm~1 discards old directions and adds the component of the
current gradient orthogonal to the retained space. Theorem~3
establishes gradient inclusion and an orthonormal direction matrix.
This is a prescribed update, not arbitrary truncated Anderson
acceleration. The experimental choice $m=25$ is not a theoretical
threshold; Appendix~F.4 also varies the memory size.

The cost of the approximation appears in the slower terms of the
rates. In Scieur's notation, let $L$ be the Hessian Lipschitz
constant, $R$ the level-set radius, $P^i$ the orthogonal projector
onto the current search space, $\delta$ the uniform relative
secant-error bound, and $\kappa$ the direction-matrix condition
bound. The points $z^i$ in the formula below are its stored secant
base points; they are not our corrected-gradient aggregates. Theorem~6 gives an $O(C_2R^2/k+LR^3/k^2)$ bound, apart
from a faster initialization term, with
\[
 C_2=\frac{\delta L}{2}(\kappa+2\kappa^2)
 +\max_{i\le k}\|\nabla^2f(x^i)-P^i\nabla^2f(x^i)P^i\|.
\]
For the recommended nonaccelerated update, $\kappa=1$.
With forward-difference length $h$ and at most $m$ retained
directions, monotonicity and the bounded level set give
$\delta\le\sqrt m(h+4R)$: each secant-error entry is
$h+2\|z^i-x^k\|$, where both centers lie in that set.
Thus fixed memory is compatible with the theorem, but it does
not give a general $1/m$ improvement of $C_2$.

Increasing the memory to the problem dimension, $m=d$, does not by
itself yield a pure $O(k^{-2})$ bound. If the retained directions span $\R^d$
at every iteration, then $P^i=I$ and only the projection term in
$C_2$ vanishes. With orthonormal directions, the remaining coefficient
is $3\delta L/2$: finite differences and curvature information from
older points still contribute to $\delta$.
Scieur distinguishes full memory, exact Hessian actions, and evaluation
of all curvature information at the current point; their combination
recovers full cubic Newton \cite[Section~3.3.2]{scieur2024}.
This construction uses $d$ current HVPs if the Hessian is obtained
column by column. For a dense model, storage is $O(d^2)$ and the
spectral solve costs $O(d^3)$ arithmetic operations, followed by a
scalar solve. The subproblem remains an explicit cubic model,
not exact minimization of the original nonlinear objective.

The accelerated Algorithm~7 has the different bound in Theorem~8.
Its coefficient $C_4$ involves secant error and the Hessian coupling
$\|(I-P)\nabla^2f(x)P\|$, where $x$ is a model base and $P$
projects onto its search space, rather than the full omitted-Hessian
term in $C_2$. Requirements~1b--3 must hold
at the actual extrapolated model bases; gradient inclusion only
at an outer iterate is insufficient. Neither theorem asserts the
pure full-Newton rate for a general fixed-dimensional space.
These are upper bounds, not lower bounds on actual convergence.
The accelerated estimate-sequence minimizer is explicit
\cite[Proposition~1]{scieur2024}; its $O(k^{-2})$ guarantee does not
require a generic full-dimensional Anderson minimization.
The paper tests both variants. Our experimental baseline uses the
recommended nonaccelerated Algorithms~1, 3, and~4.

Our HVPs give the exact reduced Hessian, so secant-error assumptions
are unnecessary. Linearly dependent directions are discarded, and our
bounds do not require angles between the remaining directions to
stay away from zero. Both methods solve explicit models; exact
minimization of the original objective is a feature of classical
SESOP.

\paragraph{Linear convergence without knowing $\mu$.}
This property is shared by several methods. For example, the
SESOP descent bound \cite[Proposition~1]{narkiss2005}, for an
$L$-Lipschitz gradient and strong convexity of modulus $\mu$, gives
\[
 f(x^{k+1})-f_\star
 \le(1-\mu/L)\bigl(f(x^k)-f_\star\bigr),
\]
although the algorithm does not use $\mu$. The CaCuAdGD entry in
the table is likewise a consequence of its descent bounds under
the additional conditions stated there, rather than a separate
theorem in that paper. Krylov CRN proves a strong-convexity result
explicitly \cite[Theorem~2]{jiang2024}.
For Scieur's nonaccelerated Type-I method the same property follows
from a model comparison, although it is not stated as a local theorem
in~\cite{scieur2024}. On a neighborhood of the minimizer, let
$L_{\mathrm{loc}}>0$ bound both the Hessian norm and the quadratic
form in the reduced model. The cubic coefficients of that model are
uniformly bounded there by the level-set curvature and the secant-error
and conditioning bounds. Gradient inclusion therefore permits the
model competitor $s=-\nabla f(x^k)/(2L_{\mathrm{loc}})$. Exact model
minimization and the accepted upper-model test give
\[
 f(x^{k+1})-f(x^k)
 \le-\frac{3}{8L_{\mathrm{loc}}}\|\nabla f(x^k)\|^2
     +O\!\left(\|\nabla f(x^k)\|^3\right)
 \le-\frac{1}{4L_{\mathrm{loc}}}\|\nabla f(x^k)\|^2
\]
for all sufficiently large $k$, because global convergence and local
smoothness imply $\|\nabla f(x^k)\|\to0$.
Near a positive-definite minimizer, the local PL inequality
$\|\nabla f(x^k)\|^2\ge2\mu(f(x^k)-f_\star)$ therefore gives
\[
 f(x^{k+1})-f_\star
 \le\left(1-\frac{\mu}{2L_{\mathrm{loc}}}\right)
       (f(x^k)-f_\star).
\]
The constants $L_{\mathrm{loc}}$ and $\mu$ occur only in this
comparison proof and are not supplied to the method.
Global convergence and compactness ensure entry into the local
neighborhood when the minimizer has positive definite Hessian.
Thus automatic local linear convergence is shared with this baseline,
including its fixed-memory forward-difference variant.
For Scieur's accelerated Algorithm~7, no automatic local linear
convergence guarantee is proved in the cited paper. Its model is based
at an extrapolated point, so the preceding descent comparison does not
directly give a contraction of successive outer iterates.

Small-dimensional relaxation methods provide another precedent
\cite{nesterov2021sdr}; Tupitsa proves adaptation to strong convexity
and the Polyak--\L{}ojasiewicz condition without supplying the modulus
\cite{tupitsa2020,tupitsa2022}. Li, Shi, and Yuan prove linear
convergence of forward--backward accelerated methods without that
information \cite{lishiyuan2024}. Our local theorem establishes
this property for the proposed cubic SESOP step under positive
definiteness of the Hessian at the minimizer.

\subsection{Contributions and organization}

\begin{itemize}[leftmargin=*,itemsep=5pt]
 \item \textbf{The method and its one-step certificate.}
 We replace the exact SESOP subspace minimization by an explicit cubic
 model. For the basic fixed-$H$ iteration, one new gradient and at most
 three HVPs determine the model; solving it needs no further calls to the
 original objective oracle (Proposition~\ref{prop:oracles}). The new
 gradient is no longer exactly orthogonal to the space, but its projected
 component is bounded by the decrease. This is the single error estimate
 used by all subsequent global results (Lemma~\ref{lem:defect}).
 \item \textbf{Baseline global and local convergence.}
 Under the assumptions below, Theorem~\ref{thm:main} gives
 \[
 f(x^K)-f_\star
 =O\!\left((L_0D_0^2+HD_0^3)
      \frac{\sqrt{1+\log(K+1)}}{(K+1)^2}\right).
 \]
 Here $L_0$ bounds the Hessian norm on the initial sublevel set,
 and $D_0$ bounds its radius about a minimizer; neither is an input.
 If the Hessian at the minimizer is positive definite, the same
 iteration converges locally linearly without a supplied
 strong-convexity parameter (Theorem~\ref{thm:local}).
 A local PL condition also gives linear convergence and finite
 trajectory length with singular Hessians and nonunique minimizers
 (Theorem~\ref{thm:local-pl}). The adaptive version has the same
 conclusions when the accepted coefficients stay bounded away from
 zero (Corollary~\ref{cor:adaptive-pl}); this bound is not supplied
 to the algorithm.
 \item \textbf{Constant-budget automatic local acceleration.}
 Deterministic epochs remove the logarithmic factor without adding
 oracle calls (Theorem~\ref{thm:restarted-rate}). A second controller
 chooses epoch lengths from observed function decreases. It retains
 the global $O(K^{-2})$ guarantee and, with the same budget of at most
 three HVPs per accepted step, automatically attains the accelerated local
 root rate $\exp(-c/\sqrt{\kappa_\star})$, where
 $\kappa_\star$ is the condition number of $\nabla^2f(x^\star)$.
 Neither that number, the local strong-convexity parameter, nor
 the optimal value $f_\star$ is supplied
 (Theorem~\ref{thm:automatic-acceleration}). Thus the accelerated local
 complexity is $O(\sqrt{\kappa_\star}\log(1/\varepsilon))$, although the
 local strong-convexity parameter is used only in the analysis. Combined
 with adaptive regularization, the implementation also needs no supplied
 Hessian-Lipschitz constant. This is the parameter-free, fixed-budget
 global-to-local guarantee claimed in this paper.
 This is not only an infinite-horizon statement.
 Corollary~\ref{cor:explicit-transition} gives an explicit positive
 local threshold $\varepsilon_{\mathrm{loc}}$ and a finite iteration bound for entry and restart adaptation
 $T_{\mathrm{adapt}}$. For every $0<\varepsilon<\varepsilon_{\mathrm{loc}}$, the same
 trajectory reaches $f(x^K)-f_\star\le\varepsilon$ once
 \[
  K\ge
  4T_{\mathrm{adapt}}
  +\frac{232\sqrt{\kappa_\star}}{\log2}
                   \log\frac{\varepsilon_{\mathrm{loc}}}{\varepsilon}.
 \]
 accepted steps, with all entry and adaptation terms finite and
 problem-dependent. The controller uses none of these analysis
 quantities, including $\varepsilon_{\mathrm{loc}}$ or $f_\star$.
 An adapted sharp CG/Krylov construction shows that
 the global exponent two cannot be improved uniformly in the stated
 fixed-budget gradient/HVP linear-span oracle class, even with momentum
 and unrestricted memory (Theorem~\ref{thm:oracle-lower-bound}).
 \item \textbf{Removing implementation idealizations.}
 Backtracking selects the cubic coefficient without a supplied Hessian
 Lipschitz bound and has an amortized linear trial count
 (Proposition~\ref{prop:counts}). A computable relative model residual
 replaces exact inner minimization while preserving the global and
 restart guarantees without additional objective-oracle calls. The criterion
 is independent of the requested final accuracy: one trajectory can be
 stopped, inspected, and continued without reconfiguring earlier model solves
 (Proposition~\ref{prop:inexact-model} and
 Corollary~\ref{cor:inexact-acceleration}). The four cubic-subspace
 algorithms use only three HVPs per accepted step. The preceding-step
direction may be added as a fourth practical direction without changing
the guarantees (Corollary~\ref{cor:three-hvp-restarts}).
 \item \textbf{CaCuSESOP: the combined practical method.}
 Section~\ref{sec:cacusesop} combines the CaCuTe gradient trial, the
 cubic SESOP fallback, previous-step memory, a Hessian--gradient
 direction, and the adaptive restart controller. The method first
 completes backtracking for the gradient trial using one HVP. It forms
 the remaining model columns only if the accepted trial fails the
 projected-residual certificate. The resulting method retains
 global $O(K^{-2})$ convergence and automatic accelerated local
 convergence with a finite cost of entry and restart adaptation
 (Corollary~\ref{cor:cacusesop-rates}). The practical version uses at
 most five HVPs; its three-direction core has the same convergence
 orders. The previous step and Hessian--gradient vector enlarge the
 practical model without changing the proof.
\end{itemize}

The paper follows the same dependency order as the proof. Section~\ref{sec:method}
defines the method and proves the two one-step bounds reused everywhere
else. Section~\ref{sec:baseline} first derives the global rate and then
the two local results. Section~\ref{sec:extensions} improves or
clarifies one aspect at a time: prescribed restarts, the sharpness of
the resulting global order, automatic restart control, the three-HVP
space, and finite inner solves. Each algorithmic
subsection states explicitly which earlier bounds it reuses.
Section~\ref{sec:cacusesop} assembles these ingredients into CaCuSESOP
and explains the role of the two additional practical directions.
Section~\ref{sec:experiments} gives the numerical protocol. The longer technical arguments are separated from this line:
the oracle lower bound is proved in Appendix~\ref{app:oracle-lower-bound},
backtracking is analyzed in Appendix~\ref{app:backtracking}, and the
finite inner solver in Appendix~\ref{app:finite-solves}.

\section{The method and its one-step guarantees}
\label{sec:method}

The goal of this section is to introduce only the objects needed by the
algorithm and to isolate the two estimates used by every convergence
result. We first state the assumptions, then motivate and define the
cubic subspace step, and finally prove its one-step properties.

\subsection{Notation and main assumptions}
\label{sec:assumptions}

We consider the unconstrained problem
\[
 \min_{x\in\R^d}f(x).
\]
The inner product is $u^Tv$, and $\|\cdot\|$ denotes the Euclidean
norm for vectors and the induced operator norm for matrices.
The gradient and Hessian are denoted throughout by $\nabla f(x)$ and
$\nabla^2f(x)$. Iteration indices are superscripts on vectors and
matrices, as in $x^k$, $s^k$, and $Q^k$, and subscripts on scalars,
as in $\omega_k$ and $H_k$. We use $k$ for the current iteration,
$K$ for a total number of accepted steps, and $i$ for a summation
index over steps. In a restarted run, $j$ numbers epochs or controller
phases. The epoch points $y^k$ are numbered locally from zero;
the concatenated trajectory $x^K$ counts all accepted steps. Coordinate indices
remain subscripts, so $q_1^k$ is the first column of $Q^k$.
We fix the starting point $x^0$ and make the following
three standard assumptions for convex optimization with cubic models.

\begin{assumption}[Convexity and differentiability]\label{ass:convex}
The function $f:\R^d\to\R$ is twice continuously differentiable and
convex. Thus, for every $x,y\in\R^d$,
\begin{equation}\label{eq:convexity}
 f(y)\ge f(x)+\nabla f(x)^T(y-x).
\end{equation}
\end{assumption}

The initial sublevel set contains all points whose value is no larger
than the starting value. Its boundedness is a standard assumption:
a decreasing sequence of objective values then keeps the iterates
in a bounded region.

\begin{assumption}[Bounded initial sublevel set]\label{ass:level}
The set
\begin{equation}\label{eq:level-set}
\mathcal C_0=\{x\in\R^d:f(x)\le f(x^0)\}
\end{equation}
is bounded.
\end{assumption}

By continuity, $\mathcal C_0$ is closed and hence compact.
The minimum of $f$ over this nonempty compact set is a global minimum,
since every point outside it has value greater than $f(x^0)$.
Fix a minimizer $x^\star\in\mathcal C_0$ and write
$f_\star=f(x^\star)$. Convexity also makes $\mathcal C_0$ convex.

\begin{assumption}[Hessian Lipschitz continuity]\label{ass:hessian}
There exists a constant $H>0$ such that
\begin{equation}\label{eq:hessian-lipschitz}
 \|\nabla^2f(x)-\nabla^2f(y)\|\le2H\|x-y\|
 \qquad\text{for all }x,y\in\mathcal C_0.
\end{equation}
\end{assumption}

This inequality bounds how much the Hessian can change between two
points. The basic algorithm is supplied with such a bound $H$ and uses it at
every iteration. Section~\ref{sec:adaptive} will remove the need to
supply this bound by checking the required estimates at each trial point.
Assumptions~\ref{ass:convex}--\ref{ass:hessian} remain in force
throughout the global analysis.

Two finite quantities will appear in the estimates:
\begin{equation}\label{eq:level-constants}
 L_0=\max_{x\in\mathcal C_0}\|\nabla^2f(x)\|,
 \qquad D_0=\max_{x\in\mathcal C_0}\|x-x^\star\|.
\end{equation}
Neither is an input to the method. In particular, we impose no global
bound on $\|\nabla^2f(x)\|$. The role of the bounded level set is to
provide these constants once monotonicity has been established.
We will prove that each model step remains in $\mathcal C_0$ before
using $L_0$ to estimate its decrease.

The standard Taylor estimates associated with
Assumption~\ref{ass:hessian} are the following
\cite[Lemma~1]{nesterovpolyak2006}:
\begin{align}
 f(x+s)&\le f(x)+\nabla f(x)^Ts
       +\tfrac12s^T\nabla^2f(x)s+\tfrac H3\|s\|^3,
       \label{eq:taylor-upper}\\
 \|\nabla f(x+s)-\nabla f(x)-\nabla^2f(x)s\|
       &\le H\|s\|^2.\label{eq:taylor-gradient}
\end{align}
These estimates are used only when $[x,x+s]\subseteq\mathcal C_0$.
Indeed, the two remainders are, respectively,
\[
 \int_0^1(1-t)s^T[\nabla^2f(x+ts)-\nabla^2f(x)]s\,dt,
 \qquad
 \int_0^1[\nabla^2f(x+ts)-\nabla^2f(x)]s\,dt.
\]
Bounding the Hessian difference by $2Ht\|s\|$ gives $H\|s\|^3/3$
and $H\|s\|^2$. For the descent analysis, only an upper bound on
the function remainder is needed in \eqref{eq:taylor-upper}.

\subsection{Descent analysis: the change to SESOP}
\label{sec:descent}

Let $x\in\mathcal C_0$, and let $S$ be a linear subspace containing
$\nabla f(x)$. Write $\Pi_Sv$ for the Euclidean orthogonal projection
of a vector $v$ onto $S$: its component along the search space.
The construction has three steps.

\textbf{1. Start from the SESOP inner problem.}
Classical SESOP minimizes the original objective over the search space:
\[
 s\in\argmin_{u\in S} f(x+u)
 \quad\Longrightarrow\quad \Pi_S\nabla f(x+s)=0.
\]
The gradient at the new point is orthogonal to every search direction.
This identity is used in the $O(k^{-2})$ convergence proof. Computing
the exact minimizer is the costly part for a nonlinear function.

\textbf{2. Replace the objective by its cubic upper model.}
We choose
\begin{equation}\label{eq:model}
 s=\argmin_{u\in S}
 \left\{\nabla f(x)^Tu+\tfrac12u^T\nabla^2f(x)u
                         +\tfrac H3\|u\|^3\right\}.
\end{equation}
If $S$ has at most three dimensions, three HVPs specify the whole
model. Its first-order condition is
\[
 \Pi_S\bigl(\nabla f(x)+\nabla^2f(x)s+H\|s\|s\bigr)=0.
\]
Taking the inner product with $s$ gives
$\nabla f(x)^Ts=-s^T\nabla^2f(x)s-H\|s\|^3$.

\textbf{3. Bound the error in orthogonality by the decrease.}
For this calculation, suppose that $[x,x+s]\subseteq\mathcal C_0$;
Lemma~\ref{lem:invariance} below proves that the model step always
has this property. Substitute the last identity into the function
Taylor bound, and subtract the projected first-order condition
from the gradient Taylor bound. This gives
\begin{align}
 f(x)-f(x+s)&\ge
 \underbrace{\tfrac12s^T\nabla^2f(x)s}_{\text{quadratic decrease}}
 +\underbrace{\tfrac{2H}{3}\|s\|^3}_{\text{cubic decrease}},
 \label{eq:preview-descent}\\
 \|\Pi_S\nabla f(x+s)\|&\le H\|s\|^2+H\|s\|^2
 =2H\|s\|^2.\label{eq:preview-projection}
\end{align}
The cubic term in the first line bounds $\|s\|^3$ by
$3(f(x)-f(x+s))/(2H)$. Raising that bound to the power $2/3$
and using the second line yields
\begin{equation}\label{eq:preview-defect}
 \boxed{\quad
 \|\Pi_S\nabla f(x+s)\|
 \le(18H)^{1/3}\bigl(f(x)-f(x+s)\bigr)^{2/3}.
 \quad}
\end{equation}
This is the key change: exact SESOP has zero projected gradient;
the cubic step leaves an error whose size is controlled by its own
decrease. We remove the relevant component before adding the new
gradient to the aggregate. Section~\ref{sec:algorithm} specifies
the correction, and Section~\ref{sec:one-step} proves the one-step
bounds. The global proof then sums the correction errors using
\eqref{eq:preview-defect}.

\subsection{Fixed-coefficient cubic SESOP}
\label{sec:algorithm}

Algorithm~\ref{alg:cubic} uses the current gradient, the displacement
from $x^0$, and a weighted aggregate $z^k$
of corrected gradients. The two changes to classical SESOP are
marked in the listing: the cubic model solve and the gradient
correction. An orthonormal basis removes redundant directions.

\begin{algorithm}[H]
\caption{Fixed-coefficient cubic SESOP}\label{alg:cubic}
\begin{algorithmic}[1]
\Require Initial point $x^0$; $H>0$ satisfying Assumption~\ref{ass:hessian}
\State Evaluate $\nabla f(x^0)$
\State $z^0\gets\nabla f(x^0)$, $\omega_0\gets1$
\For{$k=0,1,\ldots$}
 \If{$\nabla f(x^k)=0$}
  \State \Return $x^k$
 \EndIf
 \State $\mathcal S^k\gets\Span\{\nabla f(x^k),x^k-x^0,z^k\}$
 \label{line:space}
 \State Form an orthonormal basis $Q^k\in\R^{d\times r_k}$ of
 $\mathcal S^k$, with $r_k\le3$
 \State $B^k\gets (Q^k)^T[\nabla^2f(x^k)Q^k]$
 \Comment{$r_k$ HVPs}
 \Statex \textbf{Change 1: minimize the cubic model instead of $f$.}
 \State $\displaystyle a^k\gets\argmin_{a\in\R^{r_k}}
 \left\{\nabla f(x^k)^TQ^k a+\tfrac12a^TB^k a
                            +\tfrac H3\|a\|^3\right\}$
 \label{line:model}
 \Statex \hspace{\algorithmicindent}Solve this small problem by
 the scalar equation~\eqref{eq:secular}.
 \State $s^k\gets Q^k a^k$, $x^{k+1}\gets x^k+s^k$
 \State Evaluate $\nabla f(x^{k+1})$
 \Statex \textbf{Change 2: correct the gradient before aggregation.}
 \State $\mathcal V^{k+1}\gets\Span\{x^{k+1}-x^0,z^k\}$
 \State $e^{k+1}\gets\Pi_{\mathcal V^{k+1}}\nabla f(x^{k+1})$
 \label{line:correction}
 \State $h^{k+1}\gets\nabla f(x^{k+1})-e^{k+1}$
 \State $\omega_{k+1}\gets(1+\sqrt{1+4\omega_k^2})/2$
 \State $z^{k+1}\gets z^k+\omega_{k+1}h^{k+1}$
\EndFor
\end{algorithmic}
\end{algorithm}

The previous step $s^{k-1}$ may optionally be added as a fourth
direction. This additional direction does not change any theorem
below, but it increases the worst-case HVP count from three to four.

\subsubsection{Backtracking for the basic and subsequent algorithms}
\label{sec:adaptive}

The supplied coefficient $H$ can be replaced by backtracking.
The search space and aggregate update are unchanged: start with
half the previous accepted coefficient and double it until both
Taylor certificates below hold. The first certificate ensures
objective decrease; the second
controls the gradient error used by the convergence proof.
Neither the Hessian Lipschitz bound nor $L_0,D_0$ is supplied.

Choose a sufficiently small initial coefficient $H_{\mathrm{init}}>0$ and initialize
$H_{-1}=H_{\mathrm{init}}$.

At a nonstationary iterate $x^k$, start with
$\widehat H=H_{k-1}/2$ and compute the model step
$s$ with that trial coefficient. Accept it when both
the function certificate
\begin{equation}\label{eq:ftest}
 f(x^k+s)\le f(x^k)+\nabla f(x^k)^Ts
        +\tfrac12s^T\nabla^2f(x^k)s+\tfrac{\widehat H}3\|s\|^3,
\end{equation}
and the gradient certificate
\begin{equation}\label{eq:gtest}
 \|\nabla f(x^k+s)-\nabla f(x^k)-\nabla^2f(x^k)s\|
                  \le \widehat H\|s\|^2.
\end{equation}
hold. Otherwise double $\widehat H$ and recompute the trial step
using the cached reduced Hessian. On acceptance set $H_k=\widehat H$.
Algorithm~\ref{alg:adaptive-cubic} gives the complete procedure.

Evaluate the trial gradient only after the function test passes,
and retain it on acceptance. The gradient certificate is needed
for the aggregate-error bound; a function test alone is insufficient.
Both certificates use the cached HVPs, so additional trials require
no new curvature products.

Detailed proofs and trial counts are in
Appendix~\ref{app:backtracking}, including the precise initialization condition.

\begin{algorithm}[H]
\caption{Adaptive cubic SESOP}\label{alg:adaptive-cubic}
\begin{algorithmic}[1]
\Require Initial point $x^0$; initial coefficient $H_{\mathrm{init}}>0$
\State Evaluate $f(x^0)$ and $\nabla f(x^0)$
\State $z^0\gets\nabla f(x^0)$, $\omega_0\gets1$,
$H_{-1}\gets H_{\mathrm{init}}$
\For{$k=0,1,\ldots$}
 \If{$\nabla f(x^k)=0$}
  \State \Return $x^k$
 \EndIf
 \State $\mathcal S^k\gets\Span\{\nabla f(x^k),x^k-x^0,z^k\}$
 \State Form an orthonormal basis $Q^k$ of $\mathcal S^k$ and cache
 $B^k=(Q^k)^T[\nabla^2f(x^k)Q^k]$ \Comment{at most three HVPs}
 \State $\widehat H\gets H_{k-1}/2$
 \Repeat
  \State Solve the reduced cubic model with coefficient $\widehat H$,
  obtaining a trial step $s=Q^k a$
  \State Evaluate $f(x^k+s)$
  \If{the function certificate~\eqref{eq:ftest} holds}
   \State Evaluate $\nabla f(x^k+s)$
   \State $\mathrm{accept}\gets$ the truth value of
   certificate~\eqref{eq:gtest}
  \Else
   \State $\mathrm{accept}\gets\mathrm{false}$
  \EndIf
  \If{not $\mathrm{accept}$}
   \State $\widehat H\gets2\widehat H$
  \EndIf
 \Until{$\mathrm{accept}$}
 \State $H_k\gets\widehat H$, $s^k\gets s$,
 $x^{k+1}\gets x^k+s^k$ \Comment{retain the accepted value and gradient}
 \State $\mathcal V^{k+1}\gets\Span\{x^{k+1}-x^0,z^k\}$
 \State $h^{k+1}\gets\nabla f(x^{k+1})-
       \Pi_{\mathcal V^{k+1}}\nabla f(x^{k+1})$
 \State $\omega_{k+1}\gets(1+\sqrt{1+4\omega_k^2})/2$
 \State $z^{k+1}\gets z^k+\omega_{k+1}h^{k+1}$
\EndFor
\end{algorithmic}
\end{algorithm}

From this point onward, the same backtracking rule is available
for every subsequent algorithm, including the restarted variants,
finite model solves, and CaCuSESOP. Retain the accepted coefficient
across epoch boundaries, use both Taylor certificates, and keep
the additional model-residual test specified for each variant.
The global convergence orders and the positive-definite local results
are preserved.
The full local PL result, including finite trajectory length, also
holds when the accepted coefficients are eventually bounded away
from zero (Corollary~\ref{cor:adaptive-pl}). This is an analysis
condition, not a lower threshold in the backtracking rule.
The precise constants and proofs are given in
Appendix~\ref{app:backtracking} and the corresponding results below.

Since $Q^k$ is an orthonormal basis, every $u\in\mathcal S^k$ can
be written as $u=Q^k a$, with $\|Q^k a\|=\|a\|$. Thus the coordinate
solve in the listing is exactly the cubic subproblem~\eqref{eq:model}.

The removed component is $e^{k+1}$; the remaining component is
$h^{k+1}$. By construction, $h^{k+1}$ is orthogonal to both
$x^{k+1}-x^0$ and $z^k$. These are the two orthogonality relations
used in the SESOP proof. The correction is controlled by
\eqref{eq:preview-defect}, since
\[
 x^{k+1}-x^0=(x^k-x^0)+s^k\in\mathcal S^k,
 \qquad \mathcal V^{k+1}\subseteq\mathcal S^k.
\]

For uniform indexing in the analysis, set $h^0=\nabla f(x^0)$,
$e^0=0$, $z^{-1}=0$, and $\omega_{-1}=0$.
Then $z^k=\sum_{i=0}^k\omega_i h^i$.
If the algorithm stops, its output is a global minimizer by
Assumption~\ref{ass:convex}. For convergence statements, extend
$x^k$ and $z^k$ constantly, put all subsequent $s^k$, $h^k$, and
$e^k$ equal to zero, and continue the weight recurrence.

\subsubsection{Computing the step through a scalar equation}
\label{sec:reduced-model}

This subsection explains how to compute the model step by solving
one scalar equation. The next proposition gives that equation,
proves that its positive solution is unique, and shows how to
recover the step. It also shows why the solve needs no further
objective-oracle calls once the model coefficients have been computed.

\begin{proposition}[Reduced solve and oracle count]\label{prop:oracles}
Under Assumptions~\ref{ass:convex}--\ref{ass:hessian}, consider a
nonstationary iterate $x^k$ of Algorithm~\ref{alg:cubic}.
Let the columns of
$Q^k\in\R^{d\times r_k}$ form a Euclidean orthonormal basis of
$\mathcal S^k$, where $r_k=\dim\mathcal S^k\le3$. Define
\[
 B^k=(Q^k)^T\nabla^2f(x^k)Q^k.
\]
The step is $s^k=Q^k a^k$, where
\begin{equation}\label{eq:reduced-solution}
 a^k=-(B^k+\lambda_k I)^{-1}(Q^k)^T\nabla f(x^k),
\end{equation}
and $\lambda_k$ is the unique positive solution of
\begin{equation}\label{eq:secular}
 \lambda_k=H\bigl\|(B^k+\lambda_k I)^{-1}
                                  (Q^k)^T\nabla f(x^k)\bigr\|.
\end{equation}
It satisfies $\lambda_k\le\sqrt{H\|\nabla f(x^k)\|}$.
After initialization, each iteration needs one new gradient and
at most three HVPs, and no function evaluations.
\end{proposition}
\begin{proof}
Since $Q^k$ is orthonormal, substituting $s=Q^k a$ into the model gives
\[
 \nabla f(x^k)^TQ^k a+\tfrac12 a^TB^k a+\tfrac H3\|a\|^3.
\]
Convexity gives $B^k\succeq0$. Since $H>0$, this reduced objective
is strictly convex and coercive, so it has a unique minimizer.
Its first-order condition is
$(B^k+H\|a^k\|I)a^k=-(Q^k)^T\nabla f(x^k)$.
The projected gradient is nonzero because the full gradient is
nonzero and belongs to $\mathcal S^k$. Thus $a^k\ne0$ and
$\lambda_k=H\|a^k\|>0$.
Also $B^k\succeq0$, so $B^k+\lambda_k I$ is invertible.
This proves \eqref{eq:reduced-solution}--\eqref{eq:secular}.

To verify uniqueness directly, diagonalize $B^k$.
The norm on the right of \eqref{eq:secular} is a strictly decreasing
continuous function of positive $\lambda_k$, with positive
(possibly infinite) limit at zero and limit zero at infinity.
The left side increases strictly from zero to infinity, so there
is exactly one intersection. Since
$\|(B^k+\lambda_k I)^{-1}\|\le1/\lambda_k$, the equation also gives
$\lambda_k^2\le H\|(Q^k)^T\nabla f(x^k)\|=H\|\nabla f(x^k)\|$.

The directions defining $\mathcal S^k$ are available before the model
solve. One HVP for each column of $Q^k$ determines
$\nabla^2f(x^k)Q^k$ and hence $B^k$. Forming
$(Q^k)^T\nabla f(x^k)$ uses the current gradient already in memory.
The scalar solve and all reduced linear algebra then use only
these coefficients, with no further objective-oracle calls.
The single new gradient is $\nabla f(x^{k+1})$, used in the
correction and in the next iteration. With the supplied $H$, the
algorithm requires no function-value test.
\end{proof}

The bound of three HVPs comes from the three directions in
line~\ref{line:space} of Algorithm~\ref{alg:cubic}; adding the previous
step raises this bound to four. At $k=0$, the only nonzero direction is
$\nabla f(x^0)$, so one HVP suffices. At later iterations, any
linearly dependent directions are discarded before constructing
$Q^k$. A product computed at an earlier point is not being reused
as an exact current-Hessian product.

We first analyze exact model minimization and exact projections.
Section~\ref{sec:inexact-model} later replaces exact minimization by
a finite residual test while retaining exact reduced coefficients
and projections.
Equation~\eqref{eq:secular} characterizes the minimizer but does not
assert an elementary closed-form solution or a fixed arithmetic
cost. Its numerical solution, basis construction, and vector
operations are additional computation. The fixed bound above
concerns calls to the original objective's gradient and HVP oracles.

\subsection{The one-step guarantees used throughout}
\label{sec:one-step}

This subsection establishes the properties of a single iteration.
Section~\ref{sec:global} will combine these properties across
iterations to prove the global convergence rate.

We first prove that a model step stays in the initial sublevel set.
This justifies the Taylor estimates used in Section~\ref{sec:descent}.
We then use the decrease of the objective to bound the gradient
correction and the current gradient.

\begin{lemma}[Model step and level-set invariance]\label{lem:invariance}
Under Assumptions~\ref{ass:convex}--\ref{ass:hessian}, let
$x\in\mathcal C_0$, let $S$ be a linear subspace containing
$\nabla f(x)$, and let $s$ be defined by the cubic subproblem
\eqref{eq:model}.  The minimizer exists and is unique, and it satisfies
\begin{equation}\label{eq:model-stationarity}
 \Pi_S\bigl(\nabla f(x)+\nabla^2f(x)s+H\|s\|s\bigr)=0.
\end{equation}
Moreover, $[x,x+s]\subseteq\mathcal C_0$, and
\begin{equation}\label{eq:two-decreases}
 f(x)-f(x+s)
 \ge\tfrac12s^T\nabla^2f(x)s+\tfrac{2H}{3}\|s\|^3.
\end{equation}
If $\nabla f(x)\ne0$, then $s\ne0$ and the decrease is strict.
\end{lemma}

\begin{proof}
Convexity gives $\nabla^2f(x)\succeq0$.  Hence the objective in
\eqref{eq:model} is strictly convex and coercive on $S$, owing to the term
$H\|u\|^3/3$.  It has a unique minimizer, and its first-order condition is
\eqref{eq:model-stationarity}.  Taking the inner product of that condition
with $s\in S$ gives
\begin{equation}\label{eq:stationarity-along-step}
 \nabla f(x)^Ts=-s^T\nabla^2f(x)s-H\|s\|^3.
\end{equation}

If $\nabla f(x)=0$, uniqueness gives $s=0$, and all the conclusions are
immediate.  Suppose therefore that $\nabla f(x)\ne0$.  Because
$-\nabla f(x)\in S$ is a strict descent direction for the cubic objective
at zero, its minimizer is nonzero.  Equation
\eqref{eq:stationarity-along-step} then gives
$\nabla f(x)^Ts<0$, so differentiability implies
$f(x+ts)<f(x)\le f(x^0)$ for all sufficiently small $t>0$.

For every $0<t\le1$, substituting
\eqref{eq:stationarity-along-step} into the full cubic expression at $ts$
gives
\begin{align*}
 &t\nabla f(x)^Ts
 +\frac{t^2}{2}s^T\nabla^2f(x)s
 +\frac{Ht^3}{3}\|s\|^3\\
 &\qquad=-t(1-t/2)s^T\nabla^2f(x)s
          -Ht(1-t^2/3)\|s\|^3<0.
\end{align*}
Suppose that the segment from $x$ to $x+s$ leaves $\mathcal C_0$, and set
\[
 t_{\rm exit}=\inf\{t\in[0,1]:f(x+ts)>f(x^0)\}.
\]
The initial strict decrease gives $t_{\rm exit}>0$.  Continuity gives
$f(x+t_{\rm exit}s)=f(x^0)$, and the segment preceding this point is
contained in $\mathcal C_0$.  We may therefore apply
\eqref{eq:taylor-upper} on that segment.  The preceding strict inequality
then yields
\begin{align*}
 f(x+t_{\rm exit}s)
 &\le f(x)+t_{\rm exit}\nabla f(x)^Ts
   +\frac{t_{\rm exit}^2}{2}s^T\nabla^2f(x)s
   +\frac{Ht_{\rm exit}^3}{3}\|s\|^3\\
 &<f(x)\le f(x^0),
\end{align*}
contradicting the definition of $t_{\rm exit}$.  Thus
$[x,x+s]\subseteq\mathcal C_0$.

The Taylor upper bound is now valid for the full step.  Using
\eqref{eq:stationarity-along-step},
\begin{align*}
 f(x+s)-f(x)
 &\le\nabla f(x)^Ts+\tfrac12s^T\nabla^2f(x)s
                      +\tfrac H3\|s\|^3\\
 &=-\tfrac12s^T\nabla^2f(x)s-\tfrac{2H}{3}\|s\|^3.
\end{align*}
Rearranging proves \eqref{eq:two-decreases}.  Its right-hand side is
strictly positive when $s\ne0$.
\end{proof}

The next lemma bounds the component of the new gradient along the
search space. This controls the size of the correction made before
updating the aggregate, using the objective decrease just proved.

\begin{lemma}[Decrease controls the projected gradient]
\label{lem:defect}
Under Assumptions~\ref{ass:convex}--\ref{ass:hessian}, let $x$, $S$, and
$s$ be as in Lemma~\ref{lem:invariance}.  Then
\begin{equation}\label{eq:descent-defect}
 \|\Pi_S\nabla f(x+s)\|
 \le2H\|s\|^2
 \le(18H)^{1/3}\bigl(f(x)-f(x+s)\bigr)^{2/3}.
\end{equation}
\end{lemma}

\begin{proof}
Lemma~\ref{lem:invariance} places the entire step segment in
$\mathcal C_0$, so the gradient Taylor estimate
\eqref{eq:taylor-gradient} applies.  Subtracting
\eqref{eq:model-stationarity} from the projected new gradient gives
\begin{align*}
 \|\Pi_S\nabla f(x+s)\|
 &\le\|\nabla f(x+s)-\nabla f(x)-\nabla^2f(x)s\|
       +H\|s\|^2\\
 &\le2H\|s\|^2.
\end{align*}
The cubic term in \eqref{eq:two-decreases} gives
\[
 \|s\|^3\le\frac{3}{2H}\bigl(f(x)-f(x+s)\bigr).
\]
Taking the power $2/3$ and multiplying by $2H$ proves the second
inequality because
$2H(3/(2H))^{2/3}=(18H)^{1/3}$.
\end{proof}

We also need a bound on the current gradient. The next lemma shows
that the objective decrease controls its squared norm; this will
let us bound the gradient aggregate $z^k$ in the global proof.

\begin{lemma}[Gradient decrease]\label{lem:gradient-decrease}
Under Assumptions~\ref{ass:convex}--\ref{ass:hessian}, let $x$, $S$, and
$s$ be as in Lemma~\ref{lem:invariance}.  Suppose
$\nabla f(x^0)\ne0$. Then $L_0+\sqrt{2HL_0D_0}>0$, and
\begin{equation}\label{eq:gradient-budget}
 f(x)-f(x+s)\ge\frac{\|\nabla f(x)\|^2}{3\bigl(L_0+\sqrt{2HL_0D_0}\bigr)}.
\end{equation}
\end{lemma}

\begin{proof}
For this proof, abbreviate the gradient-decrease coefficient by
\begin{equation}\label{eq:decrease-constant}
 \mathcal L:=L_0+\sqrt{2HL_0D_0}.
\end{equation}
The segment from $x^\star$ to $x\in\mathcal C_0$ lies in the convex set
$\mathcal C_0$.  Since $\nabla f(x^\star)=0$, integration of the Hessian
along this segment gives
\begin{equation}\label{eq:gradient-level-bound}
 \|\nabla f(x)\|\le L_0\|x-x^\star\|\le L_0D_0.
\end{equation}
In particular, applying this bound at $x^0$ shows that
$\nabla f(x^0)\ne0$ implies $L_0D_0>0$, and hence $\mathcal L>0$.

By Lemma~\ref{lem:invariance}, the Taylor upper bound applies to the
accepted step.  By the minimizing property in \eqref{eq:model}, its full
cubic expression is no larger than the value at
$-\nabla f(x)/\mathcal L\in S$.  Consequently,
\begin{align*}
 f(x+s)-f(x)
 &\le\nabla f(x)^Ts+\tfrac12s^T\nabla^2f(x)s
                      +\tfrac H3\|s\|^3\\
 &\le-\frac{\|\nabla f(x)\|^2}{\mathcal L}
 +\frac{\nabla f(x)^T\nabla^2f(x)\nabla f(x)}{2\mathcal L^2}
 +\frac{H\|\nabla f(x)\|^3}{3\mathcal L^3}\\
 &\le-\frac{\|\nabla f(x)\|^2}{\mathcal L}
 \left(1-\frac{L_0}{2\mathcal L}
          -\frac{H\|\nabla f(x)\|}{3\mathcal L^2}\right).
\end{align*}
The last inequality uses
$\|\nabla^2f(x)\|\le L_0$.  Also $L_0/(2\mathcal L)\le1/2$, and
$\mathcal L^2\ge2HL_0D_0\ge2H\|\nabla f(x)\|$ by
\eqref{eq:gradient-level-bound}, so
$H\|\nabla f(x)\|/(3\mathcal L^2)\le1/6$.  The last parenthesis is
therefore at least $1/3$, which gives
\[
 f(x+s)-f(x)\le-\frac{\|\nabla f(x)\|^2}{3\mathcal L}.
\]
Rearranging proves \eqref{eq:gradient-budget}.  Only the accepted step
uses a Taylor estimate; the comparison point
$x-\nabla f(x)/\mathcal L$ is used solely to compare explicit cubic
expressions and need not belong to $\mathcal C_0$.
\end{proof}

To sum the one-step bounds, we need the identities satisfied by
the weights in the aggregate. The next lemma gives these identities
and the growth of the weights with the iteration number.

\begin{lemma}[Weight identities]\label{lem:weights}
The weights in Algorithm~\ref{alg:cubic} satisfy, for $k\ge0$,
\begin{equation}\label{eq:weights}
 \omega_k^2-\omega_{k-1}^2=\omega_k,\qquad
 \sum_{i=0}^k\omega_i=\omega_k^2,\qquad
 \frac{k+2}{2}\le\omega_k\le k+1.
\end{equation}
For every $k\ge0$, $1/2\le\omega_{k+1}-\omega_k<1$.
\end{lemma}
\begin{proof}
The recurrence is the positive root of
$\omega_{k+1}^2-\omega_{k+1}-\omega_k^2=0$.
The difference identity follows, including $k=0$ under
$\omega_{-1}=0$. Summation proves the cumulative identity.
For $t>0$, the inequalities
$2t\le\sqrt{1+4t^2}<2t+1$ imply that the next weight lies between
$t+1/2$ and $t+1$. Starting from $\omega_0=1$ gives the growth bounds.
\end{proof}

We now apply the model-step bounds to the search space and gradient
correction in Algorithm~\ref{alg:cubic}. The following proposition
collects the orthogonality relations and the two inequalities that
Section~\ref{sec:global} will use.

\begin{proposition}[Properties of one iteration]\label{prop:iteration}
Under Assumptions~\ref{ass:convex}--\ref{ass:hessian},
Algorithm~\ref{alg:cubic} is well defined, remains in $\mathcal C_0$,
and satisfies $f(x^{k+1})\le f(x^k)$. Its corrected gradient obeys
\begin{equation}\label{eq:orthogonalities}
 h^k\perp(x^k-x^0),\qquad h^k\perp z^{k-1},\qquad
 \|h^k\|\le\|\nabla f(x^k)\|.
\end{equation}
If $\nabla f(x^0)\ne0$, then
\begin{align}
 \|\nabla f(x^k)\|^2
 &\le3\bigl(L_0+\sqrt{2HL_0D_0}\bigr)\bigl(f(x^k)-f(x^{k+1})\bigr),
 \label{eq:algorithm-gradient}\\
 \|e^{k+1}\|
 &\le(18H)^{1/3}\bigl(f(x^k)-f(x^{k+1})\bigr)^{2/3}.
 \label{eq:defect-budget}
\end{align}
\end{proposition}
\begin{proof}
Apply Lemma~\ref{lem:invariance} inductively, since
$\nabla f(x^k)\in\mathcal S^k$ and $x^0\in\mathcal C_0$.
This proves existence and uniqueness of every model step, invariance,
and monotonicity. Orthogonal projections exist for all subspaces,
including those with redundant or zero directions.
For $k\ge1$, \eqref{eq:orthogonalities} follows from
$h^k=(I-\Pi_{\mathcal V^k})\nabla f(x^k)$; for $k=0$ it follows from
$x^k-x^0=0$ and $z^{-1}=0$.
The gradient estimate is Lemma~\ref{lem:gradient-decrease}.
Finally, $\mathcal V^{k+1}\subseteq\mathcal S^k$, and therefore
\[
 \|e^{k+1}\|=\|\Pi_{\mathcal V^{k+1}}\nabla f(x^{k+1})\|
 \le\|\Pi_{\mathcal S^k}\nabla f(x^{k+1})\|.
\]
Lemma~\ref{lem:defect} now gives \eqref{eq:defect-budget}.
\end{proof}

\section{Baseline convergence}
\label{sec:baseline}

We now use the one-step bounds in two different ways. The global
argument controls the corrected SESOP aggregate without any local
curvature assumption. The local arguments then show what improves near
a solution, first for a nondegenerate minimizer and then under a PL
condition that allows a singular Hessian.

\subsection{Global convergence}
\label{sec:global}

We now combine the one-step guarantees from Section~\ref{sec:one-step}
to bound the objective gap after any number of iterations.
The gradient bound will control the aggregate, and the correction
bound will control the error caused by replacing the SESOP inner solve.

The proof has three steps. First, we sum the one-step inequalities
over an interval of iterations. Next, we apply this sum from the
initial point to bound the accumulated terms. Finally, we use those
bounds on shorter intervals where the weights approximately double;
this gives the recurrence needed for the final rate.

Throughout this section, Assumptions~\ref{ass:convex}--\ref{ass:hessian}
hold, and we fix the minimizer $x^\star$ used in
\eqref{eq:level-constants}.  We first consider the nonstationary case
$\nabla f(x^0)\ne0$; the stationary case is immediate and will be included
in the final theorem.  Proposition~\ref{prop:iteration} shows that all
iterates are well defined, remain in $\mathcal C_0$, and decrease the
objective.  It is therefore meaningful to introduce the objective error
and the decrease produced by one iteration,
\[
 \Delta_k:=f(x^k)-f_\star,
 \qquad
 d_k:=f(x^k)-f(x^{k+1})=\Delta_k-\Delta_{k+1}\ge0.
\]
For the proofs below, write
$\mathcal L=L_0+\sqrt{2HL_0D_0}$; this is a coefficient in the
one-step estimate, not a smoothness parameter supplied to the algorithm.
The two estimates from Proposition~\ref{prop:iteration} that drive the
global proof are
\begin{equation}\label{eq:global-two-budgets}
 \|\nabla f(x^k)\|^2\le3\mathcal L d_k,
 \qquad
 \|e^{k+1}\|\le(18H)^{1/3}d_k^{2/3},
\end{equation}
The coefficient $\mathcal L$ is positive in the nonstationary
case considered here. The vector $e^{k+1}$ is the component
removed from the new gradient before updating the aggregate.
The first inequality is \eqref{eq:algorithm-gradient}, and the
second is \eqref{eq:defect-budget}.

For the proofs below, collect the two scales into
\begin{equation}\label{eq:proof-scale}
 C_{\mathrm{global}}
 :=3\mathcal L\|x^0-x^\star\|^2+\frac{512}{3}HD_0^3.
\end{equation}
This quantity is used only in the analysis; it is not an input to the
algorithm.  We use the conventions $\omega_{-1}=0$ and $z^{-1}=0$ below.
They are consistent with the initialization in
Algorithm~\ref{alg:cubic}.

The first lemma relates the decrease over an interval to the change
in the gradient aggregate and to the projection errors. Keeping
part of the decrease on the left will let us control these error
terms in the next two lemmas.

\begin{lemma}[Weighted interval estimate]\label{lem:reserve}
For all integers $0\le \ell\le k$,
\begin{equation}\label{eq:window-reserve}
\begin{aligned}
 \omega_k^2\Delta_{k+1}
 +\frac34\sum_{i=\ell}^k\omega_i^2d_i
 &\le \omega_{\ell-1}^2\Delta_\ell
 +(x^0-x^\star)^T(z^k-z^{\ell-1})\\
 &\quad
 -\frac{\|z^k\|^2-\|z^{\ell-1}\|^2}{12\bigl(L_0+\sqrt{2HL_0D_0}\bigr)}
 +D_0\sum_{i=\ell}^k\omega_i\|e^i\|.
\end{aligned}
\end{equation}
\end{lemma}

\begin{proof}
The weight identities in Lemma~\ref{lem:weights} give
$\omega_i=\omega_i^2-\omega_{i-1}^2$.  Summation by parts therefore yields
the exact identity
\begin{equation}\label{eq:weighted-gap-identity}
 \omega_k^2\Delta_{k+1}
 +\sum_{i=\ell}^k\omega_i^2d_i
 =\omega_{\ell-1}^2\Delta_\ell+\sum_{i=\ell}^k\omega_i\Delta_i.
\end{equation}
Indeed,
\[
 \sum_{i=\ell}^k\omega_i\Delta_i
 =\omega_k^2\Delta_k-\omega_{\ell-1}^2\Delta_\ell
   +\sum_{i=\ell}^{k-1}\omega_i^2(\Delta_i-\Delta_{i+1}),
\]
and substituting $\Delta_k=\Delta_{k+1}+d_k$ gives
\eqref{eq:weighted-gap-identity}.

Convexity at $x^i$ and the decomposition
$\nabla f(x^i)=h^i+e^i$ imply
\[
\begin{aligned}
 \Delta_i
 &\le \nabla f(x^i)^T(x^i-x^\star)\\
 &=(h^i)^T(x^i-x^0)+(h^i)^T(x^0-x^\star)
   +(e^i)^T(x^i-x^\star)\\
 &\le (h^i)^T(x^0-x^\star)+D_0\|e^i\|.
\end{aligned}
\]
Here $(h^i)^T(x^i-x^0)=0$ by \eqref{eq:orthogonalities}, and
$\|x^i-x^\star\|\le D_0$ because $x^i\in\mathcal C_0$.  Since
$z^i-z^{i-1}=\omega_i h^i$, summing the last inequality after multiplication
by $\omega_i$ gives
\[
 \sum_{i=\ell}^k\omega_i\Delta_i
 \le (x^0-x^\star)^T(z^k-z^{\ell-1})
     +D_0\sum_{i=\ell}^k\omega_i\|e^i\|.
\]

The second orthogonality in \eqref{eq:orthogonalities} gives a Pythagorean
identity at every update:
\[
 \|z^i\|^2
 =\|z^{i-1}+\omega_i h^i\|^2
 =\|z^{i-1}\|^2+\omega_i^2\|h^i\|^2.
\]
Moreover, $h^i$ is the orthogonal-projection residual of
$\nabla f(x^i)$, so $\|h^i\|\le\|\nabla f(x^i)\|$.  Telescoping and using
\eqref{eq:algorithm-gradient} now give
\begin{equation}\label{eq:aggregate-interval}
 \|z^k\|^2-\|z^{\ell-1}\|^2
 =\sum_{i=\ell}^k\omega_i^2\|h^i\|^2
 \le3\mathcal L\sum_{i=\ell}^k\omega_i^2d_i.
\end{equation}
Split the decrease sum on the left of
\eqref{eq:weighted-gap-identity} into three quarters and one quarter.
By \eqref{eq:aggregate-interval}, the latter quarter is at least
$(\|z^k\|^2-\|z^{\ell-1}\|^2)/(12\mathcal L)$.  Moving this lower bound to
the right proves \eqref{eq:window-reserve}.
\end{proof}

We next apply the interval inequality starting at iteration zero.
This bounds both the objective gap and the norm of the aggregate.
These bounds will control what has already accumulated when a
later interval begins.

\begin{lemma}[Prefix estimates]\label{lem:prefix}
For every integer $k\ge0$,
\begin{equation}\label{eq:window-prefix}
\begin{aligned}
 \omega_k^2\Delta_{k+1}
 +\frac14\sum_{i=0}^k\omega_i^2d_i
 &\le3\bigl(L_0+\sqrt{2HL_0D_0}\bigr)\|x^0-x^\star\|^2
   +\frac{512}{3}HD_0^3\log(k+1).
\end{aligned}
\end{equation}
Consequently,
\begin{equation}\label{eq:window-aggregate-bound}
\begin{aligned}
 \|z^k\|^2
 &\le12\bigl(L_0+\sqrt{2HL_0D_0}\bigr)\bigl(1+\log(k+1)\bigr)\\
 &\quad\cdot\Bigl[3\bigl(L_0+\sqrt{2HL_0D_0}\bigr)
                 \|x^0-x^\star\|^2+\tfrac{512}{3}HD_0^3\Bigr].
\end{aligned}
\end{equation}
The initial gap and every shifted prefix gap satisfy
\begin{equation}\label{eq:prefix-gap}
\begin{aligned}
 \Delta_0&\le3\bigl(L_0+\sqrt{2HL_0D_0}\bigr)\|x^0-x^\star\|^2,\\
 \Delta_{\ell-1}&\le\frac{4(1+\log\ell)}{\ell^2}
 \Bigl[3\bigl(L_0+\sqrt{2HL_0D_0}\bigr)\|x^0-x^\star\|^2
                          +\tfrac{512}{3}HD_0^3\Bigr]
 \quad(\ell\ge1).
\end{aligned}
\end{equation}
\end{lemma}

\begin{proof}
Apply Lemma~\ref{lem:reserve} with $\ell=0$.  The conventions
$\omega_{-1}=0$ and $z^{-1}=0$, together with $e^0=0$, reduce its right-hand
side to
\[
 (x^0-x^\star)^Tz^k-\frac{\|z^k\|^2}{12\mathcal L}
 +D_0\sum_{i=1}^k\omega_i\|e^i\|.
\]
Completing the square gives
\begin{equation}\label{eq:prefix-complete-square}
 (x^0-x^\star)^Tz^k-\frac{\|z^k\|^2}{12\mathcal L}
 =3\mathcal L\|x^0-x^\star\|^2
  -\frac{\|z^k-6\mathcal L(x^0-x^\star)\|^2}{12\mathcal L}
 \le3\mathcal L\|x^0-x^\star\|^2.
\end{equation}

We use the following elementary scalar inequality.  For
$\alpha\ge0$, $\beta>0$, and $t\ge0$,
\begin{equation}\label{eq:weighted-young}
 \alpha t^{2/3}-\beta t\le\frac{4\alpha^3}{27\beta^2}.
\end{equation}
To verify it, write $t=y^3$ with $y\ge0$.  The derivative of
$\alpha y^2-\beta y^3$ is $y(2\alpha-3\beta y)$, so the maximum is attained
at $y=2\alpha/(3\beta)$ (or at $y=0$ if $\alpha=0$), and its value is the
right-hand side of \eqref{eq:weighted-young}.

For $i\ge1$, the defect estimate \eqref{eq:defect-budget} and
\eqref{eq:weighted-young}, with
\[
 \alpha=D_0(18H)^{1/3}\omega_i,\qquad
 \beta=\frac{\omega_{i-1}^2}{2},\qquad
 t=d_{i-1},
\]
give
\begin{equation}\label{eq:prefix-defect-young}
 D_0\omega_i\|e^i\|-\frac12\omega_{i-1}^2d_{i-1}
 \le\frac{32}{3}HD_0^3
       \frac{\omega_i^3}{\omega_{i-1}^4}.
\end{equation}
After summing \eqref{eq:prefix-defect-young} from $i=1$ to $k$, the
decrease left after this absorption is at least one quarter of the whole
prefix, because
\[
 \frac34\sum_{i=0}^k\omega_i^2d_i
 -\frac12\sum_{i=1}^k\omega_{i-1}^2d_{i-1}
 =\frac14\sum_{i=0}^{k-1}\omega_i^2d_i
  +\frac34\omega_k^2d_k
 \ge\frac14\sum_{i=0}^k\omega_i^2d_i.
\]

The bounds in \eqref{eq:weights} imply, for $i\ge1$,
\[
 \frac{\omega_i^3}{\omega_{i-1}^4}
 \le\frac{(i+1)^3}{((i+1)/2)^4}
 =\frac{16}{i+1}.
\]
For $k\ge1$, monotonicity of $u\mapsto1/u$ gives
\[
 \sum_{i=1}^k\frac1{i+1}
 \le\int_1^{k+1}\frac{du}{u}=\log(k+1);
\]
for $k=0$ both sides of the corresponding empty-sum estimate are zero.
Combining this bound with \eqref{eq:window-reserve},
\eqref{eq:prefix-complete-square}, and
\eqref{eq:prefix-defect-young} proves the first inequality in
\eqref{eq:window-prefix}.  The second follows from the definition
\eqref{eq:proof-scale}.

Equation \eqref{eq:aggregate-interval} with $\ell=0$, followed by
\eqref{eq:window-prefix}, gives
\[
 \|z^k\|^2
 \le3\mathcal L\sum_{i=0}^k\omega_i^2d_i
 \le12\mathcal L C_{\mathrm{global}}(1+\log(k+1)),
\]
which is \eqref{eq:window-aggregate-bound}.

It remains to prove \eqref{eq:prefix-gap}.  Convexity and
\eqref{eq:algorithm-gradient} give
\[
 \Delta_0
 \le\|x^0-x^\star\|\,\|\nabla f(x^0)\|,
 \qquad
 \|\nabla f(x^0)\|^2\le3\mathcal L d_0
 \le3\mathcal L\Delta_0.
\]
If $\Delta_0=0$, the first assertion is immediate.  Otherwise, divide the
combined inequality by $\sqrt{\Delta_0}$ and square to obtain
$\Delta_0\le3\mathcal L\|x^0-x^\star\|^2$.

For $\ell=1$, this initial estimate and \eqref{eq:proof-scale} imply
$\Delta_0\le C_{\mathrm{global}}\le4C_{\mathrm{global}}(1+\log1)$.  If $\ell\ge2$, use
\eqref{eq:window-prefix} with $k=\ell-2$ and the weight lower bound
$\omega_{\ell-2}\ge \ell/2$:
\[
 \Delta_{\ell-1}
 \le\frac{C_{\mathrm{global}}(1+\log(\ell-1))}{\omega_{\ell-2}^2}
 \le\frac{4C_{\mathrm{global}}(1+\log \ell)}{\ell^2}.
\]
This completes the proof.
\end{proof}

We now apply the interval inequality where the weights approximately
double. The next lemma bounds the correction errors on such an
interval and gives the recurrence used in the main theorem.
These intervals only organize the proof; the algorithm continues
without restarting.

\begin{lemma}[Comparable-weight window]\label{lem:window}
Fix an integer $\ell\ge1$, and let $k\ge \ell$ be the first index such that
\[
 \omega_k\ge2\omega_{\ell-1}.
\]
Then
\begin{equation}\label{eq:window-gap-recurrence}
\begin{aligned}
 \omega_k^2\Delta_{k+1}
 +\frac14\sum_{i=\ell}^k\omega_i^2d_i
 &\le\omega_{\ell-1}^2\Delta_\ell\\
 &\quad+16\sqrt{1+\log\ell}
 \Bigl[3\bigl(L_0+\sqrt{2HL_0D_0}\bigr)\|x^0-x^\star\|^2
                       +\tfrac{512}{3}HD_0^3\Bigr].
\end{aligned}
\end{equation}
\end{lemma}

\begin{proof}
The lower weight bound in \eqref{eq:weights} shows that the defining index
$k$ exists.  Since $\omega_{\ell-1}\ge1$ and the weight recurrence gives the
strict inequality $\omega_\ell<\omega_{\ell-1}+1\le2\omega_{\ell-1}$, we have
$k\ge \ell+1$.  By minimality, $\omega_{k-1}<2\omega_{\ell-1}$.  The increment
bounds in Lemma~\ref{lem:weights} therefore give
\[
 \omega_{\ell-1}+\frac{k-\ell}{2}
 \le\omega_{k-1}<2\omega_{\ell-1},
 \qquad\text{hence}\qquad k-\ell<2\omega_{\ell-1}.
\]
For $\ell<i\le k$, we also have
$\omega_{i-1}\ge\omega_{\ell-1}$ and
$\omega_i\le\omega_{i-1}+1\le2\omega_{i-1}$.  Thus
\begin{equation}\label{eq:window-ratio-sum}
 \sum_{i=\ell+1}^k\frac{\omega_i^3}{\omega_{i-1}^4}
 \le\sum_{i=\ell+1}^k\frac{8}{\omega_{i-1}}
 \le\frac{8(k-\ell)}{\omega_{\ell-1}}
 \le16.
\end{equation}

We first bound the aggregate boundary term in Lemma~\ref{lem:reserve}.
The Pythagorean identity in \eqref{eq:aggregate-interval} shows that
$\|z^k\|\ge\|z^{\ell-1}\|$.  Put
$t=\|z^k\|-\|z^{\ell-1}\|\ge0$.  The triangle inequality gives
$\|z^k-z^{\ell-1}\|\le2\|z^{\ell-1}\|+t$, while
\[
 \|z^k\|^2-\|z^{\ell-1}\|^2
 =2\|z^{\ell-1}\|t+t^2.
\]
Consequently,
\begin{align*}
 &(x^0-x^\star)^T(z^k-z^{\ell-1})
 -\frac{\|z^k\|^2-\|z^{\ell-1}\|^2}{12\mathcal L}\\
 &\quad\le
 2\|x^0-x^\star\|\,\|z^{\ell-1}\|
 +\left(\|x^0-x^\star\|-\frac{\|z^{\ell-1}\|}{6\mathcal L}\right)t
 -\frac{t^2}{12\mathcal L}\\
 &\quad\le
 2\|x^0-x^\star\|\,\|z^{\ell-1}\|
 +3\mathcal L\|x^0-x^\star\|^2.
\end{align*}
For the last line, the maximum over $t\ge0$ of the last two terms is
\[
 3\mathcal L
 \max\!\left\{\|x^0-x^\star\|-
       \frac{\|z^{\ell-1}\|}{6\mathcal L},0\right\}^2
 \le3\mathcal L\|x^0-x^\star\|^2.
\]
Using \eqref{eq:window-aggregate-bound} with $k=\ell-1$ and
$3\mathcal L\|x^0-x^\star\|^2\le C_{\mathrm{global}}$, we obtain
\begin{equation}\label{eq:window-boundary}
\begin{aligned}
 &(x^0-x^\star)^T(z^k-z^{\ell-1})
 -\frac{\|z^k\|^2-\|z^{\ell-1}\|^2}{12\mathcal L}\\
 &\qquad\le
 4C_{\mathrm{global}}\sqrt{1+\log \ell}+C_{\mathrm{global}}
 \le5C_{\mathrm{global}}\sqrt{1+\log \ell}.
\end{aligned}
\end{equation}

The bound for $e^\ell$ uses $d_{\ell-1}$, the decrease immediately
before the current interval.  Since
$d_{\ell-1}\le\Delta_{\ell-1}$, Lemma~\ref{lem:prefix},
$\omega_\ell\le \ell+1\le2\ell$, and \eqref{eq:defect-budget} give
\begin{align*}
 D_0\omega_\ell\|e^\ell\|
 &\le2D_0(18H)^{1/3}\ell
       \left(\frac{4C_{\mathrm{global}}(1+\log \ell)}{\ell^2}\right)^{2/3}\\
 &=2\,4^{2/3}D_0(18H)^{1/3}C_{\mathrm{global}}^{2/3}
       \left(\frac{(1+\log \ell)^2}{\ell}\right)^{1/3}.
\end{align*}
The definition of $C_{\mathrm{global}}$ implies
\[
 \bigl[D_0(18H)^{1/3}\bigr]^3
 =18HD_0^3\le\frac{27}{256}C_{\mathrm{global}}.
\]
The function $v\mapsto(1+\log v)^2/v$ on $[1,\infty)$ has derivative
$(1+\log v)(1-\log v)/v^2$, and hence its maximum is $4/e$, attained at
$v=e$.  Substitution gives the explicit bound
\begin{equation}\label{eq:first-window-defect}
 D_0\omega_\ell\|e^\ell\|
 \le(54/e)^{1/3}C_{\mathrm{global}}<4C_{\mathrm{global}}.
\end{equation}

Every remaining correction is paid for inside the current window.  For
$\ell<i\le k$, the same scalar inequality \eqref{eq:weighted-young} used in
the prefix proof gives
\[
 D_0\omega_i\|e^i\|-\frac12\omega_{i-1}^2d_{i-1}
 \le\frac{32}{3}HD_0^3
       \frac{\omega_i^3}{\omega_{i-1}^4}.
\]
Summing and using \eqref{eq:window-ratio-sum} bounds the unabsorbed part by
\begin{equation}\label{eq:internal-window-defects}
 \frac{32}{3}HD_0^3
 \sum_{i=\ell+1}^k\frac{\omega_i^3}{\omega_{i-1}^4}
 \le\frac{512}{3}HD_0^3\le C_{\mathrm{global}}.
\end{equation}
The absorbed decreases are precisely those indexed from $\ell$ through
$k-1$.  Hence the decrease remaining on the left of
\eqref{eq:window-reserve} satisfies
\[
 \frac34\sum_{i=\ell}^k\omega_i^2d_i
 -\frac12\sum_{i=\ell+1}^k\omega_{i-1}^2d_{i-1}
 =\frac14\sum_{i=\ell}^{k-1}\omega_i^2d_i
  +\frac34\omega_k^2d_k
 \ge\frac14\sum_{i=\ell}^k\omega_i^2d_i.
\]
Finally, combine \eqref{eq:window-boundary},
\eqref{eq:first-window-defect}, and
\eqref{eq:internal-window-defects}.  Since
$\sqrt{1+\log \ell}\ge1$, their sum is at most
\[
 5C_{\mathrm{global}}\sqrt{1+\log \ell}+4C_{\mathrm{global}}+C_{\mathrm{global}}
 \le10C_{\mathrm{global}}\sqrt{1+\log \ell}
 \le16C_{\mathrm{global}}\sqrt{1+\log \ell}.
\]
Substitution in \eqref{eq:window-reserve} proves
\eqref{eq:window-gap-recurrence}.
\end{proof}

\subsubsection{Main result: cubic descent with SESOP aggregation}

The preceding estimates now give the global rate. The proof uses
intervals on which the weights approximately double. Their role is
only to organize the analysis; Algorithm~\ref{alg:cubic} carries
the same starting point, weights, and aggregate throughout its run.

After initialization, each iteration uses one new gradient and
at most three HVPs, with no function evaluations.

\begin{theorem}[Global convergence with three HVPs]\label{thm:main}
Under Assumptions~\ref{ass:convex}--\ref{ass:hessian},
Algorithm~\ref{alg:cubic}, with exact model solves and projections,
is well defined and generates a monotone sequence in $\mathcal C_0$.
If $\nabla f(x^0)\ne0$, then for every $K\ge0$,
\begin{equation}\label{eq:explicit-rate}
 f(x^K)-f_\star\le
 \frac{2048\left[3\bigl(L_0+\sqrt{2HL_0D_0}\bigr)\|x^0-x^\star\|^2
                    +\tfrac{512}{3}HD_0^3\right]
                  \sqrt{1+\log(K+1)}}{(K+1)^2},
\end{equation}

Consequently, for a universal constant $C$,
\begin{equation}\label{eq:main-rate}
 f(x^K)-f_\star\le C(L_0D_0^2+HD_0^3)
             \frac{\sqrt{1+\log(K+1)}}{(K+1)^2}.
\end{equation}
The same claims hold trivially for a stationary initial point and
after constant extension following finite termination.
\end{theorem}
\begin{proof}
Propositions~\ref{prop:iteration} and~\ref{prop:oracles} give
well-definedness, monotonicity, and the oracle count.
If $\nabla f(x^0)=0$, convexity makes $x^0$ a minimizer and there
is nothing to prove. Otherwise the preceding lemmas apply.
Recall that $\Delta_k=f(x^k)-f_\star$, $d_k=\Delta_k-\Delta_{k+1}$,
and that the bracket in \eqref{eq:explicit-rate} is the constant
$C_{\mathrm{global}}$ in \eqref{eq:proof-scale}.

Partition the positive indices into the intervals of
Lemma~\ref{lem:window}. Set $k_0=1$. For each interval number
$j\ge0$, let $k_{j+1}$ be one plus the first index at least $k_j$
for which $\omega_{k_{j+1}-1}\ge2\omega_{k_j-1}$.
Thus interval $j$ consists of the steps $k_j,\ldots,k_{j+1}-1$.
To track the weighted gap at its beginning, define
\[
 U_j=\omega_{k_j-1}^2\Delta_{k_j}.
\]
These intervals cover all positive indices because the weights
increase without bound.

For one such interval, write $\ell=k_j$ and $k=k_{j+1}-1$.
Monotonicity of the weights and telescoping of the decreases give
\[
 \sum_{i=\ell}^k\omega_i^2d_i
 \ge\omega_\ell^2\sum_{i=\ell}^k d_i
 =\omega_\ell^2(\Delta_\ell-\Delta_{k+1}).
\]
Substituting this into \eqref{eq:window-gap-recurrence} yields
\[
 \left(\omega_k^2-\tfrac14\omega_\ell^2\right)\Delta_{k+1}
 \le\left(\omega_{\ell-1}^2-\tfrac14\omega_\ell^2\right)\Delta_\ell
            +16C_{\mathrm{global}}\sqrt{1+\log \ell}.
\]
For the following algebra only, put
\[
 \alpha=\frac{\omega_\ell^2}{\omega_{\ell-1}^2},\qquad
 \beta=\frac{\omega_k^2}{\omega_{\ell-1}^2}.
\]
The weight increment bound and $\omega_{\ell-1}\ge1$ give
$1\le \alpha<4$, while the definition of $k$ gives $\beta\ge4$.
As $U_{j+1}=\omega_k^2\Delta_{k+1}$, the preceding inequality is
\[
 U_{j+1}\le
 \frac{1-\alpha/4}{1-\alpha/(4\beta)}U_j
 +\frac{16C_{\mathrm{global}}\sqrt{1+\log k_j}}{1-\alpha/(4\beta)}.
\]
Both coefficients can be bounded independently of the interval.
Indeed,
\[
 \frac{1-\alpha/4}{1-\alpha/(4\beta)}
 \le\frac{1-\alpha/4}{1-\alpha/16}\le\frac45,
 \qquad
 \frac1{1-\alpha/(4\beta)}\le\frac43.
\]
For the second inequality in the first chain, the rational function
decreases for $1\le \alpha<4$, so its maximum is attained at $\alpha=1$.
We have therefore obtained the contraction
\begin{equation}\label{eq:endpoint-contraction}
 U_{j+1}\le\frac45U_j+22C_{\mathrm{global}}\sqrt{1+\log k_j}.
\end{equation}

The initial value satisfies
$U_0=\Delta_1\le\Delta_0\le C_{\mathrm{global}}$ by Lemma~\ref{lem:prefix}.
Induction in \eqref{eq:endpoint-contraction}, using $k_{j+1}>k_j$,
gives
\[
 U_j\le110C_{\mathrm{global}}\sqrt{1+\log k_j}.
\]
The constant $110$ is chosen so that $(4/5)110+22=110$.
For an arbitrary $K\ge1$, choose the interval containing it,
$k_j\le K<k_{j+1}$. By monotonicity, $\Delta_K\le\Delta_{k_j}$.
The defining minimality of $k_{j+1}-1$ gives
$\omega_{K-1}<2\omega_{k_j-1}$, and
Lemma~\ref{lem:weights} gives $\omega_{K-1}\ge(K+1)/2$.
Combining these relations,
\begin{align*}
 \Delta_K
 &\le\frac{U_j}{\omega_{k_j-1}^2}
 \le\frac{16U_j}{(K+1)^2}\\
 &\le\frac{1760C_{\mathrm{global}}\sqrt{1+\log(K+1)}}{(K+1)^2}.
\end{align*}
For $K=0$, Lemma~\ref{lem:prefix} gives $\Delta_0\le C_{\mathrm{global}}$.
Thus the uniform factor $2048$ in \eqref{eq:explicit-rate} covers
all indices. At any terminating index the gap becomes zero, so
constant extension preserves the estimate.

Finally, $\|x^0-x^\star\|\le D_0$ and the elementary inequality
$\sqrt{ab}\le(a+b)/2$ give
\[
 \mathcal L D_0^2
 =L_0D_0^2+\sqrt{2(L_0D_0^2)(HD_0^3)}
 \le\tfrac32L_0D_0^2+HD_0^3.
\]
Absorbing the numerical constants proves \eqref{eq:main-rate}.
\end{proof}

The proof identifies where the logarithm enters and why its square
root is sufficient. Summing the correction costs from the initial
iteration produces a harmonic sum. That estimate bounds the norm
of $z^k$ by a square root of a logarithm. On an interval of comparable
weights, the internal correction costs have a uniformly bounded
sum, and only the aggregate at the interval boundary carries the
square-root logarithmic factor. The retained decrease then yields
the endpoint contraction \eqref{eq:endpoint-contraction}.

\subsection{Automatic local linear convergence}
\label{sec:local}

The linear rate in this section is local: it holds once the iterates
are sufficiently close to the minimizer. Under the additional
assumption below, the global theorem ensures that the iterates
reach this region from any starting point satisfying our assumptions.
The algorithm needs no test for entering the region and no
strong-convexity parameter.

We first prove entry into the local region. We then show that the
model step leaves only the error outside the search space, up to a
smaller remainder. Finally, the gradient direction gives a fixed
contraction of the error. We finish by replacing positive definiteness
with a local PL inequality, which also permits nonunique minimizers.

\begin{assumption}[Nondegenerate minimizer]
\label{ass:local}
The Hessian at the fixed minimizer $x^\star$ is positive definite:
\[
 v^T\nabla^2f(x^\star)v>0\qquad\text{for every }v\ne0.
\]
\end{assumption}

The local estimates use the smallest and largest eigenvalues of the
Hessian at the minimizer and their ratio:
\[
 \mu_\star=\lambda_{\min}(\nabla^2f(x^\star))>0,\qquad
 L_\star=\lambda_{\max}(\nabla^2f(x^\star)),\qquad
\kappa_\star=L_\star/\mu_\star.
\]
Thus $\kappa_\star$ is the local condition number. We measure vector
errors in the norm induced by the Hessian itself:
\begin{equation}\label{eq:local-error-norm}
 \|v\|_\star=\bigl(v^T\nabla^2f(x^\star)v\bigr)^{1/2}.
\end{equation}
It is equivalent to the Euclidean norm, since
\[
 \sqrt{\mu_\star}\|v\|\le\|v\|_\star\le\sqrt{L_\star}\|v\|.
\]
These quantities describe the local behavior; none is supplied to
Algorithm~\ref{alg:cubic}.

The global theorem gives convergence of function values. The next
lemma turns this into convergence of the iterates, so the local
estimates used below eventually apply.

\begin{lemma}[The iterates enter the local region]
\label{lem:local-entry}
Under Assumptions~\ref{ass:convex}, \ref{ass:level},
\ref{ass:hessian}, and~\ref{ass:local}, the minimizer $x^\star$
is unique and the iterates of Algorithm~\ref{alg:cubic} converge
to it. In particular, for all sufficiently large $k$,
\[
 \nabla^2f(x^k)\succeq\tfrac12\mu_\star I.
\]
\end{lemma}
\begin{proof}
Continuity of the Hessian gives a ball about $x^\star$ on which
$\nabla^2f(x)\succeq(\mu_\star/2)I$. Since
$\nabla f(x^\star)=0$, integrating along the segment from $x^\star$
to a point in that ball gives
\[
 f(x)\ge f(x^\star)+\tfrac14\mu_\star\|x-x^\star\|^2.
\]
Thus $x^\star$ is a strict local minimizer. If a second global
minimizer existed, convexity would make every point of the segment
between the two a minimizer. Points on that segment arbitrarily
close to $x^\star$ would contradict the displayed inequality.
Hence the global minimizer is unique.

Theorem~\ref{thm:main} gives $f(x^k)\to f(x^\star)$, and all
iterates lie in the compact initial sublevel $\mathcal C_0$.
By continuity, every cluster point must therefore be a global
minimizer and so must equal $x^\star$. This implies convergence
of the whole sequence: otherwise some subsequence would remain
a positive distance from $x^\star$, while compactness would give
it a cluster point with that same property. The Hessian bound now
follows from entry into the ball chosen above.
\end{proof}

The next lemma separates the error after a step into two parts:
the part of the current error outside the search space, and a
remainder quadratic in the current error. The final theorem will
bound the first part using the gradient direction. No independence
or angle condition on the directions defining the search space is needed.
Here $\dist_\star(v,\mathcal S)=\inf_{u\in\mathcal S}\|v-u\|_\star$
denotes distance in the norm~\eqref{eq:local-error-norm}.

\begin{lemma}[The error after one model step]
\label{lem:local-projection}
Under Assumptions~\ref{ass:convex}--\ref{ass:hessian} and~\ref{ass:local},
let $P^k$ be the orthogonal
projector onto $\mathcal S^k$ for the inner product $v^T\nabla^2f(x^\star)w$.
There exists a finite constant $C_1$, depending on the problem
and the fixed coefficient $H$ but independent of $k$ and the choice
of search space, such that, for all sufficiently large $k$,
\begin{equation}\label{eq:local-projection}
 \begin{aligned}
 x^{k+1}-x^\star&=(I-P^k)(x^k-x^\star)+w^k,\\
 w^k&\in\mathcal S^k,\qquad
 \|w^k\|_\star\le C_1\|x^k-x^\star\|_\star^2.
 \end{aligned}
\end{equation}
Consequently, using the distance defined just above the lemma,
\begin{equation}\label{eq:local-distance-bound}
 \|x^{k+1}-x^\star\|_\star\le\dist_\star(x^k-x^\star,\mathcal S^k)+C_1\|x^k-x^\star\|_\star^2.
\end{equation}
\end{lemma}
\begin{proof}
Write $u^k=x^k-x^\star$. The initial sublevel is convex and
contains both $x^k$ and $x^\star$. Assumption~\ref{ass:hessian}
therefore applies along their segment. In particular,
\begin{align}
 \|\nabla f(x^k)-\nabla^2f(x^\star)u^k\|
 &\le\int_0^1 2Ht\|u^k\|^2\,dt
 =H\|u^k\|^2,\label{eq:local-gradient-expansion}\\
 \|\nabla^2f(x^k)-\nabla^2f(x^\star)\|&\le2H\|u^k\|.
 \label{eq:local-hessian-expansion}
\end{align}
The first bound also gives
$\|\nabla f(x^k)\|\le L_\star\|u^k\|+H\|u^k\|^2=O(\|x^k-x^\star\|_\star)$.

Since $s^k\in\mathcal S^k$, taking its inner product with the
projected stationarity equation~\eqref{eq:model-stationarity}
gives
\[
 (s^k)^T\nabla^2f(x^k)s^k+H\|s^k\|^3
 =-\nabla f(x^k)^Ts^k.
\]
Lemma~\ref{lem:local-entry} and Cauchy--Schwarz imply
\[
 \tfrac12\mu_\star\|s^k\|^2
 \le\|\nabla f(x^k)\|\|s^k\|,
 \qquad
 \|s^k\|\le\frac{2}{\mu_\star}\|\nabla f(x^k)\|=O(\|x^k-x^\star\|_\star).
\]
The last inequality also holds when $s^k=0$.

Define the stationarity remainder
\[
 r^k=\nabla^2f(x^\star)u^k-\nabla f(x^k)
       +(\nabla^2f(x^\star)-\nabla^2f(x^k))s^k-H\|s^k\|s^k.
\]
The two Taylor bounds and the step-length estimate yield
\[
 \|r^k\|
 \le H\|u^k\|^2+2H\|u^k\|\|s^k\|+H\|s^k\|^2
 =O(\|x^k-x^\star\|_\star^2).
\]
For every $v\in\mathcal S^k$, equation~\eqref{eq:model-stationarity}
now gives
\[
 v^T\nabla^2f(x^\star)(u^k+s^k)=v^Tr^k.
\]
By the defining property of the projector orthogonal in the norm $\|\cdot\|_\star$,
$P^k(u^k+s^k)=P^k[\nabla^2f(x^\star)]^{-1}r^k$.
Also $(I-P^k)s^k=0$, because $s^k$ lies in the search space.
Together these identities prove~\eqref{eq:local-projection} with
$w^k=P^k[\nabla^2f(x^\star)]^{-1}r^k$. Orthogonal projection does not increase
the norm $\|\cdot\|_\star$, so
\[
 \|w^k\|_\star\le\|[\nabla^2f(x^\star)]^{-1}r^k\|_\star
 \le \mu_\star^{-1/2}\|r^k\|=O(\|x^k-x^\star\|_\star^2).
\]
All estimates are uniform for sufficiently large $k$, independently
of the dimension of $\mathcal S^k$ or the angles between its defining directions.
Finally, $(I-P^k)u^k$ is the error of the best approximation from
$\mathcal S^k$ in the norm $\|\cdot\|_\star$. The triangle inequality in
\eqref{eq:local-projection} gives~\eqref{eq:local-distance-bound}.
\end{proof}

It remains to show that the error decreases by a fixed factor.
We now use the gradient direction to bound the part outside the
search space. Together with the smaller remainder from the preceding
lemma, this gives local linear convergence for the same algorithm.

\begin{theorem}[Local linear rate without a supplied modulus]
\label{thm:local}
Under Assumptions~\ref{ass:convex}, \ref{ass:level},
\ref{ass:hessian}, and~\ref{ass:local}, the iterates converge to $x^\star$. Unless the method terminates
finitely, there is a finite constant $C$ such that, eventually,
\begin{equation}\label{eq:local-rate}
 \|x^{k+1}-x^\star\|_\star\le\frac{\kappa_\star-1}{\kappa_\star+1} \|x^k-x^\star\|_\star+C\|x^k-x^\star\|_\star^2.
\end{equation}
For every $q\in(\frac{\kappa_\star-1}{\kappa_\star+1},1)$, this gives
$\|x^{k+1}-x^\star\|_\star\le q\|x^k-x^\star\|_\star$ for all sufficiently large $k$. In the absence
of finite termination, the objective gaps also satisfy
\[
 \limsup_{k\to\infty}
 \frac{f(x^{k+1})-f_\star}{f(x^k)-f_\star}
 \le\left(\frac{\kappa_\star-1}{\kappa_\star+1}\right)^2.
\]
No strong-convexity parameter, restart, or test for entering a local
phase is supplied to the algorithm.
\end{theorem}
\begin{proof}
Convergence follows from Lemma~\ref{lem:local-entry}.
Because $\nabla f(x^k)\in\mathcal S^k$, we can approximate
$u^k=x^k-x^\star$ by
$2\nabla f(x^k)/(L_\star+\mu_\star)$ inside that space.
The approximation error is
\[
 u^k-\frac{2\nabla f(x^k)}{L_\star+\mu_\star}
 =\left(I-\frac{2\nabla^2f(x^\star)}{L_\star+\mu_\star}\right)u^k
   +\frac{2(\nabla^2f(x^\star)u^k-\nabla f(x^k))}{L_\star+\mu_\star}.
\]
The second term has size $O(\|x^k-x^\star\|_\star^2)$
in the norm $\|\cdot\|_\star$, by
\eqref{eq:local-gradient-expansion} and equivalence of the norms.
For the first term, diagonalize $\nabla^2f(x^\star)$. A component corresponding to
an eigenvalue $\lambda\in[\mu_\star,L_\star]$ is multiplied by
$1-2\lambda/(L_\star+\mu_\star)$, whose absolute value is at most
$\frac{\kappa_\star-1}{\kappa_\star+1}$. Therefore
\[
 \begin{aligned}
 \dist_\star(u^k,\mathcal S^k)
 &\le\left\|u^k-\frac{2\nabla f(x^k)}{L_\star+\mu_\star}\right\|_\star\\
 &\le\frac{\kappa_\star-1}{\kappa_\star+1} \|x^k-x^\star\|_\star+O(\|x^k-x^\star\|_\star^2).
 \end{aligned}
\]
Combining this with Lemma~\ref{lem:local-projection} proves
\eqref{eq:local-rate}. Given $q>\frac{\kappa_\star-1}{\kappa_\star+1}$, convergence ensures
$C\|x^k-x^\star\|_\star\le q-\frac{\kappa_\star-1}{\kappa_\star+1}$ eventually; dividing
\eqref{eq:local-rate} by $\|x^k-x^\star\|_\star$ then gives the claimed contraction.

Finally, the Hessian Lipschitz bound on the segment from $x^\star$
to $x^k$ gives the two-sided Taylor estimate
\[
 \left|f(x^k)-f_\star-\tfrac12\|x^k-x^\star\|_\star^2\right|
 \le\tfrac H3\|x^k-x^\star\|^3=O(\|x^k-x^\star\|_\star^3).
\]
Thus the objective gap is $\tfrac12\|x^k-x^\star\|_\star^2(1+O(\|x^k-x^\star\|_\star))$.
Substituting~\eqref{eq:local-rate} into the ratio of consecutive
gaps proves the last assertion.
\end{proof}

This is a local conclusion under an additional assumption at the
minimizer, rather than a global strong-convexity requirement.
The global theorem ensures entry, and the same fixed-$H$ iteration
then contracts automatically. The displayed factor is the usual
spectral interval bound from the gradient direction. The other
directions may improve the actual steps; the theorem does not
assert a faster dependence on the local condition number.

\subsubsection{A local PL condition allows a singular Hessian}

Positive definiteness gives the sharper geometric analysis above.
For linear objective convergence, a local
Polyak--\L{}ojasiewicz (PL) inequality is sufficient. The following
argument also proves that the iterates stay in the relevant
neighborhood; it does not assume entry near an arbitrarily selected
minimizer. We write $B(x,r)=\{y:\|y-x\|<r\}$ for the open ball.

\begin{theorem}[Local PL convergence and finite trajectory length]
\label{thm:local-pl}
Under Assumptions~\ref{ass:convex}--\ref{ass:hessian}, consider
a nonstationary run of the fixed-coefficient model steps, with
search spaces containing the current gradient and with restarts
permitted. Their decrease bounds are
\begin{equation}\label{eq:pl-two-budgets}
 \begin{aligned}
 f(x^k)-f(x^{k+1})&\ge
 \frac{\|\nabla f(x^k)\|^2}{3(L_0+\sqrt{2HL_0D_0})},\\
 f(x^k)-f(x^{k+1})&\ge
 \frac{2H}{3}\|x^{k+1}-x^k\|^3.
 \end{aligned}
\end{equation}
For an infinite run, suppose an accumulation point $\bar x$ has a
neighborhood on which, for some $r,\mu>0$,
\begin{equation}\label{eq:local-pl-condition}
 \|\nabla f(x)\|^2\ge2\mu(f(x)-f_\star)
 \quad\text{for all }x\in\mathcal C_0\cap B(\bar x,r).
\end{equation}
Decrease $\mu$ if necessary so that
$0<2\mu<3(L_0+\sqrt{2HL_0D_0})$, and put
\[
 q=1-\frac{2\mu}{3(L_0+\sqrt{2HL_0D_0})}\in(0,1).
\]
Then $\bar x$ is a minimizer and,
for some $k_0$ and every $k\ge k_0$,
\begin{equation}\label{eq:local-pl-conclusions}
 \begin{aligned}
 f(x^{k+1})-f_\star&\le q(f(x^k)-f_\star),\\
 \sum_{i=k}^\infty\|x^{i+1}-x^i\|
 &\le\frac{(f(x^k)-f_\star)^{1/3}}
              {(2H/3)^{1/3}(1-q^{1/3})},\\
 \|x^k-\bar x\|
 &\le\frac{(f(x^{k_0})-f_\star)^{1/3}}
              {(2H/3)^{1/3}(1-q^{1/3})}\,q^{(k-k_0)/3}.
 \end{aligned}
\end{equation}
Thus the gaps converge Q-linearly and the iterates converge
R-linearly to $\bar x$. Neither uniqueness of the minimizer nor a
positive-definite Hessian is required. The PL constant is used only
in the analysis. Finite termination at a stationary point is the
immediate alternative.
\end{theorem}
\begin{proof}
Within the proof, write
\[
 C_{\mathrm{grad}}=3(L_0+\sqrt{2HL_0D_0}),\qquad
 c_{\mathrm{step}}=2H/3.
\]
The argument uses only the two decrease bounds with these positive
coefficients. It therefore also applies when other variants supply
different positive coefficients, with $q=1-2\mu/C_{\mathrm{grad}}$.
Summing the first bound gives
$\sum_k\|\nabla f(x^k)\|^2<\infty$. All accumulation points are
therefore stationary, hence minimize $f$ by convexity. Compactness,
continuity, and monotonicity imply $f(x^k)\to f_\star$.
Write $\Delta_k=f(x^k)-f_\star$ within this proof.
At an iterate in the PL neighborhood, the two bounds give
\[
 \Delta_{k+1}\le q\Delta_k,\qquad
 \|x^{k+1}-x^k\|\le c_{\mathrm{step}}^{-1/3}\Delta_k^{1/3}.
\]
Choose an index $k_0$ along a subsequence converging to $\bar x$
such that
\[
 \|x^{k_0}-\bar x\|
 +\frac{\Delta_{k_0}^{1/3}}
       {c_{\mathrm{step}}^{1/3}(1-q^{1/3})}<r.
\]
If $x^K$ were the first subsequent iterate outside the ball, the
contraction holds at every preceding iterate, including $x^{K-1}$.
Consequently
\[
 \|x^K-\bar x\|
 \le\|x^{k_0}-\bar x\|
 +c_{\mathrm{step}}^{-1/3}\Delta_{k_0}^{1/3}
       \sum_{i=k_0}^{K-1}q^{(i-k_0)/3}<r,
\]
a contradiction. This proves invariance even if an individual
trial segment is not assumed to lie in the PL neighborhood.
Summing the step bound from any later $k$ proves the finite-length
estimate. The sequence is Cauchy, and its limit must equal its
accumulation point $\bar x$. The remaining estimate follows from
the geometric gap bound.
\end{proof}

The two bounds follow from Lemmas~\ref{lem:invariance}
and~\ref{lem:gradient-decrease} and remain valid across restarts.
Thus the result requires no positive-definite Hessian; it uses
only the stated local PL condition at an accumulation point.

For example,
\[
 f(u,v)=\frac{u^2}{2}+\frac{(|v|-1)_+^3}{3},
 \qquad t_+=\max\{t,0\},
\]
is convex and $C^2$, has a globally Lipschitz Hessian and compact
sublevel sets, and has minimizer set $\{0\}\times[-1,1]$.
Near every $(0,\bar v)$ with $|\bar v|<1$, it satisfies PL with
$\mu=1$, although its Hessian is $\operatorname{diag}(1,0)$.
No PL condition is asserted at $(0,\pm1)$. The accumulation-point
qualification in the theorem is essential: objective convergence
does not imply entry near an arbitrary prescribed minimizer.

\section{Beyond the baseline method}
\label{sec:extensions}

All results in this section keep the same cubic subspace step and reuse
the one-step estimates from Section~\ref{sec:method}. We first improve
the global bound and identify the oracle-model limit of that improvement.
We then change one ingredient at a time: the restart schedule, the
search space, and the accuracy of the reduced solve. Each variant also
admits the backtracking rule already given in Section~\ref{sec:adaptive}. This order separates the exact baseline theory from
the mechanisms that make it adaptive and finitely implementable.

\subsection{Removing the logarithm by deterministic restarts}
\label{sec:restarted-rate}

The basic method combines a global rate with automatic local linear
convergence. Its global estimate still contains a square root of a
logarithm. We now remove that factor by restarting the gradient
aggregate at prescribed iteration counts. The cubic step itself is
unchanged, so the same three-HVP bound and local linear guarantee
remain available.

The reason for restarting is visible in the preceding global proof.
The correction errors contribute a sum of weight ratios, as in
\eqref{eq:prefix-defect-young}. That sum grows logarithmically along
an uninterrupted run. On an epoch whose weights vary by only a
constant factor, it is bounded by a constant. We therefore start each
epoch with a larger weight and stop it before the weights have grown
too much. First we bound the progress within one such epoch; then we
show how successive epochs give the inverse-square rate.
We then state a matching lower bound for the specified vector-oracle
class, with its proof deferred to Appendix~\ref{app:oracle-lower-bound}.

\subsubsection{The restart rule}

Algorithm~\ref{alg:deterministic-restarts} states the complete schedule.
Here $j$ numbers the epochs, and $k$ numbers steps within the current
epoch; its points are $y^k$, with $y^0$ its starting point.
The step in line~\ref{line:det-step} may use either the supplied fixed
coefficient of Algorithm~\ref{alg:cubic} or the adaptive coefficient
selection of Algorithm~\ref{alg:adaptive-cubic}.

\begin{algorithm}[H]
\caption{Cubic SESOP with deterministic restarts}
\label{alg:deterministic-restarts}
\begin{algorithmic}[1]
\Require Initial point $x^0$; fixed or adaptive cubic-coefficient rule
\State Evaluate $\nabla f(x^0)$ and set $y^0\gets x^0$
\For{$j=0,1,\ldots$}
 \State $t\gets2^j$, $z^0\gets t\nabla f(y^0)$,
 $\omega_0\gets t$
 \For{$k=0,\ldots,4t$}
  \If{$\nabla f(y^k)=0$}
   \State \Return $y^k$
  \EndIf
  \State $\mathcal S^k\gets
   \Span\{\nabla f(y^k),y^k-y^0,z^k\}$
  \State Compute an exact cubic-model step $s^k\in\mathcal S^k$
  using the selected coefficient rule \Comment{at most three HVPs}
  \label{line:det-step}
  \State $y^{k+1}\gets y^k+s^k$ and evaluate $\nabla f(y^{k+1})$
  \State $\mathcal V^{k+1}\gets\Span\{y^{k+1}-y^0,z^k\}$
  \State $h^{k+1}\gets\nabla f(y^{k+1})-
       \Pi_{\mathcal V^{k+1}}\nabla f(y^{k+1})$
  \State $\omega_{k+1}\gets(1+\sqrt{1+4\omega_k^2})/2$
  \State $z^{k+1}\gets z^k+\omega_{k+1}h^{k+1}$
 \EndFor
 \State $y^0\gets y^{4t+1}$ \Comment{starting point of the next epoch}
\EndFor
\end{algorithmic}
\end{algorithm}

As before, stop whenever the current gradient is zero.
The endpoint gradient is already available for initializing the
next aggregate. No function value or estimated problem constant
is used to choose the epoch length. For adaptive steps, initialize the
coefficient as in Algorithm~\ref{alg:adaptive-cubic} and retain the last
accepted coefficient across epoch boundaries.

\subsubsection{Progress over one epoch}

The next lemma explains what is gained and what is lost by starting
with weight $t$ instead of one. The correction-error sum becomes
bounded over the epoch. In exchange, its estimate includes the
starting objective gap multiplied by $t^2-t$. The subsequent theorem
will show that the preceding epoch already makes this term small
enough.

\begin{lemma}[An epoch with an arbitrary initial weight]
\label{lem:epoch-bound}
Under Assumptions~\ref{ass:convex}--\ref{ass:hessian}, start an
epoch at $y^0\in\mathcal C_0$ with an integer weight $t\ge1$ and
aggregate $t\nabla f(y^0)$.
For every prefix of $k+1$ steps, where $0\le k\le4t$,
\begin{equation}\label{eq:epoch-bound}
\begin{aligned}
 \omega_k^2\bigl(f(y^{k+1})-f_\star\bigr)
 &\le(t^2-t)\bigl(f(y^0)-f_\star\bigr)
       +\frac32\bigl(L_0+\sqrt{2HL_0D_0}\bigr) D_0^2\\
 &\quad+\frac{32}{3}HD_0^3
          \sum_{i=1}^k\frac{\omega_i^3}{\omega_{i-1}^4}.
\end{aligned}
\end{equation}
In particular, the completed epoch satisfies
\begin{equation}\label{eq:completed-epoch}
 f(y^{4t+1})-f_\star
 \le\frac{(t^2-t)\bigl(f(y^0)-f_\star\bigr)
              +\tfrac32\bigl(L_0+\sqrt{2HL_0D_0}\bigr) D_0^2
              +\tfrac{1024}{3}HD_0^3}{9t^2}.
\end{equation}
The incoming previous displacement may be retained from before
the epoch.
\end{lemma}
\begin{proof}
Abbreviate $\mathcal L=L_0+\sqrt{2HL_0D_0}$ within this proof.
If $\nabla f(y^0)=0$, constant extension gives both estimates
immediately. Otherwise $\mathcal L>0$ by
\eqref{eq:gradient-level-bound}. Gradient inclusion gives
monotonicity and invariance by Lemma~\ref{lem:invariance}.
Moreover,
\[
 y^{i+1}-y^0=(y^i-y^0)+s^i\in\mathcal S^i,
 \qquad
 \Span\{y^{i+1}-y^0,z^i\}\subseteq\mathcal S^i.
\]
Thus the one-step bounds \eqref{eq:global-two-budgets} and the
corrected-gradient orthogonalities remain valid with $y^0$ as the epoch's starting point.
None of these facts requires the incoming displacement to be zero.

For this proof, write
\[
 \Delta_i=f(y^i)-f_\star,\qquad
 d_i=f(y^i)-f(y^{i+1})\ge0,
\]
and let $h^i$ and $e^i$ denote the corrected gradient and the
removed component within this epoch. Set
$h^0=\nabla f(y^0)$, $e^0=0$, and $z^{-1}=0$.
The weight recurrence gives
\[
 \omega_i^2-\omega_{i-1}^2=\omega_i\quad(i\ge1),
 \qquad
 t+\frac i2\le\omega_i\le t+i.
\]
The increment bounds follow from the same calculation as in
Lemma~\ref{lem:weights}. Telescoping the difference identity gives
$\sum_{i=0}^k\omega_i=\omega_k^2-t^2+t$.
Consequently, summation by parts yields
\begin{equation}\label{eq:epoch-gap-identity}
 \omega_k^2\Delta_{k+1}
 +\sum_{i=0}^k\omega_i^2d_i
 =(t^2-t)\Delta_0+\sum_{i=0}^k\omega_i\Delta_i.
\end{equation}
This is the only change to the starting boundary term in
\eqref{eq:weighted-gap-identity}.

As in the proof of Lemma~\ref{lem:reserve}, convexity and
$(h^i)^T(y^i-y^0)=0$ imply
\[
 \sum_{i=0}^k\omega_i\Delta_i
 \le(y^0-x^\star)^Tz^k
      +D_0\sum_{i=1}^k\omega_i\|e^i\|.
\]
The orthogonality $(h^i)^Tz^{i-1}=0$ gives
\[
 \|z^k\|^2
 =\sum_{i=0}^k\omega_i^2\|h^i\|^2
 \le3\mathcal L\sum_{i=0}^k\omega_i^2d_i.
\]
Use half of the weighted decrease in
\eqref{eq:epoch-gap-identity} to bound the aggregate. Completing
the square gives
\begin{align*}
 \omega_k^2\Delta_{k+1}
 +\frac12\sum_{i=0}^k\omega_i^2d_i
 &\le(t^2-t)\Delta_0+(y^0-x^\star)^Tz^k
       -\frac{\|z^k\|^2}{6\mathcal L}
       +D_0\sum_{i=1}^k\omega_i\|e^i\|\\
 &\le(t^2-t)\Delta_0
       +\frac32\mathcal L\|y^0-x^\star\|^2
       +D_0\sum_{i=1}^k\omega_i\|e^i\|.
\end{align*}
The remaining half of the decrease pays for the correction errors.
Indeed, the defect bound and \eqref{eq:weighted-young} give
\[
 D_0\omega_i\|e^i\|
 -\frac12\omega_{i-1}^2d_{i-1}
 \le\frac{32}{3}HD_0^3
       \frac{\omega_i^3}{\omega_{i-1}^4}
 \qquad(i\ge1).
\]
Summing this inequality and discarding the remaining nonnegative
term $\omega_k^2d_k/2$ proves \eqref{eq:epoch-bound}.

For a completed epoch, $k=4t$ and $\omega_{4t}\ge3t$.
Since $\omega_{i-1}\ge t\ge1$ and
$\omega_i\le\omega_{i-1}+1\le2\omega_{i-1}$,
\[
 \sum_{i=1}^{4t}\frac{\omega_i^3}{\omega_{i-1}^4}
 \le\sum_{i=1}^{4t}\frac8{\omega_{i-1}}
 \le32.
\]
Substitution in \eqref{eq:epoch-bound} proves
\eqref{eq:completed-epoch}.
\end{proof}

The proof uses only bounds on the sublevel containing the epoch.
Consequently, it also applies with $L_0$ and $D_0$ replaced by the
Hessian bound and radius of the smaller set
$\{x:f(x)\le f(y^0)\}$, and with the corresponding value of
$\mathcal L$. This observation will later let us estimate an epoch
more sharply near the minimizer.

\subsubsection{Combining the epochs}

We now use the one-epoch estimate to prove the global rate.
Doubling $t$ asks the next epoch to start with a gap four times
smaller. Inequality~\eqref{eq:completed-epoch} provides exactly
this reduction. Monotonicity then extends the bound from epoch
boundaries to every iterate.

\begin{theorem}[Inverse-square convergence with deterministic epochs]
\label{thm:restarted-rate}
Under Assumptions~\ref{ass:convex}--\ref{ass:hessian}, use the
restart rule above with exact model solves and projections. Index
the concatenated trajectory by the total number $K$ of cubic steps.
It is monotone, remains in $\mathcal C_0$, and satisfies
\begin{equation}\label{eq:restarted-rate}
 f(x^K)-f_\star\le
 \frac{81D_0^2}{(K+1)^2}
 \left[\frac32\bigl(L_0+\sqrt{2HL_0D_0}\bigr)+\frac{1024}{3}HD_0\right]
 \qquad(K\ge0),
\end{equation}
Under Assumption~\ref{ass:local}, it also retains the local linear
guarantee of Theorem~\ref{thm:local}.
\end{theorem}
After initialization, each step uses one new gradient and at most
three HVPs, with no function evaluations. Besides the supplied $H$,
the rule needs no smoothness bound, strong-convexity parameter,
radius, or optimal value.

\begin{proof}
For this induction, abbreviate the numerator of
\eqref{eq:restarted-rate} by
\[
 C_{\mathrm{restart}}=81D_0^2
 \left[\tfrac32\bigl(L_0+\sqrt{2HL_0D_0}\bigr)
                      +\tfrac{1024}{3}HD_0\right].
\]
Gradient inclusion is preserved at every restart, so
Lemma~\ref{lem:invariance} gives monotonicity and invariance
throughout the run. If a zero gradient is encountered, constant
extension preserves the claimed bound and requires no further
oracle calls. We therefore prove the estimate on the remaining
nonstationary trajectory.

The completed-epoch bound has additive term
$C_{\mathrm{restart}}/81$, where $C_{\mathrm{restart}}$ is the numerator
of the stated rate.
Suppose an epoch with initial weight $t$ starts at $y^0$ with
$f(y^0)-f_\star\le C_{\mathrm{restart}}/(81t^2)$. Then
\eqref{eq:completed-epoch} gives
\[
 f(y^{4t+1})-f_\star
 \le\frac{(t^2-t)C_{\mathrm{restart}}/(81t^2)+(C_{\mathrm{restart}}/81)}{9t^2}
 \le\frac{2(C_{\mathrm{restart}}/81)}{9t^2}
 \le\frac{C_{\mathrm{restart}}}{81(2t)^2}.
\]
This is the required starting estimate for the next epoch.
For the first epoch, $t=1$. Integrating the Hessian on the segment
from $x^\star$ to $x^0$ gives
\[
 f(x^0)-f_\star
 \le\frac{L_0}{2}\|x^0-x^\star\|^2
 \le\frac{L_0D_0^2}{2}\le C_{\mathrm{restart}}/81.
\]
Induction proves the starting estimate in every epoch. Neither the gap nor $C_{\mathrm{restart}}$ is used by the algorithm.

Let $\tau_j$ be the global index at the start of epoch $j$.
Summing the preceding epoch lengths gives
\[
 \tau_j=\sum_{\ell=0}^{j-1}(4\cdot2^\ell+1)
       =4(2^j-1)+j.
\]
If $\tau_j\le K<\tau_{j+1}$ and $t=2^j$, monotonicity and the
starting estimate give $f(x^K)-f_\star\le C_{\mathrm{restart}}/(81t^2)$.
Moreover,
\[
 K+1\le\tau_{j+1}=8t+j-3\le9t,
\]
where $j\le2^j=t$ for every integer $j\ge0$.
Thus $f(x^K)-f_\star\le C_{\mathrm{restart}}/(K+1)^2$, which is
\eqref{eq:restarted-rate}.

Each restarted space has at most three dimensions. Forming its
reduced Hessian and updating the gradient therefore has the same
oracle count as Proposition~\ref{prop:oracles}. Restarting uses
the current gradient already in memory and requires no additional
oracle call.

Finally, under Assumption~\ref{ass:local}, the minimizer is unique
by the argument in Lemma~\ref{lem:local-entry}. Compactness and
\eqref{eq:restarted-rate} imply $x^K\to x^\star$. The local
projection estimate in Lemma~\ref{lem:local-projection} and the
proof of Theorem~\ref{thm:local} use only gradient inclusion,
the fixed cubic coefficient, and the local Hessian bounds. They
do not depend on a fixed epoch starting point or aggregate. They consequently
apply across the epoch boundaries and give the same local rate.
\end{proof}

The increasing initial weight is part of this restart rule.
Resetting $\omega_0$ to one at every epoch would leave a logarithm
of the epoch length in \eqref{eq:epoch-bound}. Thus
Theorem~\ref{thm:restarted-rate} proves a logarithm-free bound for
the stated restarted method; the bound for the uninterrupted
Algorithm~\ref{alg:cubic} remains Theorem~\ref{thm:main}.

\subsection{Why the inverse-square global order is sharp}
\label{sec:oracle-lower-bound}

The restart rule gives a global inverse-square rate with a constant
number of new gradients and HVPs per step. The next theorem shows
that this exponent cannot be improved uniformly over the problem
class by adding momentum or retaining more history. We specify the
oracle model in the statement: every query and output must stay in
the linear span of the vector information already obtained. The
budget counts all gradient and HVP calls, including inner calls.

We adapt the classical extremal conjugate-gradient example of
Meinardus~\cite{meinardus1963}, described in detail by
Li~\cite{li2008meinardus}, within the sharp polynomial framework
of Greenbaum~\cite{greenbaum1979}. The adaptation fixes the problem
constants, adds a smooth nonquadratic coordinate, and counts both
gradients and HVPs. Its purpose is to establish sharpness for our
oracle model using this classical construction.

\newcommand{\OracleLowerBoundStatement}{%
A vector-oracle call returns either $\nabla f(x)$ or
$\nabla^2f(x)v$. Consider algorithms starting at $0$ for which every
query point, HVP input vector, and output point belongs to the linear
span of previously returned vectors; before the first call this span
is $\{0\}$. Scalar coefficients may depend arbitrarily on the observed
information. Memory, scalar computations, and function-value queries
within the available span are unrestricted. Every gradient or HVP
evaluation counts, including the initial gradient and any inner calls.

For every integer call budget $M\ge1$, there exists a convex nonquadratic
$C^\infty$ function $f:\R^{M+2}\to\R$ with compact sublevel sets
and a unique minimizer $x^\star$, such that
\[
 0\preceq\nabla^2f(x)\preceq3I\qquad(x\in\R^{M+2}).
\]
Its smallest global Hessian Lipschitz constant is exactly $2$,
so $H=1$ is a valid supplied bound. Its third derivative is globally
$32$-Lipschitz in the induced multilinear norm. For
$\mathcal C_0=\{x:f(x)\le f(0)\}$, the level-set constants are
\[
 L_0=\max_{x\in\mathcal C_0}\|\nabla^2f(x)\|=1,
 \qquad D_0=\max_{x\in\mathcal C_0}\|x-x^\star\|=1.
\]
On this single function, every algorithm in the stated class and
every output $x$ obtained after at most $M$ vector-oracle calls satisfy
\[
 f(x)-f_\star\ge\frac{1}{2\cosh^2(2)(M+1)^2},
 \qquad f_\star=f(x^\star).
\]
Monotonicity is not required. In particular, a uniformly bounded
number of gradient and HVP calls per outer iteration cannot yield
a uniform dimension-independent $O(k^{-3})$ guarantee with constants
depending only on the displayed smoothness and level-set bounds.
}

\begin{theorem}[Adapted classical Krylov lower bound]
\label{thm:oracle-lower-bound}
\OracleLowerBoundStatement
\end{theorem}

The full statement and proof are given in
Appendix~\ref{app:oracle-lower-bound}. For the fixed-coefficient
method and its restart variants, the total vector-oracle budget
through $K$ steps is at most $1+4K$. Thus the inverse-square
exponent is sharp within this oracle class. The theorem allows
momentum, orthogonalization, HVP outputs as new directions, and exact
minimization inside the available span. It excludes external probe
directions and transformations that create vectors outside that span.
The hard function and its dimension depend on $M$; this is a global
worst-case statement and does not restrict the faster local rates
proved for individual problems.

\subsection{Automatic accelerated local convergence by restart control}
\label{sec:automatic-acceleration}

The basic method already converges linearly near a nondegenerate
minimizer. When the curvature varies greatly between directions,
its guaranteed contraction can nevertheless be slow. We now improve
this local guarantee by choosing how often to restart. The target is
an iteration count proportional to the square root of the local
condition number, rather than to the condition number itself.
The method will not need to know that number.

There are two questions to answer. First, how long should an epoch
be near the minimizer? The next lemma shows that an appropriate
epoch reduces the objective gap by a fixed factor. Second, how can
we find such a length without knowing either the curvature or the
optimal value? We compare the decreases observed over consecutive
epochs and increase the length when those decreases do not shrink
fast enough.

This length search must also work when the starting point is far
from a minimizer. We therefore alternate it with the deterministic
schedule from Section~\ref{sec:restarted-rate}. Those deterministic
stages retain the global inverse-square guarantee. The stages that
choose their length from observed decreases provide the faster local
rate. Both stages advance the same sequence of points, using the
same cubic model step; they do not run competing copies of the method.

Choosing restarts from observed progress has precedents for accelerated
gradient methods \cite{fercoqqu2019,alamo2022}. The results below give
the global and local guarantees for this cubic SESOP controller.
All model solves remain exact, as in the basic algorithm.
Assumptions~\ref{ass:convex}--\ref{ass:hessian} remain in force.
The local condition number is
$\kappa_\star=L_\star/\mu_\star$, with $\mu_\star,L_\star$ the extreme
Hessian eigenvalues defined in Section~\ref{sec:local}.
It is used only to state and prove the rate.

\subsubsection{What one epoch gives near a minimizer}

We first convert the one-epoch estimate into a reduction of the
current objective gap. Near the minimizer, strong convexity bounds
the size of the current sublevel in terms of that gap. This is the
link between the global epoch estimate and a local contraction.

\begin{lemma}[A locally contracting epoch]
\label{lem:local-epoch}
Start an epoch at $y\in\mathcal C_0$, with initial weight $t\ge1$
and aggregate $t\nabla f(y)$. Perform $4t+1$ steps, where $t$ is an
integer, and denote the endpoint by $y^+$.
Suppose the current sublevel set
$\mathcal C(y)=\{x:f(x)\le f(y)\}$ is contained in a neighborhood
where $\nabla^2 f(x)\succeq\mu I$, with $\mu>0$. Define
\[
 D=\max_{x\in\mathcal C(y)}\|x-x^\star\|,\qquad
 L=\max_{x\in\mathcal C(y)}\|\nabla^2 f(x)\|.
\]
Here $D$ is the radius of the current sublevel and $L$ is its
Hessian norm bound. For $\Delta=f(y)-f_\star$,
\begin{equation}\label{eq:local-epoch-factor}
 f(y^+)-f_\star
 \le
 \frac{t^2-t+3\bigl(L+\sqrt{4HLD}\bigr)/\mu
                 +(4096/3)HD/\mu}{9t^2}\,\Delta.
\end{equation}
The incoming previous displacement may be retained from before
the epoch.
\end{lemma}
\begin{proof}
The epoch proof \eqref{eq:epoch-bound} uses only gradient inclusion,
the corrected-gradient orthogonalities relative to the epoch's
starting point, and the local decrease bounds. It does not require that
the incoming displacement be zero. Apply it on $\mathcal C(y)$,
using the one-step bounds for fixed coefficients or, for adaptive
steps, those proved in Appendix~\ref{app:backtracking}.
For every $x\in\mathcal C(y)$, strong convexity along the segment
from $x^\star$ to $x$ gives
\[
 \frac\mu2\|x-x^\star\|^2
 \le f(x)-f_\star\le f(y)-f_\star=\Delta.
\]
Taking the maximum over that set gives $D^2\le2\Delta/\mu$.
Substitute this and $D^3\le2D\Delta/\mu$ into that proof's
bound for $k=4t$, and use $\omega_{4t}\ge3t$.
\end{proof}

As the starting sublevel shrinks to a minimizer with positive
definite Hessian, the numerator contribution
$3\bigl(L+\sqrt{4HLD}\bigr)/\mu+(4096/3)HD/\mu$ is of order
$\kappa_\star$. Thus epochs with $t$ of order
$\sqrt{\kappa_\star}$ contract the gap by a fixed factor. This
observation does not prescribe an implementable length, since
the local curvature bounds are unknown.

\subsubsection{The restart controller}

The objective gap cannot be observed, because $f_\star$ is unknown.
The decrease over an epoch can be observed. If each epoch reduces
the gap by a factor of at most one third, consecutive decreases
shrink by at least a factor of two; the proof below verifies this
implication. We use that decrease comparison to decide whether an
epoch length should be increased.

We also start this search afresh at regular intervals. This prevents
an epoch length chosen far from the solution from remaining fixed
throughout the local phase. Algorithm~\ref{alg:controlled-restarts}
combines that search with the global guarantee already proved. An
``epoch with weight $t$'' means the inner epoch loop of
Algorithm~\ref{alg:deterministic-restarts}, initialized with that weight
and run for the stated number of accepted steps.

\begin{algorithm}[H]
\caption{Cubic SESOP with controlled restarts}
\label{alg:controlled-restarts}
\begin{algorithmic}[1]
\Require Initial point $x^0$; fixed or adaptive cubic-coefficient rule
\State Evaluate $f(x^0)$ and $\nabla f(x^0)$; set $k\gets0$
\For{$j=0,1,\ldots$} \Comment{$j$ is the phase number}
 \State $T\gets2^j$
 \Statex \textit{Global stage}
 \State From $x^k$, run Algorithm~\ref{alg:deterministic-restarts}
 for $T$ accepted steps, restarting its epoch number at zero;
 store its endpoint as $x^{k+T}$ and set $k\gets k+T$
 \Statex \textit{Tuned stage}
 \State $t\gets1$, $d_{\mathrm{prev}}\gets\varnothing$, $b\gets T$
 \While{$b>0$}
  \State $y^0\gets x^k$, $\ell\gets\min\{4t+1,b\}$
  \State From $y^0$, run an epoch with initial weight $t$ for
  $\ell$ accepted steps; store its endpoint as $x^{k+\ell}$
  \State $k\gets k+\ell$, $b\gets b-\ell$
  \If{$\ell=4t+1$} \Comment{the epoch was not truncated}
   \State $d\gets f(y^0)-f(x^k)$
   \If{$d_{\mathrm{prev}}\ne\varnothing$ and
          $d>d_{\mathrm{prev}}/2$}
    \State $t\gets2t$, $d_{\mathrm{prev}}\gets\varnothing$
   \Else
    \State $d_{\mathrm{prev}}\gets d$
   \EndIf
  \EndIf
 \EndWhile
\EndFor
\end{algorithmic}
\end{algorithm}

In this listing, $k$ counts all accepted steps, $j$ numbers phases,
$T$ is the budget of each stage, $b$ is its remaining budget, and
$\ell$ is the length of the next epoch. The scalar
$d_{\mathrm{prev}}$ stores the previous complete epoch decrease.
At either stage boundary, truncate the last epoch if its budget
would be exceeded. A truncated epoch is not used in the decrease
comparison. Every new epoch resets its starting point, aggregate, and
weights, including when $t$ is retained. Keep the last accepted
regularization coefficient across every boundary when backtracking is
used. Stop at a zero gradient.

Only endpoint function values are new inputs to the restart decision.
The restart decision uses neither $f_\star$, local Hessian bounds,
nor a strong-convexity parameter. In the fixed-coefficient version these endpoint evaluations
are an additional cost; they are already available in the adaptive
version.

\subsubsection{A global inverse-square rate and a local accelerated rate}

The next theorem gives both guarantees for this single trajectory.
Its global statement counts every step in both stages. Its local
statement says that the objective gap eventually decreases
exponentially, with an exponent proportional to
$1/\sqrt{\kappa_\star}$. The proof first protects the global rate,
then bounds the epoch lengths near the minimizer, and finally counts
the progress made by the stages that adapt those lengths.

\begin{theorem}[Automatic local acceleration with a global guarantee]
\label{thm:automatic-acceleration}
Under Assumptions~\ref{ass:convex}--\ref{ass:hessian}, use exact
fixed-coefficient steps or the adaptive steps of Section~\ref{sec:adaptive}.
Apply the preceding controller and index its single
trajectory by the total number $K$ of accepted cubic steps.
The trajectory is monotone, stays in $\mathcal C_0$, and satisfies
\begin{equation}\label{eq:controlled-global-rate}
 f(x^K)-f_\star\le\frac{5184D_0^2}{(K+1)^2}
 \left[\frac32\bigl(L_0+\sqrt{4HL_0D_0}\bigr)
                    +\frac{2048}{3}HD_0\right]
 \qquad(K\ge0).
\end{equation}
Under Assumption~\ref{ass:local}, with $\kappa_\star$ defined above,
$x^K\to x^\star$ and
\begin{equation}\label{eq:controlled-local-root}
 \limsup_{K\to\infty}(f(x^K)-f_\star)^{1/K}
 \le \exp\left(-\frac{\log 2}{116\sqrt{\kappa_\star}}\right).
\end{equation}
For a positive initial gap, reducing it by a factor $\varepsilon$
requires at most
\[
 \left(\frac{116\sqrt{\kappa_\star}}{\log2}+o(1)\right)
 \log(1/\varepsilon)
 \qquad\text{as }\varepsilon\downarrow0,
\]
where the $o(1)$ is for the fixed problem.
\end{theorem}
The root estimate describes the local R-rate for a fixed problem;
it does not by itself quantify the transient. The next corollary
bounds that finite cost separately. The controller is supplied
with neither $\kappa_\star$ nor a strong-convexity parameter.

Every accepted step uses at most three HVPs. The fixed version uses
one new gradient per step and at most $K+1$ function evaluations
through step $K$, with caching. For the adaptive version the trial
and oracle bounds of Proposition~\ref{prop:counts} are unchanged.
Finite termination is covered by constant extension without
additional oracle calls.

\begin{proof}
For this proof, define
\[
 C_{\mathrm{restart}}=81D_0^2
 \left[\frac32\bigl(L_0+\sqrt{4HL_0D_0}\bigr)
                    +\frac{2048}{3}HD_0\right].
\]
Every reset preserves gradient inclusion and the correction-space
containment. The accepted-step decrease and invariance arguments
therefore apply throughout. The global stage also retains its
inverse-square estimate when it starts with an incoming previous
displacement: as in Lemma~\ref{lem:local-epoch}, the epoch potential
does not use a zero displacement. The original level-set bounds
$L_0,D_0$ and Assumption~\ref{ass:hessian} apply on every starting sublevel.

\paragraph{Step 1: retain the global bound.}
For the proof, write $T_j=2^j$ for the stage budget in phase $j$.
That phase has $2T_j$ steps and starts after $2(T_j-1)$ accepted
steps. For $j\ge1$, the preceding global stage
performed $T_j/2$ steps. Its bound and monotonicity give, everywhere
in the current phase,
\[
 f(x^K)-f_\star\le\frac{C_{\mathrm{restart}}}{(T_j/2+1)^2}
 \le\frac{4C_{\mathrm{restart}}}{T_j^2}.
\]
Since $K+1\le4T_j$ in that phase, this proves
\eqref{eq:controlled-global-rate}. The initial phase follows from
$f(x^0)-f_\star\le L_0D_0^2/2\le C_{\mathrm{restart}}$ and monotonicity.

\paragraph{Step 2: bound the epoch lengths near the minimizer.}
Positive definiteness makes the minimizer unique. The global
bound and compactness imply $x^K\to x^\star$.
Write $\mu_\star=\lambda_{\min}(\nabla^2f(x^\star))$ and
$L_\star=\lambda_{\max}(\nabla^2f(x^\star))$.
Choose a sufficiently small objective threshold $r>0$ and its sublevel
$\mathcal C_r=\{x:f(x)\le f_\star+r\}\subset\mathcal C_0$.
It shrinks to $x^\star$ as $r\downarrow0$.
On this sublevel we may take $\mu=\mu_\star/2$, bound the Hessian
norm by $2L_\star$, and use
\[
 D_r=\max_{x\in\mathcal C_r}\|x-x^\star\|,\qquad
 \mathcal L=2L_\star+\sqrt{8HL_\star D_r}\le3L_\star.
\]
Here $D_r$ is the radius of this smaller sublevel and
$\mathcal L$ is its effective gradient-decrease bound.
Further reduce $r$ so that
$(4096/3)HD_r/\mu\le\kappa_\star$. Define
\[
 \Gamma=\frac{3\mathcal L}{\mu}
                  +\frac{4096HD_r}{3\mu}\le19\kappa_\star,
 \qquad \theta=\frac{2\mu}{3\mathcal L}\ge\frac{1}{9\kappa_\star}.
\]
Here $\Gamma$ bounds the curvature contribution to one epoch,
and $0<\theta<1$ is a lower bound on the fraction of the current
gap removed by its first step. Let $t_\star$ be the smallest power of two not less
than $\max\{1,\sqrt{\Gamma/2}\}$, and set
\[
 R=\log_2t_\star,\qquad \ell_{\max}=4t_\star+1.
\]
$R$ counts possible doublings of the epoch weight, and
$\ell_{\max}$ bounds the length of any tuned epoch.
Since $\kappa_\star\ge1$, these choices give
$t_\star<7\sqrt{\kappa_\star}$ and
$\ell_{\max}\le29\sqrt{\kappa_\star}$.

Lemma~\ref{lem:local-epoch} bounds the contraction of every complete
epoch inside $\mathcal C_r$ by
$(t^2-t+\Gamma)/(9t^2)$. At $t=t_\star$ this is at most $1/3$.
For this paragraph, number successive complete epochs by $j$,
write $y^j$ for their starting points, and put
$\Delta_j=f(y^j)-f_\star$ and $d_j=f(y^j)-f(y^{j+1})$.
Their decreases satisfy
\[
 d_j\ge\tfrac23\Delta_j,\qquad
 d_{j+1}\le\Delta_{j+1}\le\tfrac13\Delta_j,
 \qquad d_{j+1}\le d_j/2.
\]
The controller can therefore never double beyond $t_\star$ in a
tuned stage that starts inside $\mathcal C_r$.
Every epoch of that stage has length at most $\ell_{\max}$.

Strong convexity is needed only along the segment joining $x$ to
$x^\star$, which lies in the same sublevel. It gives
\[
 f(x)-f_\star
 \le\nabla f(x)^T(x-x^\star)-\frac\mu2\|x-x^\star\|^2
 \le\frac{\|\nabla f(x)\|^2}{2\mu}.
\]
Thus the first-step gradient decrease bound also controls the
whole epoch decrease:
\[
 d_j\ge \theta\Delta_j,
\]
using the gradient decrease bound with $3\mathcal L$.

\paragraph{Step 3: count progress even when comparisons fail.}
Fix a consecutive group of $J\ge1$ completed epochs at the same
value of $t$. Number these epochs by $j=0,\ldots,J-1$ and denote their
decreases by $d_j$.
Write $\Delta_{\mathrm{start}}$ and $\Delta_{\mathrm{end}}$ for the objective
gaps at the beginning and end of this group.
All comparisons except possibly the last pass, since a failure
ends the group. For $J\ge2$,
\[
 d_{J-2}\le2^{-(J-2)}d_0
 \le2^{-(J-2)}\Delta_{\mathrm{start}}.
\]
The first-step bound for epoch $J-2$ and monotonicity give
\[
 \Delta_{\mathrm{end}}
 \le\frac{d_{J-2}}{\theta}
 \le\frac4\theta\,2^{-J}\Delta_{\mathrm{start}}.
\]
For $J=1$ the same bound follows from monotonicity and $\theta<1$.
To count the work, retain the actual length of each group rather
than assigning length $\ell_{\max}$ to every epoch.
For a visited level $t=2^\ell$, let $J_\ell$ be the number of its
completed epochs, and let its logarithmic gap reduction be
\[
 G_\ell=\log_2
 \frac{\text{gap before these epochs}}{\text{gap after these epochs}}.
\]
Assume these gaps are positive; finite termination is immediate.
For a level with no completed epoch, set $J_\ell=G_\ell=0$.
Monotonicity and the preceding group bound give
$G_\ell\ge\max\{0,J_\ell-\log_2(4/\theta)\}$, hence
$J_\ell\le\log_2(4/\theta)+G_\ell$. Each such epoch has length
$4\cdot2^\ell+1$. The level only increases, and $0\le\ell\le R$.
The sum of these distinct lengths is bounded by
\[
 \sum_{\ell=0}^R(4\cdot2^\ell+1)
 =8t_\star+R-3\le3\ell_{\max},
\]
where $R\le t_\star-1$ and $\ell_{\max}=4t_\star+1$.
Now let $\Delta_{\mathrm{start}},\Delta_{\mathrm{end}}$ denote the gaps at the
endpoints of the whole tuned stage, of budget $T$, and put
$G=\log_2(\Delta_{\mathrm{start}}/\Delta_{\mathrm{end}})$.
The completed groups are consecutive,
and the final partial epoch is monotone, so
$\sum_{\ell=0}^R G_\ell\le G$. That partial epoch has fewer than $\ell_{\max}$
steps. Summing the actual work of the completed groups yields
\[
 \begin{aligned}
 T-\ell_{\max}&\le\sum_{\ell=0}^R(4\cdot2^\ell+1)J_\ell\\
 &\le3\ell_{\max}\log_2(4/\theta)
                  +\ell_{\max}\sum_{\ell=0}^R G_\ell\\
 &\le3\ell_{\max}\log_2(4/\theta)+\ell_{\max}G.
 \end{aligned}
\]
The same inequality holds if no epoch completes, since then $T<\ell_{\max}$.
It follows that $G\ge T/\ell_{\max}-1-3\log_2(4/\theta)$, or equivalently
\begin{equation}\label{eq:tuned-stage}
 \Delta_{\mathrm{end}}\le F\exp(-\beta T)\Delta_{\mathrm{start}},
 \qquad
 F=2(4/\theta)^3,\qquad \beta=\frac{\log2}{\ell_{\max}}.
\end{equation}
The scalar $F$ accounts for adapting the epoch length, while
$\beta$ is the progress exponent within a tuned stage. Both are
analysis constants; the controller uses neither.

\paragraph{Step 4: combine the progress of successive phases.}
Choose a phase $j_0$ whose starting point belongs to $\mathcal C_r$,
and let $k_0=2(T_{j_0}-1)$ be its starting iteration index.
All later iterates stay in that sublevel. Before any iterate in
phase $j>j_0$, the completed tuned stages have total budget
$\sum_{\ell=j_0}^{j-1}T_\ell=T_j-T_{j_0}$.
The intervening global stages are monotone. Repeatedly applying
\eqref{eq:tuned-stage} therefore gives
\[
 f(x^K)-f_\star
 \le\bigl(f(x^{k_0})-f_\star\bigr)
       F^{j-j_0}\exp[-\beta(T_j-T_{j_0})].
\]
Since $K<4T_j$ and $j/K\to0$, taking $K$th roots gives a bound
$\exp(-\beta/4)$, which proves \eqref{eq:controlled-local-root}
using $\ell_{\max}\le29\sqrt{\kappa_\star}$. For every fixed positive
$c<(\log2)/(116\sqrt{\kappa_\star})$, the root bound implies
$f(x^K)-f_\star\le \exp(-cK)$ for all sufficiently large $K$.
For a positive initial gap, the number of steps needed to reduce
it by a factor $\varepsilon$ consequently has limiting upper
ratio to $\log(1/\varepsilon)$ at most $1/c$. Letting $c$
increase to $(\log2)/(116\sqrt{\kappa_\star})$ proves the
stated asymptotic iteration bound. A zero initial gap is covered
by immediate termination.

Resets need neither new gradients nor HVPs. Fixed-coefficient
function evaluations are made only at selected accepted iterates,
so caching bounds their number by $K+1$. In the adaptive case the
same values are already available. The regularization coefficient
is retained across boundaries, so its amortized trial proof is
unchanged.
\end{proof}

The local assertion is an R-rate for objective gaps, not a claim
of the same contraction at every individual step. The corresponding
root bounds for iterate errors and gradient norms have half the
exponent, by the local quadratic and gradient estimates.
The entry cost depends on the initial point, Hessian variation,
and curvature at the minimizer. The following bound makes these
dependencies explicit. In particular, \eqref{eq:controlled-local-root} belongs
to this restart controller and is not a theorem about the
uninterrupted baseline or the increasing-epoch rule
alone. 

\subsubsection{Explicit entry and adaptation costs}

The Hessian Lipschitz assumption also quantifies entry into the
local regime. The constants below are conservative and used only
in the analysis; the controller is unchanged.
The bound separates three costs: reaching a sufficiently small
sublevel, adapting the epoch length, and subsequently reducing
the objective gap to the requested accuracy.

\begin{corollary}[A finite global-to-local complexity bound]
\label{cor:explicit-transition}
Under the assumptions of Theorem~\ref{thm:automatic-acceleration},
suppose $\nabla^2 f(x^\star)\succ0$. Write
$\mu_\star=\lambda_{\min}(\nabla^2 f(x^\star))$ and retain
$\kappa_\star$ from that theorem and $H$ from
Assumption~\ref{ass:hessian}. Define the
local objective threshold and a sufficient entry index by
\begin{align*}
 \varepsilon_{\mathrm{loc}}&=\frac{\mu_\star^3}{2^{27}H^2},\\
 K_{\mathrm{entry}}&=\left\lceil72D_0
 \sqrt{\frac{\tfrac32(L_0+\sqrt{4HL_0D_0})
                   +\tfrac{2048}{3}HD_0}
                  {\varepsilon_{\mathrm{loc}}}}\right\rceil.
\end{align*}
Let $T_{\mathrm{adapt}}$ be the smallest power of two satisfying both
\begin{align*}
 T_{\mathrm{adapt}}&\ge K_{\mathrm{entry}}+2,\\
 T_{\mathrm{adapt}}&\ge\frac{58\sqrt{\kappa_\star}}{\log2}
 \left[\log2+3\log(36\kappa_\star)\right].
\end{align*}
This is the budget of a phase of the existing controller, not a
parameter supplied to it. For every $K\ge2(T_{\mathrm{adapt}}-1)$,
\begin{equation}\label{eq:controlled-finite-tail}
 f(x^K)-f_\star\le \varepsilon_{\mathrm{loc}}
 \exp\left[-\frac{\log2}{232\sqrt{\kappa_\star}}
                    (K-4T_{\mathrm{adapt}})\right].
\end{equation}
Consequently, for every $0<\varepsilon<\varepsilon_{\mathrm{loc}}$, objective
accuracy $f(x^K)-f_\star\le\varepsilon$ is guaranteed once
\begin{equation}\label{eq:controlled-finite-complexity}
 K\ge\left\lceil4T_{\mathrm{adapt}}
       +\frac{232\sqrt{\kappa_\star}}{\log2}
                     \log\frac{\varepsilon_{\mathrm{loc}}}{\varepsilon}\right\rceil.
\end{equation}
In particular, a sufficient number of accepted steps is
\begin{equation}\label{eq:controlled-separated-costs}
\begin{aligned}
 O\!\Biggl(&1+D_0\sqrt{
 \frac{\tfrac32(L_0+\sqrt{4HL_0D_0})
                         +\tfrac{2048}{3}HD_0}
      {\varepsilon_{\mathrm{loc}}}}\\
 &+\sqrt{\kappa_\star}\left[
       \log(2+\kappa_\star)+\log\frac{\varepsilon_{\mathrm{loc}}}{\varepsilon}
                       \right]\Biggr),
\end{aligned}
\end{equation}
with a universal implicit constant. None of these analysis
parameters changes the oracle counts of
Theorem~\ref{thm:automatic-acceleration}.
\end{corollary}
\begin{proof}
The entry index is $\lceil8\sqrt{C_{\mathrm{restart}}/
\varepsilon_{\mathrm{loc}}}\rceil$, using the numerator abbreviation
from the proof of Theorem~\ref{thm:automatic-acceleration}.
Set $\rho=\mu_\star/(4096H)$ and
$L_\star=\lambda_{\max}(\nabla^2 f(x^\star))$.
For $x\in\mathcal C_0$ with $\|x-x^\star\|\le\rho$, the Hessian
Lipschitz bound gives
\[
 \|\nabla^2 f(x)-\nabla^2 f(x^\star)\|
 \le2H\rho\le\frac{\mu_\star}{1024}.
\]
In particular, the Hessian there is bounded below by
$(\mu_\star/2)I$ and above by $2L_\star I$.
If a point $x\in\mathcal C_0$ lies at distance at least $\rho$
from $x^\star$, its segment to $x^\star$ lies in the convex set
$\mathcal C_0$. At the first point $z$ on that segment at distance
$\rho$ from $x^\star$, integration of the Hessian lower bound gives
$f(z)-f_\star\ge \mu_\star\rho^2/4$.
Convexity along the ray implies $f(x)\ge f(z)$. Since
$\varepsilon_{\mathrm{loc}}=\mu_\star\rho^2/8$, we have proved
\[
 \mathcal C_0\cap\{x:f(x)-f_\star\le \varepsilon_{\mathrm{loc}}\}
 \subset B(x^\star,\rho).
\]
Only segments inside $\mathcal C_0$ were used; the full ball need
not lie in that set.

On this local sublevel, take $\mu=\mu_\star/2$ and let $D$ be the
maximum distance to $x^\star$, so that $D\le\rho$.
The constants in the proof of
Theorem~\ref{thm:automatic-acceleration} then satisfy
\[
 2L_\star+\sqrt{8HL_\star D}\le3L_\star,\qquad
 \frac{4096HD}{3\mu}\le\frac23\le\kappa_\star.
\]
Substitution into Lemma~\ref{lem:local-epoch} and the first-step
decrease bound reproduces the local estimates used in that proof.
Its tuned-stage argument \eqref{eq:tuned-stage} therefore gives,
for a stage of $T$ accepted steps starting in this set,
\[
 \Delta_{\mathrm{end}}\le F_{\mathrm{loc}}\exp(-\beta_{\mathrm{loc}} T)
                         \Delta_{\mathrm{start}},\qquad
 F_{\mathrm{loc}}=2(36\kappa_\star)^3,
 \quad\beta_{\mathrm{loc}}=\frac{\log2}{29\sqrt{\kappa_\star}},
\]
Here $\Delta_{\mathrm{start}}$ and $\Delta_{\mathrm{end}}$ are the objective
gaps at the stage endpoints; $F_{\mathrm{loc}}$ and
$\beta_{\mathrm{loc}}$ are explicit versions of the stage constants
in \eqref{eq:tuned-stage}. The global bound gives
$f(x^K)-f_\star\le \varepsilon_{\mathrm{loc}}$ for every $K\ge K_{\mathrm{entry}}$.
The phase of budget $T_{\mathrm{adapt}}$ starts at
$2(T_{\mathrm{adapt}}-1)\ge K_{\mathrm{entry}}$ and hence lies in this sublevel.
For every subsequent tuned stage, its budget $T$ satisfies
$T\ge T_{\mathrm{adapt}}\ge2\log F_{\mathrm{loc}}/\beta_{\mathrm{loc}}$.
Its contraction is therefore at most
$\exp(-\beta_{\mathrm{loc}} T/2)$.

In a phase with budget $T_j\ge T_{\mathrm{adapt}}$, the completed tuned
stages since the phase of budget $T_{\mathrm{adapt}}$ have total budget
$T_j-T_{\mathrm{adapt}}$. The intervening steps are monotone, so every
iterate in this phase satisfies
\[
 f(x^K)-f_\star\le \varepsilon_{\mathrm{loc}}
       \exp[-\beta_{\mathrm{loc}}(T_j-T_{\mathrm{adapt}})/2].
\]
This also holds in the first such phase, by monotonicity.
Since $K<4T_j$, it implies \eqref{eq:controlled-finite-tail}.
Solving that inequality for $K$ proves
\eqref{eq:controlled-finite-complexity}. Finally, the next power
of two is less than twice its positive threshold,
$\log F_{\mathrm{loc}}=O(\log(2+\kappa_\star))$, and
$\beta_{\mathrm{loc}}^{-1}=O(\sqrt{\kappa_\star})$.
These facts prove \eqref{eq:controlled-separated-costs}.
\end{proof}

The first nonconstant term in \eqref{eq:controlled-separated-costs}
is an entry cost; it depends on the curvature scale $\mu_\star$,
the Hessian Lipschitz constant $2H$, and the initial sublevel scales. The next
term pays for restart adaptation, and the last term describes
subsequent accelerated decrease. For accuracies
$\varepsilon\ge \varepsilon_{\mathrm{loc}}$ the global estimate alone suffices.
The bound is an accuracy guarantee, not a stopping test involving
the unknown optimal value.

\subsection{An optional previous-step direction}

\subsubsection{Why three directions suffice}

The proofs use the current gradient and the two directions that carry
the corrected aggregate. They do not require the previous displacement,
which is why Algorithms~\ref{alg:cubic}--\ref{alg:controlled-restarts}
use only three HVPs. Some implementations nevertheless benefit from
retaining the preceding step. The next statement permits this practical
addition without making it part of the theoretical construction.

\begin{corollary}[Optional fourth memory direction]
\label{cor:three-hvp-restarts}
In any of Algorithms~\ref{alg:cubic}--\ref{alg:controlled-restarts},
let $s^{k-1}$ be the preceding accepted step and enlarge the stated
search space to
\[
 \widehat{\mathcal S}^k
 =\Span\{\nabla f(y^k),s^{k-1},y^k-y^0,z^k\}.
\]
Keep the same cubic solve, corrected aggregate, weights, coefficient
rule, and restart decisions.
For exact model solves, the baseline guarantee of
Theorem~\ref{thm:main}, the logarithm-free deterministic-restart
guarantee of Theorem~\ref{thm:restarted-rate}, the global estimate
\eqref{eq:controlled-global-rate}, the local accelerated bound
\eqref{eq:controlled-local-root}, and the finite transition bound of
Corollary~\ref{cor:explicit-transition} all hold with the same constants.
These conclusions apply to both the fixed-coefficient and adaptive
versions whenever the corresponding original theorem does.
Each accepted step then uses at most four HVPs; the function, gradient,
and adaptive-trial bounds are unchanged.
The ordinary local linear conclusion of Theorem~\ref{thm:local}
and the local PL conclusion of Theorem~\ref{thm:local-pl} also apply
under their respective assumptions.
\end{corollary}
\begin{proof}
The original three-direction space is contained in
$\widehat{\mathcal S}^k$. The current gradient belongs to this enlarged
space, and $y^{k+1}-y^0=(y^k-y^0)+s^k\in\widehat{\mathcal S}^k$.
Therefore
$\Span\{y^{k+1}-y^0,z^k\}\subseteq\widehat{\mathcal S}^k$.
These are the only direction inclusions used by the cubic decrease bounds,
the epoch potential, and the proof of
Theorem~\ref{thm:automatic-acceleration}. Its estimates and their
finite-time consequence are consequently unchanged.
Theorem~\ref{thm:local} requires only gradient inclusion, while
Theorem~\ref{thm:local-pl} requires the same decrease bounds.
Finally, a basis of $\widehat{\mathcal S}^k$ has at most four columns,
so forming the reduced Hessian takes at most four HVPs, reused
throughout backtracking.
\end{proof}

This statement does not predict better iteration counts or running
time. It only permits the memory direction while keeping the guarantees.
The next subsection permits finite model solves in both variants.

\subsection{A finite stopping rule for the reduced cubic solve}
\label{sec:inexact-model}

The reduced problem need not be solved to an exact argmin. A relative
stationarity test uses only its already computed coefficients and
preserves the basic, adaptive, and deterministic-epoch guarantees
specified below. Corollary~\ref{cor:inexact-acceleration} then extends
the accelerated restart controller and its finite-time bounds to
the same criterion. The required absolute
accuracy decreases with the step; a fixed absolute solver tolerance
does not provide this guarantee.

\begin{proposition}[Inexact reduced cubic solves]\label{prop:inexact-model}
Retain the assumptions of Theorem~\ref{thm:main} for the fixed
version and Theorem~\ref{thm:ag-rate} for the adaptive version.
For the nondegenerate local claims, retain the assumptions of
Theorem~\ref{thm:local} and bounded model coefficients.
At a nonstationary iterate, replace the exact model minimizer by any
$s^k\in\mathcal S^k$ satisfying
\begin{equation}\label{eq:inexact-model}
 \left\|\Pi_{\mathcal S^k}\left(
 \nabla f(x^k)+\nabla^2 f(x^k)s^k
       +H_k\|s^k\|s^k\right)\right\|
 \le\frac{H_k}{6}\|s^k\|^2,\qquad H_k>0.
\end{equation}
Here $H_k=H$ for the fixed algorithm, while in backtracking $H_k$
denotes the coefficient of the current trial. The projection
$\Pi_{\mathcal S^k}$ is Euclidean.
Keep the same search spaces, aggregate update, and, in the adaptive
case, the same two Taylor acceptance tests
\eqref{eq:ftest}--\eqref{eq:gtest}.
Then monotonicity, bounded backtracking parameters, and the global
$O(\sqrt{1+\log(K+1)}/(K+1)^2)$ rate continue to hold, with changed
universal constants. With the deterministic epochs of
Section~\ref{sec:restarted-rate}, the global rate is $O(K^{-2})$,
again with changed constants.
Under Assumption~\ref{ass:local}, the local projection decomposition
and the local linear guarantee of Theorem~\ref{thm:local} also remain
valid, including across restarts.
The local PL conclusions of Theorem~\ref{thm:local-pl} hold under
its accumulation-point assumption. For adaptive coefficients,
additionally suppose $H_k\ge H_{\min}>0$ for all sufficiently
large $k$, where $H_{\min}$ is not supplied to the algorithm.
The contraction factor is
\[
 q=1-\frac{2\mu}{7(L_0+\sqrt{(12/5)HL_0D_0})}.
\]
Decrease $\mu$ if needed to ensure $0<q<1$. In the finite-length
and iterate-error bounds, replace $(2H/3)^{1/3}$ by $(H/2)^{1/3}$
for fixed coefficients, or by $(H_{\min}/2)^{1/3}$ for adaptive
coefficients under the stated lower-bound condition.

For each trial, any algebraic solver whose iterates converge to the
reduced-model minimizer eventually meets \eqref{eq:inexact-model}.
Checking and refining this
criterion uses the cached reduced Hessian and gradient, and requires
no additional function, gradient, or Hessian-vector oracle calls.
For a specified geometric bisection solver,
Proposition~\ref{prop:cubic-inner-work} gives an explicit bound on the
number of reduced linear solves. Projections and the reduced coefficients
remain exact.
\end{proposition}
The proof is given in Appendix~\ref{app:finite-solves}.

The residual test must also preserve the progress used by the restart
controller. The next statement makes its changed global constant explicit
and retains the accelerated local rate and finite transition bounds.

\begin{corollary}[Accelerated restart control with finite model solves]
\label{cor:inexact-acceleration}
Under the fixed or adaptive assumptions of
Proposition~\ref{prop:inexact-model}, apply the controller of
Section~\ref{sec:automatic-acceleration}, requiring
\eqref{eq:inexact-model} at every model solve.
Use either the stated three-direction space or the optional
four-direction space of Corollary~\ref{cor:three-hvp-restarts}.
Then
\begin{equation}\label{eq:inexact-controlled-global}
 f(x^K)-f_\star\le\frac{5184D_0^2}{(K+1)^2}
 \left[\frac72\left(L_0+\sqrt{\frac{12}{5}HL_0D_0}\right)
                            +2048HD_0\right]
 \qquad(K\ge0).
\end{equation}
If $\nabla^2f(x^\star)\succ0$, the local root bound
\eqref{eq:controlled-local-root} also holds.
Corollary~\ref{cor:explicit-transition}, including its finite
accuracy bound, remains valid with the entry index
\[
 K_{\mathrm{entry}}=\left\lceil72D_0
 \sqrt{\frac{\tfrac72(L_0+\sqrt{(12/5)HL_0D_0})
                             +2048HD_0}
                  {\varepsilon_{\mathrm{loc}}}}\right\rceil.
\]
Its local objective threshold and other numerical constants are unchanged.

The function, gradient, and adaptive-trial bounds of
Theorem~\ref{thm:automatic-acceleration} remain valid. Each accepted
step uses at most three HVPs, or four with the optional memory direction.
Meeting the finite model-residual criterion requires no additional
objective-oracle calls. Reduced coefficients and projections remain
exact, as stipulated in Proposition~\ref{prop:inexact-model}.
\end{corollary}
The proof, including the transition constants, is given in
Appendix~\ref{app:finite-solves}.

\section{CaCuSESOP: a gradient trial with subspace memory and restarts}
\label{sec:cacusesop}

Following the Casual Cubic Technique (CaCuTe)
\cite{tupitsa2025}, we first try a cheap scaled-gradient step and
use a fallback when its certificate fails. Here the candidate is
the explicit gradient-line minimum of the cubic model. We first
complete its backtracking search using one HVP. We then test whether
the accepted line step is sufficiently stationary in the full SESOP
space. If it is not, we form the remaining HVPs and minimize the cubic
model over that space. The name CaCuSESOP combines
\emph{Casual Cubic} with \emph{Sequential Subspace Optimization}.

The coefficient is selected by backtracking, and the epochs follow
Algorithm~\ref{alg:controlled-restarts}. The practical space contains
the three directions required by the proof, the previous step, and a
Hessian--gradient direction. We first state the gradient-trial
certificate and prove the resulting convergence guarantees. We then
explain the roles of the two additional directions.

\subsection{An explicit gradient candidate and its certificate}

Fix a nonstationary iterate $x^k$ within an epoch starting at $y^0$.
Here $k$ is the accepted-step counter of the controller; $y^0$ remains
fixed until the next restart. The superscripts on the stored aggregate
and step refer to the same counter.
Let $\mathcal S^k$ be the span of
$\nabla f(x^k),s^{k-1},x^k-y^0,z^k$, and
$\nabla^2f(x^k)\nabla f(x^k)$. To construct it, first set
\[
 q_1^k=\frac{\nabla f(x^k)}{\|\nabla f(x^k)\|},
 \qquad v^k=\nabla^2f(x^k)q_1^k,
\]
which costs one HVP. The vector $v^k$ is a scaled
Hessian--gradient direction. Choose an orthonormal basis $Q^k$ of
$\mathcal S^k$ whose first column is $q_1^k$. Linearly dependent
directions are removed, so the number of columns satisfies $r_k\le5$.
For a positive trial coefficient $\widehat H$, the model is
\[
 m_{\widehat H}(u)
 =\nabla f(x^k)^Tu+\tfrac12u^T\nabla^2f(x^k)u
                         +\tfrac{\widehat H}{3}\|u\|^3.
\]
Its minimum along the negative gradient is $p=-\alpha q_1^k$, where
\begin{equation}\label{eq:cacu-gradient-root}
 \widehat H\alpha^2+
 \bigl((q_1^k)^T\nabla^2f(x^k)q_1^k\bigr)\alpha
 =\|\nabla f(x^k)\|,\qquad \alpha>0.
\end{equation}
This is a scalar quadratic, with an explicit positive root. Only
the already computed vector $v^k$ is needed to compute it.

Function decrease alone does not justify replacing the subspace step.
We also require
\begin{equation}\label{eq:cacu-trial-residual}
 \left\|\Pi_{\mathcal S^k}
  \bigl(\nabla f(x^k)+\nabla^2f(x^k)p
                         +\widehat H\|p\|p\bigr)\right\|
 \le \widehat H\|p\|^2.
\end{equation}
Here $\Pi_{\mathcal S^k}$ is the Euclidean orthogonal projector.
The component along $q_1^k$ is zero by
\eqref{eq:cacu-gradient-root}. Thus the certificate is equivalently
\begin{equation}\label{eq:cacu-transverse-test}
 \left\|\bigl((Q^k)^T\nabla^2f(x^k)q_1^k\bigr)_{2:r_k}\right\|
 \le \widehat H\alpha,
\end{equation}
where the subscript selects all components except the first, and the
norm is zero when $r_k=1$. This form uses the first HVP alone and avoids
subtracting quantities whose longitudinal components cancel.

The order of the tests is essential for the oracle cost. We first find
a coefficient for which the line candidate satisfies the two Taylor
certificates. Only then do we apply
\eqref{eq:cacu-transverse-test}. If this residual test passes, the
iteration ends after one HVP. If it fails, the line candidate is
discarded, the remaining model columns are formed, and the full
subspace step starts from the same coefficient. Further backtracking
changes only the scalar coefficient and reuses all cached HVPs.

\begin{algorithm}[H]
\caption{CaCuSESOP: cascaded adaptive step within a controlled epoch}
\label{alg:cacusesop}
\begin{algorithmic}[1]
\Require Current $x^k,f(x^k),\nabla f(x^k)$; epoch starting point $y^0$,
aggregate $z^k$, weight $\omega_k$, previous step $s^{k-1}$;
previous coefficient $H_{k-1}>0$
\If{$\nabla f(x^k)=0$}
 \State \Return $x^k$
\EndIf
\State $q_1^k\gets\nabla f(x^k)/\|\nabla f(x^k)\|$
\State $v^k\gets\nabla^2f(x^k)q_1^k$ \Comment{one HVP}
\State $\mathcal S^k\gets
 \Span\{\nabla f(x^k),s^{k-1},x^k-y^0,z^k,v^k\}$
\State Form an orthonormal basis $Q^k$ of $\mathcal S^k$
with first column $q_1^k$
\State $\widehat H\gets H_{k-1}/2$
\Repeat
 \State Compute $\alpha$ from \eqref{eq:cacu-gradient-root}
 and set $p\gets-\alpha q_1^k$
 \State Evaluate $f(x^k+p)$
 \State $a_{\mathrm{line}}\gets\mathrm{false}$
 \If{the function certificate \eqref{eq:ftest} holds}
  \State Evaluate $\nabla f(x^k+p)$
  \State $a_{\mathrm{line}}\gets$ the truth value of
  the gradient certificate \eqref{eq:gtest}
 \EndIf
 \If{not $a_{\mathrm{line}}$}
  \State $\widehat H\gets2\widehat H$
 \EndIf
\Until{$a_{\mathrm{line}}$}
\If{\eqref{eq:cacu-transverse-test} holds}
 \State $s\gets p$
\Else
 \State Form and cache the remaining columns of
 $\nabla^2f(x^k)Q^k$
 \Repeat
  \State $s\gets\argmin_{u\in\mathcal S^k}m_{\widehat H}(u)$
  \Comment{cached reduced model}
  \State Evaluate $f(x^k+s)$
  \State $a_{\mathrm{model}}\gets\mathrm{false}$
  \If{the function certificate \eqref{eq:ftest} holds}
   \State Evaluate $\nabla f(x^k+s)$
   \State $a_{\mathrm{model}}\gets$ the truth value of
   the gradient certificate \eqref{eq:gtest}
  \EndIf
  \If{not $a_{\mathrm{model}}$}
   \State $\widehat H\gets2\widehat H$
  \EndIf
 \Until{$a_{\mathrm{model}}$}
\EndIf
\State $s^k\gets s$, $x^{k+1}\gets x^k+s^k$, $H_k\gets\widehat H$
\State $\mathcal V^{k+1}\gets\Span\{x^{k+1}-y^0,z^k\}$
\State $h^{k+1}\gets\nabla f(x^{k+1})-
               \Pi_{\mathcal V^{k+1}}\nabla f(x^{k+1})$
\State $\omega_{k+1}\gets(1+\sqrt{1+4\omega_k^2})/2$
\State $z^{k+1}\gets z^k+\omega_{k+1}h^{k+1}$
\end{algorithmic}
\end{algorithm}

Use this step in every epoch of
Algorithm~\ref{alg:controlled-restarts}. At an epoch boundary, reset
the starting point, aggregate, and weight as specified there; keep the current
gradient, the last accepted coefficient, and the previous displacement.
At the first point of the whole run, take $s^{-1}=0$ and
$H_{-1}=H_{\mathrm{init}}$, with the initialization conditions of
Algorithm~\ref{alg:adaptive-cubic}.
The two Taylor tests and the coefficient doubling are those already
defined in Section~\ref{sec:adaptive}. The difference is only their
order: the line search is completed before the fallback is considered.

We next verify the two estimates needed by the global and restart
proofs. The gradient candidate has exact stationarity along its own
line, while its residual in the larger space is controlled by the test.

\begin{lemma}[One-step bounds for CaCuSESOP]\label{lem:cacu-budgets}
Under Assumptions~\ref{ass:convex}--\ref{ass:hessian}, let $s^k$ be
an accepted step of Algorithm~\ref{alg:cacusesop} with an exact
subspace fallback. Put $d_k=f(x^k)-f(x^{k+1})$.
Backtracking terminates, and
\begin{align}
 d_k&\ge\tfrac12(s^k)^T\nabla^2f(x^k)s^k
                     +\tfrac{2H_k}{3}\|s^k\|^3,
 \label{eq:cacu-cubic-decrease}\\
 \|\nabla f(x^k)\|^2
 &\le3\bigl(L_0+\sqrt{4HL_0D_0}\bigr)d_k,
 \label{eq:cacu-gradient-budget}\\
 \|e^{k+1}\|
 &\le\left(\tfrac{243}{2}H\right)^{1/3}d_k^{2/3}.
 \label{eq:cacu-defect-budget}
\end{align}
Here $d_k$ is the accepted decrease, and
$e^{k+1}=\Pi_{\mathcal V^{k+1}}\nabla f(x^{k+1})$ is the gradient
component removed before updating the aggregate.
At most $r_k\le5$ HVPs are used in total over all trials of the step.
If the accepted line candidate passes the residual test, the total is one.
\end{lemma}
\begin{proof}
Both candidates minimize the cubic model over a space containing the
gradient: its line or $\mathcal S^k$. Both therefore satisfy
\[
 \nabla f(x^k)^Ts^k+
 (s^k)^T\nabla^2f(x^k)s^k+H_k\|s^k\|^3=0.
\]
The first-exit argument of Lemma~\ref{lem:invariance} applies to the
line candidate at every trial coefficient at least $H$, so the first
backtracking loop terminates. If the residual test fails, the full
model step starts at the coefficient accepted by that loop. The same
first-exit argument applies to this candidate at every coefficient at
least $H$, so the second loop also terminates. Both loops use the same
dyadic grid. As in Proposition~\ref{prop:counts}, the first tested
coefficient on that grid that is at least $H$ is smaller than $2H$;
hence the final accepted coefficient satisfies $H_k\le2H$.
At an accepted step, the displayed identity and the function
certificate give \eqref{eq:cacu-cubic-decrease}.
Both candidates have model value no greater than that at any negative
gradient step. The comparison in
Lemma~\ref{lem:gradient-decrease} therefore gives
\eqref{eq:cacu-gradient-budget}.

For the line candidate, \eqref{eq:cacu-trial-residual} bounds the
projected model residual by $H_k\|s^k\|^2$; for an exact fallback this
residual is zero. Adding the gradient Taylor remainder and the model
regularization gradient gives
\[
 \|\Pi_{\mathcal S^k}\nabla f(x^{k+1})\|
 \le3H_k\|s^k\|^2.
\]
Since $\mathcal V^{k+1}\subseteq\mathcal S^k$, the cubic decrease gives
\[
 \|e^{k+1}\|\le
 3H_k\left(\frac{3d_k}{2H_k}\right)^{2/3}.
\]
Cubing its coefficient yields
$27H_k^3(3/(2H_k))^2=243H_k/4$, proving
\eqref{eq:cacu-defect-budget}.
The first HVP is reused if the remaining $r_k-1$ products are formed;
all products are retained across backtracking trials. Consequently the
line branch uses one HVP and the fallback branch uses at most $r_k\le5$
in total. At the transition to the fallback, the already accepted line
candidate causes at most one extra pair of function and gradient
evaluations, but no extra HVP.
\end{proof}

\subsection{The combined global and local guarantee}

These estimates have exactly the form used in the epoch proof.
Only the defect coefficient is larger. Consequently the same
controller applies to CaCuSESOP, including the finite transition
into the accelerated local regime.

\begin{corollary}[CaCuSESOP with controlled restarts]
\label{cor:cacusesop-rates}
Under Assumptions~\ref{ass:convex}--\ref{ass:hessian}, use
Algorithm~\ref{alg:cacusesop} in the epochs of
Algorithm~\ref{alg:controlled-restarts}. With an exact fallback,
the trajectory is monotone and, for every $K\ge0$,
\[
 f(x^K)-f_\star\le\frac{5184D_0^2}{(K+1)^2}
 \left[\tfrac32\bigl(L_0+\sqrt{4HL_0D_0}\bigr)
                                  +2304HD_0\right].
\]
Under Assumption~\ref{ass:local}, the accelerated local root bound
\eqref{eq:controlled-local-root} holds without supplying the local
strong-convexity parameter. The finite entry and adaptation statement
of Corollary~\ref{cor:explicit-transition} holds with the entry index
\[
 K_{\mathrm{entry}}=\left\lceil72D_0
 \sqrt{\frac{\tfrac32(L_0+\sqrt{4HL_0D_0})
                              +2304HD_0}
                  {\varepsilon_{\mathrm{loc}}}}\right\rceil.
\]
The conclusions apply to any enlarged search space that retains the
three directions used by the proof, including the five-direction space
of Algorithm~\ref{alg:cacusesop}.

The fallback may instead satisfy the finite criterion
\eqref{eq:inexact-model}. Then, for every $K\ge0$,
\[
 f(x^K)-f_\star\le\frac{5184D_0^2}{(K+1)^2}
 \left[\tfrac72\left(L_0+\sqrt{\tfrac{12}{5}HL_0D_0}\right)
                                  +2304HD_0\right].
\]
The local root bound and finite-transition conclusion are unchanged,
except that the numerator inside the square root defining
$K_{\mathrm{entry}}$ becomes
$\tfrac72(L_0+\sqrt{(12/5)HL_0D_0})+2304HD_0$.
This criterion is required only for the fallback; the explicit
gradient candidate uses \eqref{eq:cacu-trial-residual}.
\end{corollary}
\begin{proof}
The epoch potential uses only the gradient and correction-error bounds,
the weight recurrence, and correction-space containment.
With the constants of Lemma~\ref{lem:cacu-budgets}, its completed-epoch
error term is bounded by
\[
 \frac{512}{27}\,\frac{243}{2}HD^3
 =2304HD^3,
\]
where $D$ is the radius of the current sublevel about $x^\star$.
The deterministic-epoch induction and the global stage of the
controller thus give the first displayed rate, since $64\cdot81=5184$.

On a current sublevel, let $L$ bound the Hessian norm and $D$
bound the distance to $x^\star$. If its Hessian also satisfies
$\nabla^2 f\succeq\mu I$, write
$C_{\mathrm{grad}}=3(L+\sqrt{4HLD})$ for the gradient-decrease
constant. The epoch curvature term $\Gamma$ and first-step gap fraction
$\theta$ from the controller proof become
\[
 \Gamma=\frac{C_{\mathrm{grad}}}{\mu}+\frac{4608HD}{\mu},
 \qquad \theta=\frac{2\mu}{C_{\mathrm{grad}}}.
\]
In the transition region of
Corollary~\ref{cor:explicit-transition},
$D\le \mu_\star/(4096H)$,
$\mu=(1023/1024)\mu_\star$, and
$L\le(1025/1024)L_\star$.
Hence $C_{\mathrm{grad}}<8L_\star$ and
\[
 \Gamma<\frac{1024}{1023}
             \left(8\kappa_\star+\frac98\right)
       <10\kappa_\star<19\kappa_\star,
 \qquad \theta>\frac{1}{9\kappa_\star}.
\]
The bounds on completed epochs and adaptation therefore retain the
local exponent and finite-transition constants. The entry estimate
uses the new global constant.

For finite fallbacks, Appendix~\ref{app:finite-solves} gives
$C_{\mathrm{grad}}=7(L+\sqrt{(12/5)HLD})$ and a defect coefficient whose
cube is $(2197/27)H$. These also cover the gradient bound of the
line branch, and $2197/54<243/4$. Use the larger defect coefficient
of the line branch for both. The same epoch term $2304HD^3$ and
the second displayed rate follow. The transition-region bound
$C_{\mathrm{grad}}<8L_\star$ was proved in that appendix, so the same local
argument applies. No extra objective-oracle call is needed to enforce
the finite fallback criterion.
\end{proof}

\subsection{What the two additional directions add}

The convergence proofs use the three-direction space
\[
 \mathcal S_3^k=\Span\{\nabla f(x^k),x^k-y^0,z^k\},
\]
where $y^0$ is the starting point of the current epoch.
Original SESOP also uses previous-step directions
\cite{narkiss2005,zibulevsky2014}. We may retain the preceding step
$s^{k-1}$ by using
\[
 \mathcal S_4^k=\mathcal S_3^k+\Span\{s^{k-1}\}.
\]
Adding this direction preserves the convergence guarantees. At the
same point and cubic coefficient, it can only decrease the minimum
model value:
\[
 \min_{u\in\mathcal S_4^k}m_{\widehat H}(u)
 \le\min_{u\in\mathcal S_3^k}m_{\widehat H}(u).
\]
If the preceding step already belongs to $\mathcal S_3^k$, the spaces
coincide. Any other direction, including a quasi-Newton direction,
may also be added while retaining the three required directions,
the gradient correction, and the model-solve accuracy criterion.
Each added independent direction requires one HVP to form its column
of the reduced Hessian.

CaCuSESOP uses a fifth direction,
\[
 \mathcal S_5^k
 =\mathcal S_4^k+
   \Span\{\nabla^2f(x^k)\nabla f(x^k)\}.
\]
It is available from the first HVP, up to scaling, because
$\nabla^2f(x^k)\nabla f(x^k)=
\|\nabla f(x^k)\|v^k$. Thus the cheap branch still needs only one
HVP. If the fallback is activated, the cached product supplies the
first reduced-Hessian column and at most four further HVPs supply the
remaining columns. The total is therefore at most five.

Both additional directions affect the test for accepting the cheap
gradient step. The candidate $p$ is unchanged, but its model gradient
is tested over the whole chosen subspace. Since
$\mathcal S_3^k\subseteq\mathcal S_4^k\subseteq\mathcal S_5^k$,
\[
 \|\Pi_{\mathcal S_3^k}\nabla m_{\widehat H}(p)\|
 \le\|\Pi_{\mathcal S_4^k}\nabla m_{\widehat H}(p)\|
 \le\|\Pi_{\mathcal S_5^k}\nabla m_{\widehat H}(p)\|.
\]
Thus a candidate may pass the three-direction test and fail the
five-direction test. In that case, Algorithm~\ref{alg:cacusesop}
solves the full subspace model, which can use both recent-step memory
and the Hessian--gradient direction. The test itself still uses only
the first cached HVP because $p$ is a scaled gradient.

\subsection{Why three directions can give slow, oscillating progress}

Two of the three directions can become almost fixed and give little
descent. The current gradient is then the only effective changing
direction, so the method can behave almost like gradient descent
with a line search.

This can happen late in a long epoch. The vectors $x^k-y^0$ and
$z^k$ accumulate steps and gradient corrections from its start.
Their changes at one iteration are
\[
 (x^{k+1}-y^0)-(x^k-y^0)=s^k,\qquad
 z^{k+1}-z^k=\omega_{k+1}h^{k+1}.
\]
Small steps and weighted gradient corrections, relative to these
accumulated vectors, leave both directions nearly unchanged.
Their linear independence does not ensure that they still provide
useful descent.

Gradient descent can alternate between two directions and converge
very slowly on an ill-conditioned problem. For a simple example, take
\[
 q(u)=\tfrac12(\mu u_1^2+Lu_2^2),\qquad 0<\mu<L,
\]
and suppose $\nabla q(u^k)=a(1,1)^T$, with $a\ne0$.
Exact line search gives $s^k=-2\nabla q(u^k)/(L+\mu)$ and
\[
 \begin{aligned}
 \nabla q(u^{k+1})
   &=\frac{L-\mu}{L+\mu}\,a(1,-1)^T,\\
 \nabla q(u^{k+2})
   &=\left(\frac{L-\mu}{L+\mu}\right)^2a(1,1)^T.
 \end{aligned}
\]
The direction repeats after two steps, while
\[
 q(u^{k+1})=\left(\frac{L-\mu}{L+\mu}\right)^2q(u^k).
\]
For large $L/\mu$, this factor is close to one. The iterates zigzag
and the objective decreases very little. This illustrates a possible
mechanism for slow progress; it is not a proof of an exact cycle in
cubic SESOP.

Keeping the previous step supplies a second recent direction,
separately from the two accumulated vectors. In this example,
$\Span\{\nabla q(u^{k+1}),s^k\}=\R^2$, so minimizing the quadratic
over that space reaches its minimizer in the next solve. For cubic
SESOP, the model can likewise combine the new gradient with the
previous step instead of relying on two nearly fixed directions.
The cubic regularization remains present, so the quadratic example
provides intuition rather than a finite-step guarantee.

Restarts help the three-direction version by discarding the old
accumulation. Set the new epoch's starting point to $y^0=x^k$ and
reset the aggregate and weights. After its first step,
\[
 x^{k+1}-y^0=s^k.
\]
One of the three directions is now the recent step itself.
Controlled restarts repeat this reset and retain the proved global
and accelerated local guarantees. CaCuSESOP also keeps the previous
step explicitly and uses the first HVP as a Hessian--gradient direction.

\section{Regularized logistic regression on W8A}
\label{sec:experiments}

\paragraph{Problem.}
We consider logistic regression on the W8A data set with
$\ell_2$ regularization:
\[
 f(x)=\frac1n\sum_{i=1}^n
       \log\bigl(1+\exp(-b_i a_i^Tx)\bigr)
       +\frac{\lambda}{2}\|x\|^2,
 \qquad \lambda=10^{-7}.
\]
All methods use the same data, starting point, and fixed reference
value $f_{\rm ref}$. Figure~\ref{fig:w8a} reports
$f(x^k)-f_{\rm ref}$ against accepted outer iterations and elapsed
time. The time measurement includes inner optimization, rejected
backtracking trials, and restart decisions. Each timing run starts
from a fresh optimizer object.

\paragraph{Methods.}
We compare CaCuSESOP (Algorithm~\ref{alg:cacusesop}), classical SESOP
\cite{narkiss2005}, and Scieur's nonaccelerated Generic Type-I method
\cite{scieur2024}. CaCuSESOP uses the five-direction space, cascaded
backtracking, and controlled restarts described in
Section~\ref{sec:cacusesop}.

For SESOP, we use the current gradient, one previous step, and the two
Nemirovski directions. The restricted nonlinear objective is minimized
by Newton's method with Armijo backtracking, as in the original article.
We take the inner settings from the authors'
\href{https://sites.google.com/site/michaelzibulevsky/home/matlab-code/sesop}{SESOP\_PACK}:
at most seven Newton iterations, a projected-gradient threshold of
$10^{-10}$, initial line-search step $1$, reduction factor $1/4$, and
at most $40$ line-search trials. Thus the implementation follows the
authors' own practical choice rather than tuning the inner solver for
this benchmark. The Newton iteration limit is nevertheless a genuine
implementation hyperparameter: the theory assumes an exact restricted
minimizer and does not prescribe a finite value. On this problem,
limits between five and ten Newton iterations give similar behavior;
beyond roughly fifteen iterations, the additional inner work noticeably
increases elapsed time without a comparable improvement in outer progress.

For Generic Type-I, we use the nonaccelerated variant recommended in
Scieur's experiments, with memory $m=25$, orthogonal forward estimates,
and forward-difference step $h=10^{-9}$. We also use the initial
regularization rule of \cite[Appendix~F.1]{scieur2024}. The linked
\href{https://github.com/windows7lover/QN_With_Guarantees}{repository}
contained no implementation when checked on October~1, 2026, so this
baseline is reimplemented from the published algorithms. The memory
$m$ is another practical hyperparameter. Increasing it retains more
curvature information, but increases storage and the cost of the dense
reduced problem. With fixed $m$, Scieur's theoretical bound retains
secant and projection-error terms; the full/exact rate requires the
substantially more expensive full-space, exact-curvature regime.
Consequently, a memory value that works well on one problem may not
capture enough useful curvature on a larger or more difficult one.

\begin{figure}[htbp]
\centering
\begin{minipage}[t]{0.48\textwidth}
 \centering
 \IfFileExists{figures/logistic_w8a_cacuadgd_sesop_i_l2_1e-07.pdf}{
  \includegraphics[width=\linewidth]{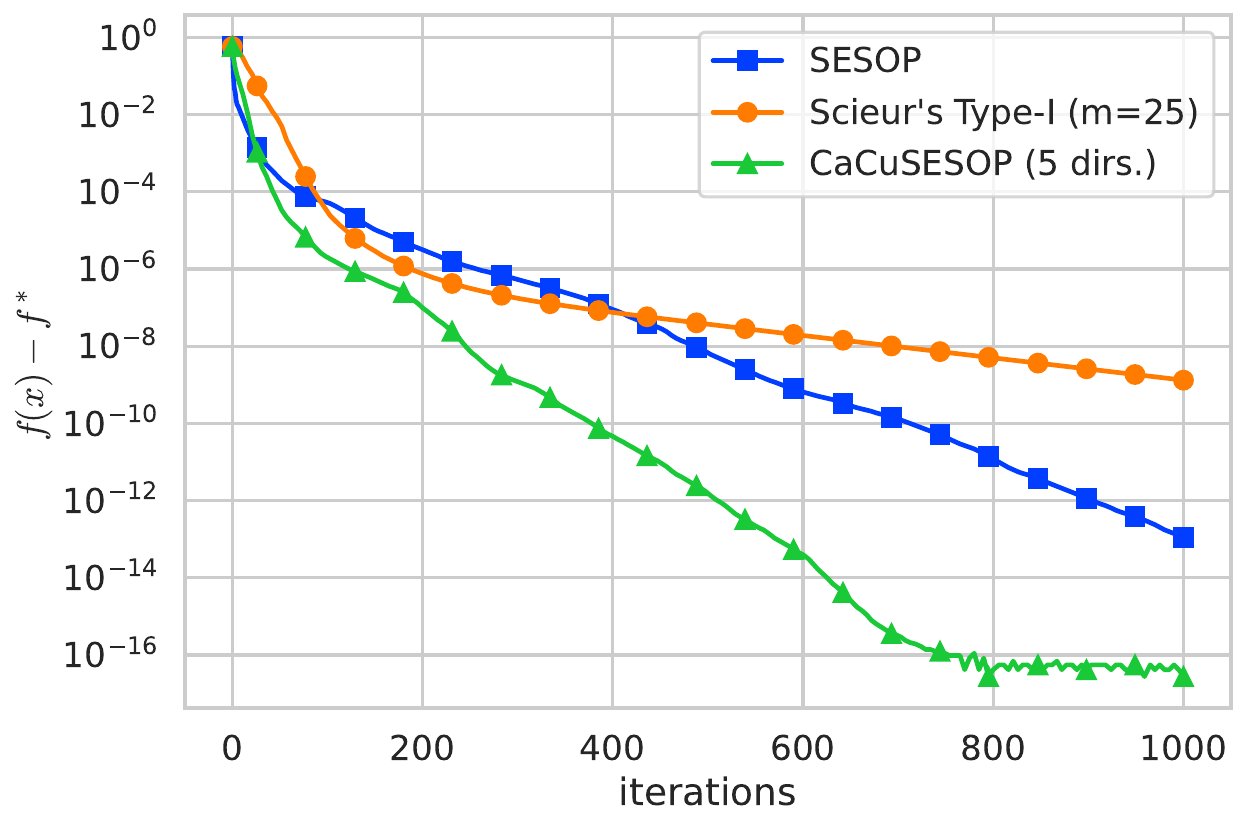}
 }{
  \fbox{\parbox[c][0.19\textheight][c]{0.92\linewidth}{
   \centering Plot to be inserted\\[4pt]
   W8A logistic regression\\[4pt]
   Objective error versus iterations}}
 }
 \par\smallskip\small (a) Accepted outer iterations.
\end{minipage}
\hfill
\begin{minipage}[t]{0.48\textwidth}
 \centering
 \IfFileExists{figures/logistic_w8a_cacuadgd_sesop_t_l2_1e-07.pdf}{
  \includegraphics[width=\linewidth]{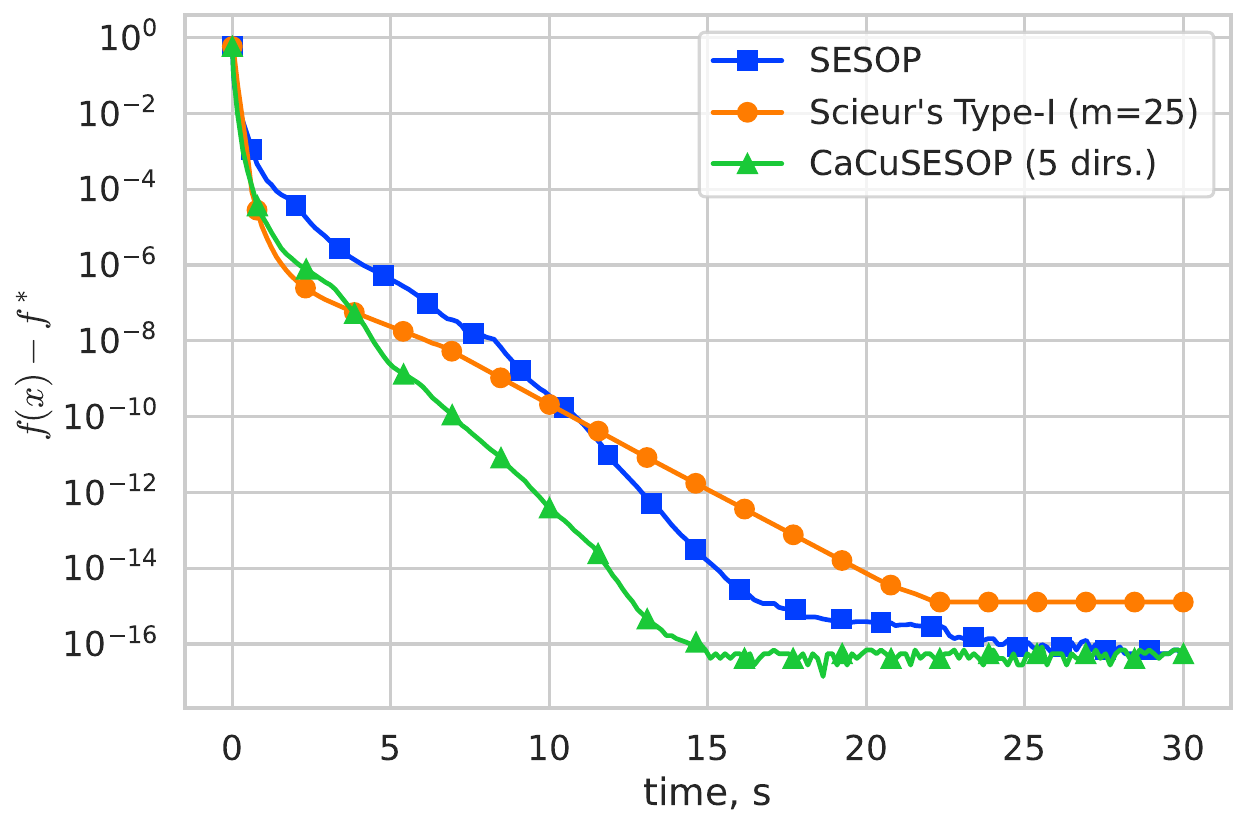}
 }{
  \fbox{\parbox[c][0.19\textheight][c]{0.92\linewidth}{
   \centering Plot to be inserted\\[4pt]
   W8A logistic regression\\[4pt]
   Objective error versus time}}
 }
 \par\smallskip\small (b) Elapsed time in seconds.
\end{minipage}
\caption{Regularized logistic regression on W8A with
$\lambda=10^{-7}$. Both panels use the same objective, starting point,
and reference value.}
\label{fig:w8a}
\end{figure}

\paragraph{Results.}
CaCuSESOP gives the fastest decrease in both accepted iterations and
elapsed time on this benchmark. Its one-HVP gradient branch avoids a
reduced solve whenever the residual certificate passes; otherwise the
five-dimensional cubic fallback is formed once and reused throughout
backtracking. Controlled restarts prevent the accumulated SESOP
directions from becoming ineffective.

Classical SESOP also makes substantial progress per outer iteration,
but every outer step invokes the Newton--Armijo solver for the original
restricted objective. This makes its elapsed-time curve less favorable,
even with the authors' seven-iteration default. Generic Type-I avoids
that nonlinear restricted solve, but its performance depends on the
fixed memory $m$. We use the author's recommended $m=25$ configuration
rather than selecting $m$ from the W8A results. The experiment therefore
compares CaCuSESOP against the practical configurations advocated by the
respective baselines, without benchmark-specific tuning of their main
inner-work parameters.

\section{Conclusion}

CaCuSESOP brings the ingredients of the analysis into one practical
method: an explicit certified gradient trial, a cubic subspace fallback,
previous-step memory, a Hessian--gradient direction, and adaptive
restart control. Three directions suffice for its convergence orders.
The practical five-direction space gives the fallback two additional
directions and imposes a stronger gradient-trial test.

The cubic model turns directional curvature into a source of
decrease. In a SESOP space, the remaining obstacle is the loss of
exact orthogonality to the new gradient. The same cubic decrease
bounds this defect and permits a corrected gradient aggregate.
This gives a global rate of order $\sqrt{\log K}/K^2$ with at most
three HVPs per iteration, followed by automatic local linear
convergence at a nondegenerate minimizer. Deterministic restarts
improve the global rate to $O(K^{-2})$ with the same oracle cost.
An adaptation of the classical sharp CG/Krylov bound shows that this
global exponent is sharp in the stated gradient/HVP linear-span
oracle class, including momentum.
Backtracking selects the regularization coefficient without a
supplied smoothness bound and preserves these guarantees with an
explicit count of all trial evaluations. Finally, a restart controller
retains the inverse-square global rate and gives an accelerated local
dependence on the Hessian condition number, without knowing it.
These extensions share the same explicit reduced model; the proofs
separate accepted steps, trial evaluations, and the finite transition
to the local regime. A finite relative-residual criterion preserves
these guarantees without further objective-oracle calls. All stated
three-direction variants use at most three HVPs per accepted step;
CaCuSESOP uses either one HVP for its accepted gradient trial or at most
five HVPs for its subspace fallback.
A local PL condition also gives linear convergence and finite
trajectory length without a nondegenerate or unique minimum.
For adaptive coefficients, this result assumes that the accepted
coefficients are eventually bounded away from zero; the algorithm
does not need to know that bound.

\clearpage
\appendix
\section{Proof of the vector-oracle lower bound}
\label{app:oracle-lower-bound}

We repeat the complete statement so that the oracle restrictions and
the constants are available together with the proof.

\begin{lowerboundrestated}
\OracleLowerBoundStatement
\end{lowerboundrestated}

The proof has four steps. We first choose a quadratic function whose
minimizer is difficult to approximate using a low-degree polynomial
of its Hessian. We then quantify that approximation error. Next we
add one coordinate to obtain a nonquadratic function with fixed
smoothness constants. Finally, we show that every permitted oracle
call reveals at most one further polynomial degree. This last step
also covers momentum and nonmonotone extrapolation.
The quadratic part is a normalized form of Meinardus' extremal CG
example~\cite{meinardus1963,li2008meinardus}. The sharp spectral
polynomial framework is classical~\cite{greenbaum1979}.
For the present oracle statement, we choose the smallest eigenvalue
as a function of the total call budget, normalize the level-set
constants, and append a smooth coordinate with a fixed nonzero
global Hessian Lipschitz constant. The final span argument explicitly
counts gradients and HVPs under the stated query restrictions.
We give all four steps to make this adaptation self-contained.

\begin{proof}[Proof of Theorem~\ref{thm:oracle-lower-bound}]
Fix an integer $M\ge1$. We construct one function that is hard for
every permitted algorithm at this budget.

\paragraph{Step 1: choose the spectrum and the quadratic minimizer.}
Put $\mu=(M+1)^{-2}$. Choose $M+1$ distinct
positive numbers
\[
 \lambda_i=\frac{1+\mu}{2}
       +\frac{1-\mu}{2}\cos\frac{i\pi}{M},
 \qquad i=0,\ldots,M.
\]
They decrease from $\lambda_0=1$ to $\lambda_M=\mu$. Define
their Lagrange interpolation polynomials and positive weights by
\[
 \ell_i(t)=\prod_{\substack{0\le j\le M\\j\ne i}}
             \frac{t-\lambda_j}{\lambda_i-\lambda_j},
 \qquad W_M=\sum_{i=0}^M|\ell_i(0)|,
 \qquad w_i=\frac{|\ell_i(0)|}{W_M}.
\]
No node is zero, so every $w_i$ is positive and $\sum_{i=0}^M w_i=1$.
Set
\[
 A=\operatorname{diag}(\lambda_0,\ldots,\lambda_M),
 \qquad (u^\star)_i=\sqrt{\frac{\mu w_i}{\lambda_i}},
 \qquad q(u)=\frac12(u-u^\star)^TA(u-u^\star).
\]
Here $(u^\star)_i$ denotes a coordinate of the minimizer, not an
iterate. The smallest and largest eigenvalues of $A$ are $\mu$
and $1$, respectively, and
\[
 q(u^\star)=0,\qquad q(0)=\frac\mu2.
\]

The connection to the classical example can be seen directly in the
weights. At the chosen Chebyshev extrema, $|\ell_i(0)|$ is
proportional to $1/\lambda_i$ at the two endpoints and to
$2/\lambda_i$ at the interior nodes. Hence $Au^\star$ is a scalar
multiple of the right-hand side in Meinardus' example as described
by Li~\cite{li2008meinardus}. The present normalization will give
level-set radius one.

\paragraph{Step 2: bound the best point in the available polynomial space.}
Define the nested spaces
\[
 \mathcal K_0=\{0\},\qquad
 \mathcal K_j=\Span\{Au^\star,A^2u^\star,\ldots,A^ju^\star\}
 \quad(j\ge1).
\]
For every $u\in\mathcal K_M$, there is a polynomial $p$ of
degree at most $M$ such that
\begin{equation}\label{eq:lower-polynomial-error}
 u-u^\star=-p(A)u^\star,\qquad p(0)=1.
\end{equation}
Indeed, writing $u=\sum_{j=1}^M a_jA^ju^\star$ gives
$p(t)=1-\sum_{j=1}^M a_jt^j$.
Interpolation of $p$ at zero and weighted Cauchy--Schwarz yield
\[
 \begin{aligned}
 1&=\left(\sum_{i=0}^M\ell_i(0)p(\lambda_i)\right)^2\\
  &\le\left(\sum_{i=0}^M\frac{\ell_i(0)^2}{w_i}\right)
        \left(\sum_{i=0}^M w_ip(\lambda_i)^2\right)
   =W_M^2\sum_{i=0}^M w_ip(\lambda_i)^2.
 \end{aligned}
\]
Consequently, \eqref{eq:lower-polynomial-error} and the choice
of $u^\star$ imply
\[
 q(u)=\frac\mu2\sum_{i=0}^M w_ip(\lambda_i)^2
      \ge\frac{\mu}{2W_M^2}.
\]

It remains to bound $W_M$ by a constant independent of $M$.
Let $\mathsf T_M$ be the Chebyshev polynomial defined by
$\mathsf T_M(\cos\theta)=\cos(M\theta)$. At the chosen nodes,
\[
 \mathsf T_M\left(\frac{2\lambda_i-1-\mu}{1-\mu}\right)=(-1)^i.
\]
The decreasing order of the nodes gives
$\operatorname{sign}\ell_i(0)=(-1)^{M+i}$: the numerator has
$M$ negative factors and the denominator has $i$ negative factors.
Interpolating the displayed degree-$M$ polynomial at zero and using
its parity therefore gives
\[
 \begin{aligned}
 W_M
 &=\mathsf T_M\left(\frac{1+\mu}{1-\mu}\right)\\
 &=\cosh\left(M\log\frac{1+\sqrt\mu}{1-\sqrt\mu}\right)
  =\cosh\left(M\log(1+2/M)\right)
 \le\cosh2.
 \end{aligned}
\]
We used $\log(1+a)\le a$ in the last step. Thus every point in
$\mathcal K_M$ satisfies
\begin{equation}\label{eq:lower-quadratic-bound}
 q(u)\ge\frac{1}{2\cosh^2(2)(M+1)^2}.
\end{equation}
This bound concerns the entire space, regardless of how a point
inside it is selected.

\paragraph{Step 3: make the objective nonquadratic with fixed constants.}
A quadratic has constant Hessian. To obtain a genuinely nonquadratic
example whose smallest global Hessian Lipschitz constant is exactly
two, define
\[
 \eta(t)=
 \begin{cases}
  \exp\!\left(4-\dfrac1{t(1-t)}\right),&0<t<1,\\
  0,&\text{otherwise},
 \end{cases}
 \qquad
 \phi(t)=\int_0^t(t-v)^2\eta(v)\,dv.
\]
The function $\eta$ is smooth, with every derivative vanishing at
$0$ and $1$. It satisfies $0\le\eta\le1$ and $\eta(1/2)=1$.
Differentiating the integral gives
\[
 \phi''(t)=2\int_0^t\eta(v)\,dv,\qquad
 \phi'''(t)=2\eta(t),\qquad
 0\le\phi''(t)<2.
\]
In particular, $\phi$ is smooth, convex, and nonnegative, and it and
all its derivatives vanish on $(-\infty,0]$. For $0<t<1$, set
$a=t(1-t)$. Then
\[
 |\eta'(t)|\le e^{4-1/a}a^{-2}\le16.
\]
For the second inequality, put $y=1/a\ge4$ and use that
$e^{4-y}y^2$ decreases on $[4,\infty)$ and equals $16$ at $4$.
Outside $[0,1]$ the derivative vanishes. Hence
$\max|\phi'''|=2$ and $|\phi''''|\le32$.

Append one scalar coordinate $t$ to $u\in\R^{M+1}$ and define
\[
 f(u,t)=q(u)+\frac12t^2+\phi(t-2).
\]
This function is convex, smooth, nonquadratic, and coercive. Its
unique minimizer is $x^\star=(u^\star,0)$, with $f_\star=0$.
Its Hessian is the block diagonal matrix
\[
 \nabla^2f(u,t)=
 \begin{pmatrix}
 A&0\\0&1+\phi''(t-2)
 \end{pmatrix},
\]
so $0\preceq\nabla^2f\preceq3I$. Its only varying Hessian entry
has derivative $\phi'''(t-2)$. Thus the global Hessian Lipschitz
constant is at most two, and it cannot be smaller because
$\phi'''(1/2)=2$. The only nonzero component of the third derivative
is $\phi'''(t-2)$ in the added coordinate. Its derivative is bounded
by $32$, proving the stated Lipschitz bound in the induced multilinear
norm. All these constants are independent of $M$.

Since $f(0,0)=\mu/2$ and $\phi\ge0$, a point in the initial
sublevel must satisfy $|t|\le\sqrt\mu\le1/2$. At every such point
$\phi(t-2)=0$. It follows in both directions that
\[
 \mathcal C_0
 =\{(u,t):(u-u^\star)^TA(u-u^\star)+t^2\le\mu\}.
\]
The Hessian is constant on this ellipsoid, with largest eigenvalue
one, so $L_0=1$. The smallest eigenvalue of its defining matrix is
$\mu$, which gives radius one about $(u^\star,0)$. That radius is
attained along an eigenvector of $A$ for $\mu$, so $D_0=1$.

\paragraph{Step 4: count information, including momentum and inner calls.}
We now prove by induction that, after $j$ vector-oracle calls, every
returned vector lies in $\mathcal K_j\times\{0\}$. Before the first
call the available space is $\{0\}$. Suppose the claim holds after
$j$ calls. Every permitted query point then has the form $(u,0)$
with $u\in\mathcal K_j$, and every permitted HVP input has the
form $(v,0)$ with $v\in\mathcal K_j$. On this entire subspace,
\[
 \nabla f(u,0)=(Au-Au^\star,0),\qquad
 \nabla^2f(u,0)(v,0)=(Av,0).
\]
Each returned vector belongs to $\mathcal K_{j+1}\times\{0\}$.
The earlier vectors remain there as well, since the spaces are
nested. Arbitrary scalar-dependent combinations, exact minimization
inside the available span, and function-value queries cannot enlarge
that span. This proves the induction.

In particular, after at most $M$ vector-oracle calls every output
has the form $(u,0)$ with $u\in\mathcal K_M$. Since
$f(u,0)=q(u)$, inequality~\eqref{eq:lower-quadratic-bound} proves
the claimed lower bound for every such output. The construction
depends only on $M$, not on the algorithm.
The added coordinate remains zero even when query points leave
$\mathcal C_0$, so neither monotonicity nor containment of the queries
in the initial sublevel was used.

To spell out the role of momentum, an extrapolate
$y^k=x^k+\beta_k(x^k-x^{k-1})$ remains in the span containing
the two iterates, for every adaptively chosen scalar $\beta_k$.
Evaluating a gradient or an HVP there is already one of the calls
covered by the induction. Retaining all earlier vectors or solving
the restricted model more accurately does not change this argument.

Finally, suppose an algorithm uses at most $b+ck$ vector-oracle
calls through outer iteration $k$, with fixed nonnegative constants
$b,c$. Apply the construction with
$M=\max\{1,\lceil b+ck\rceil\}$. The resulting worst-case gap is
at least a constant multiple of $(k+1)^{-2}$, with the constant
depending only on $b,c$.
This contradicts a uniform bound $C/(k+1)^3$ whose constant and
starting index are independent of dimension and depend only on the
fixed displayed problem bounds. The same comparison rules out any
uniform $o(k^{-2})$ bound in this oracle model.
It concerns a family of problems of increasing dimension, so it
does not preclude faster local convergence on a fixed problem.
\end{proof}

\section{Backtracking guarantees and trial counts}
\label{app:backtracking}

This appendix supplies the analysis deferred from
Section~\ref{sec:adaptive}. We first prove that the two-certificate
rule terminates and preserves the one-step estimates. We then
transfer the global and local guarantees and count the extra trials.
The coefficient is retained across restarts throughout.
We use the sufficiently small initial estimate specified in
Theorem~\ref{thm:global-bt} throughout this appendix.

\subsection{Finite termination and a uniform coefficient bound}

We first show that the trial loop always terminates. This also gives
a uniform upper bound on the accepted coefficients, which will replace
the fixed coefficient in the convergence estimates.

\begin{theorem}[Finite backtracking and bounded coefficients]
\label{thm:global-bt}
Under Assumptions~\ref{ass:convex}--\ref{ass:hessian}, take
$0<H_{\mathrm{init}}\le2H$. The adaptive rule is well defined for
any search spaces containing the current gradient.
Every trial loop is finite, all accepted iterates belong to
$\mathcal C_0$, and
\begin{equation}\label{eq:global-Mbound}
 0<H_k\le H_{\max}:=2H.
\end{equation}
\end{theorem}
\begin{proof}
Each trial coefficient is positive, so its cubic model has a unique
minimizer. Any trial with $\widehat H\ge H$ passes both certificates:
Lemma~\ref{lem:invariance} first places its segment in $\mathcal C_0$,
and the two Taylor estimates then apply there. Doubling therefore
terminates. At an accepted step, model minimality and the function
certificate give $f(x^{k+1})\le f(x^k)$.

If the first trial passes, $H_k=H_{k-1}/2$.
Otherwise the preceding rejected coefficient was less than $H$,
so $H_k<2H$. Induction proves \eqref{eq:global-Mbound}.
Rejected trials below $H$ may leave $\mathcal C_0$; no estimate
outside that set is used.
\end{proof}

The initial coefficient is a starting estimate, not a supplied
Hessian Lipschitz bound. As usual in backtracking analyses, we take
it sufficiently small; $H$ is used only in the proof.

All trials at the same iterate reuse the same three HVPs.
Proposition~\ref{prop:counts} below counts the extra value and
gradient evaluations.

\subsection{Global and local guarantees for accepted steps}

We now recover the two decrease estimates used in the global proof.
The acceptance tests give them with the actual coefficient $H_k$;
its uniform upper bound then supplies constants valid for every step.
Substituting $H_{\max}=2H$ in the bounds below gives the estimates
stated in the main text directly in terms of $H$. These same numerical
constants also cover fixed-coefficient steps, since $H\le2H$.
The convergence orders and local accelerated exponent are unchanged;
no additional parameter is supplied to any algorithm.

\begin{theorem}[Global rate with adaptive regularization]
\label{thm:ag-rate}
Under Assumptions~\ref{ass:convex}--\ref{ass:hessian},
Algorithm~\ref{alg:adaptive-cubic} without restarts satisfies
\[
 f(x^K)-f_\star\le C(L_0D_0^2+H_{\max}D_0^3)
       \frac{\sqrt{1+\log(K+1)}}{(K+1)^2}\qquad(K\ge0)
\]
for a universal constant $C$, with $H_{\max}$ from
\eqref{eq:global-Mbound}. No positive lower bound on $H_k$ is required.
\end{theorem}
\begin{proof}
The stationary initial point is immediate. Otherwise define
\begin{equation}\label{eq:adaptive-decrease-constant}
 \mathcal L=L_0+\sqrt{2H_{\max}L_0D_0}>0.
\end{equation}
Stationarity of the accepted model, followed by its function certificate,
gives
\begin{equation}\label{eq:adaptive-two-decreases}
 f(x^k)-f(x^{k+1})
 \ge\tfrac12(s^k)^T\nabla^2f(x^k)s^k+\tfrac{2H_k}{3}\|s^k\|^3.
\end{equation}
Since $H_k\le H_{\max}$ and $\|\nabla f(x^k)\|\le L_0D_0$,
the gradient comparison in Lemma~\ref{lem:gradient-decrease}
applies to the accepted model with the present $\mathcal L$.
It gives
\[
 \|\nabla f(x^k)\|^2
 \le3\mathcal L\bigl(f(x^k)-f(x^{k+1})\bigr).
\]

The projected stationarity equation and \eqref{eq:gtest} also give
$\|\Pi_{\mathcal S^k}\nabla f(x^{k+1})\|\le2H_k\|s^k\|^2$.
Since the correction space is contained in $\mathcal S^k$,
\begin{align*}
 \|e^{k+1}\|
 &\le2H_k\|s^k\|^2\\
 &\le(18H_k)^{1/3}\bigl(f(x^k)-f(x^{k+1})\bigr)^{2/3}\\
 &\le(18H_{\max})^{1/3}
                   \bigl(f(x^k)-f(x^{k+1})\bigr)^{2/3}.
\end{align*}
The second inequality uses \eqref{eq:adaptive-two-decreases} with
the actual $H_k>0$ and requires no uniform positive lower bound.

The gradient and correction estimates are precisely
\eqref{eq:algorithm-gradient}--\eqref{eq:defect-budget} with $H$ replaced by $H_{\max}$ and $\mathcal L$
set to \eqref{eq:adaptive-decrease-constant}.
The weights and all aggregate orthogonalities are unchanged.
Consequently, the interval sums, prefix estimates, and weight-window
argument of Section~\ref{sec:global} apply with these replacements.
Using
$\mathcal L D_0^2\le\tfrac32L_0D_0^2+H_{\max}D_0^3$
in \eqref{eq:explicit-rate} proves the stated bound. Constant
extension covers finite termination.
\end{proof}

The deterministic epochs use only the same one-step estimates.
They therefore remove the logarithmic factor also when the coefficient
is chosen by backtracking.

\begin{corollary}[Adaptive deterministic epochs]
\label{cor:adaptive-restarts}
Use the epoch schedule of Section~\ref{sec:restarted-rate} with
the adaptive steps above, retaining the accepted coefficient across
epoch boundaries. Under Assumptions~\ref{ass:convex}--\ref{ass:hessian},
\[
 f(x^K)-f_\star\le
 \frac{81D_0^2}{(K+1)^2}
 \left(\frac32\bigl(L_0+\sqrt{2H_{\max}L_0D_0}\bigr)
                   +\frac{1024}{3}H_{\max}D_0\right).
\]
The trial and oracle bounds of Proposition~\ref{prop:counts} hold,
including at most three HVPs per accepted step.
The method uses no supplied smoothness or strong-convexity constant.
\end{corollary}
\begin{proof}
The accepted-step gradient and correction bounds in
Theorem~\ref{thm:ag-rate} are the fixed-coefficient bounds with
$H$ replaced by $H_{\max}$ and
$\mathcal L=L_0+\sqrt{2H_{\max}L_0D_0}$.
Substitute them into the epoch estimate and induction in
Theorem~\ref{thm:restarted-rate}. This gives the stated numerator
and the same three-HVP bound. Retaining $H_k$ across boundaries
preserves the coefficient bound and the count in
Proposition~\ref{prop:counts}.
\end{proof}

Finally, adaptation preserves automatic local linear convergence.
The proof needs bounded regularization coefficients, not a fixed one;
this observation will also be used for the restart controller below.

\begin{proposition}[Automatic local convergence with adaptation]
\label{prop:adaptive-local}
Under Assumptions~\ref{ass:convex}--\ref{ass:hessian} and
\ref{ass:local}, the adaptive method converges to $x^\star$ and
satisfies all conclusions of Theorem~\ref{thm:local}. This holds for
any search spaces containing $\nabla f(x^k)$ and for any changes to the epoch's
starting point or aggregate between accepted steps. The norm and local condition number are those in
Section~\ref{sec:local}; no local curvature parameter is supplied.
\end{proposition}
\begin{proof}
The gradient decrease bound gives
$\sum_{k=0}^\infty\|\nabla f(x^k)\|^2<\infty$.
All accumulation points therefore minimize $f$. The minimizer is
unique by Lemma~\ref{lem:local-entry}; compactness gives
$x^k\to x^\star$, regardless of changes to the epoch's starting point.

For large $k$, the local Hessian lower bound and model stationarity
give $\|s^k\|\le2\|\nabla f(x^k)\|/\mu_\star$.
In Lemma~\ref{lem:local-projection}, write $u^k=x^k-x^\star$
and replace its remainder by
\[
 r^k=\nabla^2f(x^\star)u^k-\nabla f(x^k)
       +[\nabla^2f(x^\star)-\nabla^2f(x^k)]s^k-H_k\|s^k\|s^k.
\]
Because $H_k\le H_{\max}$, its norm is still
$O(\|x^k-x^\star\|_\star^2)$. The same projection decomposition
and gradient comparison in Theorem~\ref{thm:local} consequently
prove its error recurrence. The objective-gap factor follows from
the same Taylor expansion at $x^\star$.
\end{proof}

The local PL result also extends to backtracking when the accepted
coefficients stay bounded away from zero. The lower bound is a
property of the accepted sequence, not an input to the algorithm.

\begin{corollary}[Local PL convergence and finite trajectory length with backtracking]
\label{cor:adaptive-pl}
Under Assumptions~\ref{ass:convex}--\ref{ass:hessian}, use the adaptive
algorithm. Suppose there exists $H_{\min}>0$ such that
$H_k\ge H_{\min}$ for all sufficiently large $k$.
If an accumulation point $\bar x$ satisfies the local PL condition
\eqref{eq:local-pl-condition}, all conclusions of
Theorem~\ref{thm:local-pl} hold with
\[
 q=1-\frac{2\mu}{3(L_0+\sqrt{2H_{\max}L_0D_0})}
\]
and with $(2H_{\min}/3)^{1/3}$ replacing $(2H/3)^{1/3}$ in the
finite-length and iterate-error bounds. Decrease $\mu$ if necessary
to ensure $0<q<1$. Thus the iterates eventually remain in the PL
neighborhood and converge R-linearly to $\bar x$, the gaps converge
Q-linearly, and the trajectory has finite length.
Restarts are permitted. Neither $H_{\min}$ nor the PL constant is
supplied to the algorithm. Finite termination is the immediate
alternative.
\end{corollary}
\begin{proof}
The gradient-decrease estimate in Theorem~\ref{thm:ag-rate} and
\eqref{eq:adaptive-two-decreases} give, for all sufficiently large $k$,
\[
 \begin{aligned}
 f(x^k)-f(x^{k+1})
 &\ge\frac{\|\nabla f(x^k)\|^2}
                {3(L_0+\sqrt{2H_{\max}L_0D_0})},\\
 f(x^k)-f(x^{k+1})
 &\ge\frac{2H_{\min}}3\|x^{k+1}-x^k\|^3.
 \end{aligned}
\]
Apply the proof of Theorem~\ref{thm:local-pl} to this tail, with
$C_{\mathrm{grad}}=3(L_0+\sqrt{2H_{\max}L_0D_0})$ and
$c_{\mathrm{step}}=2H_{\min}/3>0$. Its neighborhood-invariance,
finite-length, and iterate-convergence arguments use only these
two bounds. The finitely many earlier steps add a finite amount
to the trajectory length. Gradient inclusion and the acceptance
tests are preserved across restarts.
\end{proof}

\subsection{Trial and oracle counts}

This subsection counts the additional evaluations in
Algorithm~\ref{alg:adaptive-cubic}. Recall that $K$ counts accepted
steps and $H_k$ is the coefficient accepted at step $k$.
Let $b_k$ be the number of coefficient doublings before acceptance
at that step. Thus there are $b_k+1$ trials. These counters are used
only here, not in the algorithm or its convergence bounds.

The number of rejected trials at iteration $k$ satisfies
\begin{equation}\label{eq:global-rejections}
 b_k\le\max\left\{0,
 \left\lceil\log_2\frac{2H}{H_{k-1}}\right\rceil
 \right\}.
\end{equation}

Indeed, every trial with coefficient at least $H$ passes by
Theorem~\ref{thm:global-bt}, so the first such coefficient bounds
$b_k$. The accepted coefficient is
$H_k=2^{b_k-1}H_{k-1}$, regardless of whether
acceptance occurs earlier.

\begin{proposition}[Amortized trial and oracle counts]
\label{prop:counts}
Suppose $K\ge1$ iterations have been accepted and their coefficients
satisfy the bound $H_{\max}$ in \eqref{eq:global-Mbound}.
Then the total number of trials satisfies
\begin{equation}\label{eq:trial-count}
 \sum_{k=0}^{K-1}(b_k+1)
 \le 2K+O(1).
\end{equation}
The additive constant is independent of $K$ and accounts for the
initial coefficient.
With cached values at accepted points, there is at most one function
and one gradient evaluation per trial, in addition to one initial
evaluation of each. The stated method uses at most $3K$ HVPs.
These bounds also hold across epochs that retain the accepted coefficient.
\end{proposition}
\begin{proof}
For every $k\ge0$,
\[
 H_k=2^{b_k-1}H_{k-1}.
\]
Taking logarithms and summing telescopes the coefficient updates.
The bound $H_{K-1}\le2H$ gives $\sum_{k=0}^{K-1}b_k\le K+O(1)$,
with a constant independent of $K$, which proves \eqref{eq:trial-count}.
The stated evaluation order uses
a trial gradient only if the function certificate passes. The accepted
gradient is retained for the next step and the aggregate update.
All reduced solves and products $\nabla^2f(x^k)s$ use the same
cached model columns, giving at most three HVPs per accepted step.
Retaining the coefficient at epoch boundaries preserves the telescoping
inequality and changes none of these counts.
\end{proof}

\section{Finite reduced-model solves: proofs and a stopping procedure}
\label{app:finite-solves}

This appendix first shows why the relative residual preserves the two
one-step bounds. It then transfers those bounds to the restart
controller and gives a finite scalar procedure for meeting the test.
All computations here use the reduced coefficients already formed by
the algorithm.

\subsection{Why the model residual preserves convergence}
\begin{proof}[Proof of Proposition~\ref{prop:inexact-model}]
Fix one trial and write $x=x^k$, $s=s^k$, and $\widehat H=H_k$.
The trial model is
\[
 m_x(v)=\nabla f(x)^Tv+\tfrac12v^T\nabla^2f(x)v
                       +\tfrac{\widehat H}{3}\|v\|^3.
\]
For this proof, let $r$ denote the projected model gradient on the
left side of \eqref{eq:inexact-model}, before taking its norm, and put
$b=s^T\nabla^2 f(x)s\ge0$. Taking its inner product with $s$ gives
\[
 \nabla f(x)^Ts
 =-b-\widehat H\|s\|^3+r^Ts
 \le-b-\tfrac56\widehat H\|s\|^3.
\]
Since the current gradient belongs to the space and is nonzero,
$s=0$ cannot satisfy the criterion. Along the segment, the trial
cubic model therefore obeys, for $0<t\le1$,
\[
 m_x(ts)\le-t(1-t/2)b
       -\widehat Ht(5/6-t^2/3)\|s\|^3<0.
\]
When $\widehat H$ is at least the true localized half-Lipschitz
constant, the first-exit argument of Lemma~\ref{lem:invariance}
applies to this strict inequality. The segment stays in the initial
sublevel set, and both Taylor certificates hold. At an accepted trial,
the function certificate gives
\begin{equation}\label{eq:inexact-decrease}
 f(x)-f(x+s)\ge-m_x(s)
 \ge\tfrac12b+\tfrac12\widehat H\|s\|^3.
\end{equation}
Thus the same backtracking thresholds and parameter bounds apply,
and accepted trials remain monotone even below the true constant.

The gradient Taylor certificate and \eqref{eq:inexact-model} imply
\begin{align*}
 \|\Pi_{\mathcal S^k}\nabla f(x+s)\|
 &\le\tfrac{13}{6}\widehat H\|s\|^2\\
 &\le\tfrac{13}{6}\,2^{2/3}\widehat H^{1/3}
              (f(x)-f(x+s))^{2/3}.
\end{align*}
The fixed-coefficient version uses the corresponding Taylor bound.
For the gradient decrease bound, positive semidefiniteness and
$\nabla^2 f(x)\preceq L_0I$ give
\[
 \|\nabla f(x)\|^2
 \le3L_0 b+\tfrac{37}{12}\widehat H^2\|s\|^4.
\]
Indeed, the projected stationarity equation expresses the gradient
as the sum of $-\Pi_{\mathcal S^k}\nabla^2 f(x)s$,
$-\widehat H\|s\|s$, and $r$; the square of its norm is at most
three times the sum of their squared norms.
The same inner-product identity also gives
\[
 \widehat H\|s\|
 \le\sqrt{\tfrac65\widehat H\|\nabla f(x)\|}
 \le\sqrt{\tfrac65\widehat H L_0D_0}.
\]
Combining these estimates with \eqref{eq:inexact-decrease} yields
\[
 \|\nabla f(x)\|^2
 \le7\left(L_0+\sqrt{\tfrac65\widehat H L_0D_0}\right)
              (f(x)-f(x+s)).
\]
The proofs of Lemmas~\ref{lem:reserve}--\ref{lem:window} and
Theorem~\ref{thm:main} use only these gradient and correction-error bounds.
Substituting the displayed constants gives the same global order;
in the adaptive case use the upper bound on $H_k$.
The same two bounds also give \eqref{eq:epoch-bound} with changed
constants. Increasing the constant in the epoch induction accordingly
preserves the inverse-square rate.

Locally, $\nabla^2f(x)\succeq\mu I$ for some $\mu>0$, and the inner-product identity
gives $\|s\|\le\|\nabla f(x)\|/\mu=O(\|x^k-x^\star\|_\star)$.
For bounded $H_k$, both the model regularization gradient and $r$
are $O(\|x^k-x^\star\|_\star^2)$. The local projected Newton decomposition thus retains
an $O(\|x^k-x^\star\|_\star^2)$ remainder in the search space. The proofs of Lemma~\ref{lem:local-projection} and
Theorem~\ref{thm:local} therefore apply unchanged. The displayed
decrease bounds also give the stated PL factor. The step-length
coefficient is $H/2>0$ for fixed coefficients, or $H_{\min}/2>0$
for the adaptive tail under the lower-bound condition in
Proposition~\ref{prop:inexact-model}. Apply the proof of
Theorem~\ref{thm:local-pl} to obtain neighborhood invariance,
finite trajectory length, and iterate convergence with the stated
constants, also across restarts.

Finally, for a nonzero current gradient the reduced cubic model is
coercive and strictly convex, with a unique nonzero minimizer.
Along any inner sequence converging to that minimizer, the residual
tends to zero, whereas the right side of \eqref{eq:inexact-model}
tends to a positive number. The criterion is therefore met after
finitely many inner iterations. Its evaluation needs only reduced
matrix-vector arithmetic after the original at most three HVPs
(four with the optional memory direction).
\end{proof}

\subsection{Why the accelerated controller is preserved}
\begin{proof}[Proof of Corollary~\ref{cor:inexact-acceleration}]
For the global induction, abbreviate
\[
 C_{\mathrm{inexact}}=81D_0^2
 \left[\tfrac72\left(L_0+\sqrt{\tfrac65H_{\max}L_0D_0}\right)
                               +1024H_{\max}D_0\right].
\]
On a current sublevel with Hessian bound $L$ and radius $D$ about
$x^\star$, Proposition~\ref{prop:inexact-model} gives
\[
 \|\nabla f(x^k)\|^2\le C_{\mathrm{grad}}d_k,\qquad
 \|e^{k+1}\|\le C_{\mathrm{corr}}d_k^{2/3},
 \qquad d_k=f(x^k)-f(x^{k+1}),
\]
where
\[
 C_{\mathrm{grad}}=7\left(L+\sqrt{\frac65H_{\max}LD}\right),
 \qquad C_{\mathrm{corr}}^3=\frac{2197}{54}H_{\max}.
\]
These estimates hold through every reset. The weighted decrease
estimate and weighted Young inequality used in
\eqref{eq:epoch-bound} therefore give, for a completed epoch
of $4t+1$ steps,
\[
 9t^2\Delta_{\mathrm{end}}
 \le (t^2-t)\Delta_{\mathrm{start}}
       +\frac{C_{\mathrm{grad}}}2D^2
       +\frac{512}{27}C_{\mathrm{corr}}^3D^3
 \le (t^2-t)\Delta_{\mathrm{start}}
       +\frac{C_{\mathrm{grad}}}2D^2+1024H_{\max}D^3.
\]
Here $\Delta_{\mathrm{start}},\Delta_{\mathrm{end}}$ are the objective gaps
at the epoch endpoints, and
$(512/27)(2197/54)=562432/729<1024$.
The deterministic-epoch induction now gives the base bound
$C_{\mathrm{inexact}}/(K+1)^2$. The global-stage argument in
Theorem~\ref{thm:automatic-acceleration} gives the stated factor
$64$ for the controller.

On a sublevel where $\nabla^2 f\succeq\mu I$, the same epoch
estimate yields
\[
 \Delta_{\mathrm{end}}\le
 \frac{t^2-t+\Gamma}{9t^2}\Delta_{\mathrm{start}},\qquad
 \Gamma=\frac{C_{\mathrm{grad}}}{\mu}+\frac{2048H_{\max}D}{\mu}.
\]
The first step also ensures an epoch decrease of at least
$\theta\Delta_{\mathrm{start}}$, where $\theta=2\mu/C_{\mathrm{grad}}$.
Thus the controller proof uses the same two estimates,
with these new analysis constants. Their proofs need only gradient
inclusion and correction-space containment, so they apply to the
three-direction variant as well.

The explicit transition constants can also be retained. Write
$\mu_\star=\lambda_{\min}(\nabla^2f(x^\star))$ and
$L_\star=\lambda_{\max}(\nabla^2f(x^\star))$.
In the transition region constructed in
Corollary~\ref{cor:explicit-transition}, the radius bound
$\rho=\mu_\star/(2048H_{\max})$ and the Hessian Lipschitz bound give
\[
 \mu=\frac{1023}{1024}\mu_\star,\qquad
 L\le\frac{1025}{1024}L_\star,\qquad D\le\rho.
\]
Consequently,
\[
 \sqrt{\frac65H_{\max}LD}\le\frac{L_\star}{32},\qquad
 C_{\mathrm{grad}}<8L_\star,
\]
and hence
\[
 \Gamma\le\frac{1024}{1023}(8\kappa_\star+1)
           <10\kappa_\star<19\kappa_\star,\qquad
 \theta\ge\frac{1023}{4096\kappa_\star}
      >\frac1{9\kappa_\star}.
\]
The geometric localization using
$\varepsilon_{\mathrm{loc}}=\mu_\star^3/(2^{25}H_{\max}^2)$ depends only on
monotonicity and the Hessian Lipschitz bound, so it is unaffected
by the model residual. The bounds on the tuner length, completed-stage
progress, and phase accumulation therefore apply with the same
root exponent and finite-tail constants.
Only the global constant used to bound entry changes.

Finally, all residual checks use the cached reduced model.
The finite-termination and oracle arguments of
Proposition~\ref{prop:inexact-model} apply to every trial.
A stationary initial point, or subsequent finite termination,
is covered by constant extension.
\end{proof}

\subsection{How to stop the scalar solve after finitely many steps}
We now specify a bracket and a stopping count for the scalar equation
from Section~\ref{sec:reduced-model}. For one trial, let $Q^k$ have
orthonormal columns spanning $\mathcal S^k$, put
$B^k=(Q^k)^T\nabla^2f(x^k)Q^k$, and let $\widehat H>0$ be its
cubic coefficient. The current gradient belongs to $\mathcal S^k$.
The following construction uses only these cached quantities.

\begin{proposition}[A constructive bound for the reduced cubic solve]
\label{prop:cubic-inner-work}
For a nonstationary trial, retain the notation $Q^k,B^k,\widehat H$
of this appendix. Choose a computable $L\ge\|B^k\|$, for example
$L=\operatorname{tr}B^k$, and define
\[
 a(\lambda)=-(B^k+\lambda I)^{-1}(Q^k)^T\nabla f(x^k),\qquad
 \psi(\lambda)=\frac{\lambda}{\widehat H\|a(\lambda)\|},
 \qquad \lambda>0.
\]
The unique root of $\psi(\lambda)=1$ lies in $[\ell,U]$, where
\[
 U=\sqrt{\widehat H\|\nabla f(x^k)\|},\qquad
 \ell=\frac{U^2}{L+U}.
\]
Let $J_{\mathrm{inner}}$ denote the number of inner bisections,
which is separate from the accepted-step count $K$. If $L>0$, set
\begin{equation}\label{eq:cubic-inner-count}
 J_{\mathrm{inner}}=\max\left\{0,\left\lceil
 \log_2\!\left(\frac{2\log(1+L/U)}{\log(7/6)}\right)
 \right\rceil\right\}.
\end{equation}
Starting with $[\ell,U]$, perform $J_{\mathrm{inner}}$ geometric bisections: evaluate
$\psi$ at the geometric mean of the endpoints, replace the lower
endpoint when this value is at most one, and otherwise replace the
upper endpoint. Let $v$ be the final upper endpoint. Then
$s=Q^k a(v)$ satisfies \eqref{eq:inexact-model} with $H_k=\widehat H$.
The procedure uses at most $J_{\mathrm{inner}}+1$ reduced linear solves, including a
final evaluation at $v$, and no additional objective-oracle calls.
If $L=0$, the bracket collapses to $\ell=U$ and one reduced solve
gives the exact model minimizer.
\end{proposition}
\begin{proof}
Since the current gradient belongs to the search space,
$\|(Q^k)^T\nabla f(x^k)\|=\|\nabla f(x^k)\|>0$.
The spectral bounds $0\preceq B^k\preceq LI$ imply
\[
 \frac{\|\nabla f(x^k)\|}{L+\lambda}
 \le\|a(\lambda)\|
 \le\frac{\|\nabla f(x^k)\|}{\lambda}.
\]
The function $\psi$ is strictly increasing. The upper norm bound
gives $\psi(U)\ge1$, while the lower bound and $\ell\le U$ give
\[
 \psi(\ell)\le\frac{\ell(L+\ell)}{U^2}
              \le\frac{\ell(L+U)}{U^2}=1.
\]
For any $0<u\le v$, diagonalizing the positive semidefinite matrix
$B^k$ gives
\[
 1\le\frac{\|a(u)\|}{\|a(v)\|}\le\frac vu,
 \qquad
 \frac{\psi(v)}{\psi(u)}\le\left(\frac vu\right)^2.
\]
Every bisection preserves the root bracket and halves its logarithmic
width. If $[u,v]$ is the final interval, then
\[
 1\le\psi(v)\le(v/u)^2
 =\exp\!\left(2^{1-J_{\mathrm{inner}}}\log(U/\ell)\right)\le\frac76,
\]
where $\log(U/\ell)=\log(1+L/U)$ and
\eqref{eq:cubic-inner-count} gives the last inequality.
For $s=Q^k a(v)$, the projected model gradient is exactly
$(\widehat H\|s\|-v)s$. Therefore its norm is
\[
 (\psi(v)-1)\widehat H\|s\|^2
 \le\frac{\widehat H}{6}\|s\|^2,
\]
as required. Each midpoint and the final endpoint require at most
one reduced solve; the initial bracket needs no numerical root test.
If $L=0$, then $B^k=0$ and $U=\ell$ solves the secular equation exactly.
\end{proof}

All systems in this procedure have order at most three (four with the
optional direction of Corollary~\ref{cor:three-hvp-restarts}) and are positive
definite because their shifts are positive. No strong-convexity parameter
is required, even when $B^k$ is singular. The count grows doubly
logarithmically with $L/U$ and concerns exact real arithmetic; it is not
a floating-point or bit-complexity guarantee. The scalar reduction is
the usual regularized-model equation. The role of the explicit bracket
and count here is to attain the relative accuracy used in our convergence
proofs.

\end{document}